\documentclass[11pt]{amsart}
\usepackage{amsmath}
\usepackage{a4wide}
\usepackage[utf8]{inputenc}
\usepackage{amssymb}
\usepackage{amsopn}
\usepackage{epsfig}
\usepackage{amsfonts}
\usepackage{latexsym}
\usepackage{amsthm}
\usepackage{enumerate}
\usepackage[UKenglish]{babel}
\usepackage{verbatim}
\usepackage{color}
\usepackage{pgf}
\usepackage{tikz}

\usepackage{mathrsfs}   

\usetikzlibrary{arrows,automata}

\newtheorem{theorem}{Theorem}[section]
\newtheorem{lemma}[theorem]{Lemma}
\newtheorem{proposition}[theorem]{Proposition}
\newtheorem{corollary}[theorem]{Corollary}

\theoremstyle{definition}

\newtheorem{example}[theorem]{Example}

\newtheorem{question}[theorem]{Question}

\theoremstyle{remark}
\newtheorem{remark}[theorem]{Remark}

\numberwithin{equation}{section}

\newcommand{\R}{\ensuremath{\mathbb{R}}}
\newcommand{\N}{\ensuremath{\mathbb{N}}}

\renewcommand{\u}{\ensuremath{\mathcal{U}}}
\newcommand{\ub}{\mathscr{U}}
\newcommand{\us}{\mathbf{U}}

 \newcommand{\vb}{\mathscr{V}}

\newcommand{\sv}{\mathbf{a}}
\newcommand{\sa}{\mathbf{a}}

\newcommand{\sd}{\mathbf{d}}
\renewcommand{\sc}{\mathbf{c}}

\newcommand{\set}[1]{\left\{#1\right\}}
\newcommand{\la}{\lambda}

\newcommand{\ep}{\varepsilon}
\newcommand{\f}{\infty}
\newcommand{\de}{\delta}

\newcommand{\al}{\alpha}
\newcommand{\lle}{\preccurlyeq}
\newcommand{\lge}{\succcurlyeq}
\newcommand{\si}{\sigma}

\begin{document}

\title{On the smallest base in which a number has a unique expansion}

\author{Pieter Allaart}
\address[P. Allaart]{Mathematics Department, University of North Texas, 1155 Union Cir \#311430, Denton, TX 76203-5017, U.S.A.}
\email{allaart@unt.edu}

\author{Derong Kong}

\address[D. Kong]{College of Mathematics and Statistics, Chongqing University, 401331, Chongqing, P.R.China}
\email{derongkong@126.com}

\dedicatory{}

\begin{abstract}
Given a real number $x>0$, we determine $q_s(x):=\inf\ub(x)$, where $\ub(x)$ is the set of all bases $q\in(1,2]$ for which $x$ has a unique expansion of $0$'s and $1$'s. We give an explicit description of $q_s(x)$ for several regions of $x$-values. For others, we present an efficient algorithm to determine $q_s(x)$ and the lexicographically smallest unique expansion of $x$. We show that the infimum is attained for almost all $x$, but there is also a set of points of positive Hausdorff dimension for which the infimum is proper. In addition, we show that the function $q_s$ is right-continuous with left-hand limits and no downward jumps, and characterize the points of discontinuity of $q_s$.

A large part of the paper is devoted to the level sets $L(q):=\{x>0:q_s(x)=q\}$. We show that $L(q)$ is finite for almost every $q$, but there are also infinitely many infinite level sets. In particular, for the Komornik-Loreti constant $q_{KL}=\min\ub(1)\approx 1.787$ we prove that $L(q_{KL})$ has both infinitely many left- and infinitely many right accumulation points.
\end{abstract}


\maketitle
\tableofcontents

\section{Introduction}

Non-integer base expansions have received much attention since the pioneering works of R\'enyi \cite{Renyi_1957} and Parry \cite{Parry_1960}. Given a base $q\in(1,2]$, each real number $x\in[0, 1/(q-1)]$ can be written as
\begin{equation}\label{eq:expansion-base-q}
x=\sum_{i=1}^\f\frac{d_i}{q^i}=\frac{d_1}{q}+\frac{d_2}{q^2}+\cdots=:(d_i)_q,\quad d_i\in\set{0, 1}~\forall i\ge 1.
\end{equation}
The infinite sequence $(d_i)\in\set{0, 1}^\N$ is called a $q$-expansion of $x$. For $q=2$, each $x\in[0, 1]$ has a unique $q$-expansion, except for countably many points having precisely two expansions. But if $q\in(1,2)$, then almost every $x\in[0, 1/(q-1)]$ has a continuum of $q$-expansions (cf.~\cite{Dajani_DeVries_2007,Sidorov_2003}). On the other hand, there also exist points having a unique $q$-expansion. The study of unique expansions goes back to Erd\H os and others (cf.~\cite{Erdos_Joo_Komornik_1990, Erdos_Horvath_Joo_1991}), who showed among other things that nontrivial points with a unique $q$-expansion exist if and only if $q>q_G:=(1+\sqrt{5})/2$. Since then, several papers (e.g. \cite{Darczy_Katai_1995,DeVries_Komornik_2008,Glendinning_Sidorov_2001,Komornik-Kong-Li-17,Kong_Li_2015}) have studied properties of the univoque set
\[
\u_q:=\{x\in[0, 1/(q-1)]: x\textrm{ has a unique $q$-expansion of the form }(\ref{eq:expansion-base-q})\}.
\]

Equally interesting is the reverse question: Given a number $x>0$, what can we say about the set
\[
\ub(x):=\set{q\in(1,2]: x\textrm{ has a unique $q$-expansion of the form }(\ref{eq:expansion-base-q})}?
\]
The most important special case is the set $\ub:=\ub(1)$, which has been well studied. It is Lebesgue null but has full Hausdorff dimension (cf.~\cite{Darczy_Katai_1995,Erdos_Joo_Komornik_1990}). Its closure $\overline{\ub}$ is a Cantor set (cf.~\cite{Komornik_Loreti_2007}). Furthermore, its smallest base $\min \ub\approx 1.78723$ was determined by Komornik and Loreti \cite{Komornik-Loreti-1998}, and was shown to be transcendental by Allouche and Cosnard \cite{Allouche_Cosnard_2000}. This number is now known as the \emph{Komornik-Loreti constant} and will be denoted by $q_{KL}$;  its importance for the study of unique $q$-expansions was shown by Glendinning and Sidorov \cite{Glendinning_Sidorov_2001}. Recently, the authors investigated in \cite{Allaart-Kong-2018} the local structure of $\ub$. For connections of $\ub$ to other areas of dynamical systems we refer to the paper \cite{Bon-Car-Ste-Giu-2013}.

However, for other $x>0$ very little is known about $\ub(x)$. L\"u et al.~\cite{Lu_Tan_Wu_2014} showed that for $x\in(0,1)$ the set $\ub(x)$ is also a Lebesgue null set of full Haudorff dimension. Moreover, Dajani et al.~\cite{Dajani-Komornik-Kong-Li-2018} showed that the algebraic difference $\ub(x)-\ub(x)$ contains an interval.  More recently, Kong et al.~{\cite{Kong-Li-Lv-Wang-Xu-2019}} studied the set $\ub(x)$ in more detail and showed how its size depends on $x$. They moreover connected the local Hausdorff dimension of $\ub(x)$ to the local dimension of $\u_q$. 

Motivated by the work of the second author \cite{Kong_2016}, we continue to investigate the {\em smallest} univoque base
\begin{equation} \label{eq:definition-of-qs}
q_s(x):=\inf\ub(x), \qquad x>0,
\end{equation}
thereby extending the result of Komornik and Loreti \cite{Komornik-Loreti-1998}.
Following the preliminary Section \ref{sec:prelim} we first develop, in Section \ref{sec:algorithm}, an efficient algorithm to determine $q_s(x)$ for any $x\in(0,1)$. This algorithm has both practical and theoretical importance: On the one hand, for many points $x$ the algorithm gives us the value of $q_s(x)$, and the corresponding unique expansion, in a small finite number of steps. {For example, we find in Example \ref{ex:algorithm-example} that $q_s(2/3)$ is the unique base $q$ for which $\big(10010(100100101100101100101)^\infty\big)_q=2/3$. This gives an algebraic equation in $q$; numerically, $q_s(2/3)\approx 1.8316$. Similarly, we find that $q_s(1/2)=\sqrt{3}$.}

On the other hand, we can infer from the algorithm that for Lebesgue almost every $x$, $q_s(x)=\min\ub(x)$; in other words, the infimum in \eqref{eq:definition-of-qs} is attained almost everywhere. Furthermore, the algorithm allows us to prove, in Section \ref{sec:continuity}, the following:

\begin{theorem} \label{thm:cadlag-intro}
The function $q_s$ is c\'adl\'ag (right-continuous with left-hand limits at every point) and has no downward jumps. As a result, $q_s$ has only countably many discontinuities.
\end{theorem}

As a corollary, we show that $q_s$ has a maximum value $q_{\max}$ and takes on every value in the interval $(1,q_{\max}]$ at least once. In fact, we can show that $q_{\max}\approx 1.88845$ is an algebraic integer and this value is uniquely attained at $x=1/q_G=(\sqrt{5}-1)/2$ (see Figure \ref{fig:1}).

\begin{figure}[h!]
 \begin{center}
\begin{tikzpicture}[xscale=6.1,yscale=15]
\draw [->] (-0.01,1.5) node[anchor=east] {$1.5$}  -- (2,1.5) node[anchor=west] {$x$};
\draw [->] (0,1.49) node[anchor=north] {$0$} -- (0,1.95) node[anchor=south] {$q$};

\draw [dashed] (1.9,1.618)--(0,1.618);
\draw (-0.01,1.618) node[anchor=east] {$\hat{q}_1\approx 1.618$} -- (0.01,1.618);
\draw [dashed] (1.9,1.7549)--(0,1.7549);
\draw (-0.01,1.7549) node[anchor=east] {$\hat{q}_2\approx 1.755$} -- (0.01,1.7549);
\draw [dashed] (1.9,1.7872)--(0,1.7872);
\draw (-0.01,1.7872) node[anchor=east] {$q_{KL}\approx 1.787$} -- (0.01,1.7872);
\draw [dashed] (1,1.94)--(1,1.49) node[anchor=north] {$1$};
\draw (-0.01,1.8885) node[anchor=east] {$q_{\max}\approx 1.888$} -- (0.01,1.8885);
\draw (0.618,1.495) node[anchor=north] {$q_G^{-1}$} -- (0.618,1.505);
\draw (0.382,1.495) node[anchor=north] {$q_G^{-2}$} -- (0.382,1.505);
\draw (0.236,1.495) node[anchor=north] {$q_G^{-3}$} -- (0.236,1.505);
\draw (0.12,1.47) node {$\cdots$};
\draw (1.625,1.495) node[anchor=north] {$\xi_1=q_G$} -- (1.625,1.505);
\draw (1.236,1.495) node[anchor=north] {$\xi_{1,1}$} -- (1.236,1.505);
\draw (1.382,1.495) node[anchor=north] {$\xi_{1,2}$} -- (1.382,1.505);
\draw (1.48,1.475) node {$\dots$};
\draw (1.055,1.495) node[anchor=north] {$\xi_2$} -- (1.055,1.505);


\draw[domain=1.618:1.7549,variable=\q] plot({\q/(pow(\q,2)-1)},{\q});  
\draw[domain=1.618:1.7549,variable=\q] plot({1/(pow(\q,2)-1)},{\q});	
\draw[domain=1.618:1.7549,variable=\q] plot({1/(\q*(pow(\q,2)-1))},{\q});	
\draw[domain=1.618:1.7549,variable=\q] plot({1/(\q*\q*(pow(\q,2)-1))},{\q});	
\draw[domain=1.618:1.7510,variable=\q] plot({1/(\q*\q*\q*(pow(\q,2)-1))},{\q});	
\draw[domain=1.618:1.7313,variable=\q] plot({1/(pow(\q,4)*(pow(\q,2)-1))},{\q});	
\draw[domain=1.618:1.7167,variable=\q] plot({1/(pow(\q,5)*(pow(\q,2)-1))},{\q});	
\draw[domain=1.618:1.7055,variable=\q] plot({1/(pow(\q,5)*(pow(\q,2)-1))},{\q});	
\draw[domain=1.618:1.6966,variable=\q] plot({1/(pow(\q,6)*(pow(\q,2)-1))},{\q});	
\draw[domain=1.618:1.6894,variable=\q] plot({1/(pow(\q,7)*(pow(\q,2)-1))},{\q});	
\draw[domain=1.618:1.6833,variable=\q] plot({1/(pow(\q,8)*(pow(\q,2)-1))},{\q});	

\draw[domain=1.7549:1.7846,variable=\q] plot({1/\q+(\q+1)/(pow(\q,5)-\q)},{\q});	
\draw[domain=1.7549:1.7846,variable=\q] plot({1/pow(\q,2)+(\q+1)/(pow(\q,6)-\q*\q)},{\q});	
\draw[domain=1.7549:1.7718,variable=\q] plot({1/pow(\q,3)+(\q+1)/(pow(\q,7)-pow(\q,3))},{\q});	
\draw[domain=1.7549:1.7846,variable=\q] plot({1/\q+1/pow(\q,3)+(\q+1)/(pow(\q,7)-pow(\q,3))},{\q});	
\draw[domain=1.7549:1.7846,variable=\q] plot({1/pow(\q,2)+1/pow(\q,4)+(\q+1)/(pow(\q,8)-pow(\q,4))},{\q}); 
\draw[domain=1.7549:1.7846,variable=\q] plot({1/pow(\q,3)+1/pow(\q,5)+(\q+1)/(pow(\q,9)-pow(\q,5))},{\q});	
\draw[domain=1.7549:1.7668,variable=\q] plot({1/pow(\q,4)+1/pow(\q,6)+(\q+1)/(pow(\q,10)-pow(\q,6))},{\q});	
\draw[domain=1.7549:1.7769,variable=\q] plot({1/\q+1/pow(\q,3)+1/pow(\q,5)+(\q+1)/(pow(\q,9)-pow(\q,5))},{\q});  
\draw[domain=1.7549:1.7692,variable=\q] plot({1/pow(\q,2)+1/pow(\q,4)+1/pow(\q,6)+(\q+1)/(pow(\q,10)-pow(\q,6))},{\q});  
\draw[domain=1.7549:1.7655,variable=\q] plot({1/pow(\q,3)+1/pow(\q,5)+1/pow(\q,7)+(\q+1)/(pow(\q,11)-pow(\q,7))},{\q});  
\draw[domain=1.7549:1.7618,variable=\q] plot({1/\q+1/pow(\q,3)+1/pow(\q,5)+1/pow(\q,7)+(\q+1)/(pow(\q,11)-pow(\q,7))},{\q});  
\draw[domain=1.7549:1.7594,variable=\q] plot({pow(\q,-2)+pow(\q,-4)+pow(\q,-6)+pow(\q,-8)+(\q+1)/(pow(\q,12)-pow(\q,8))},{\q});  



\draw[domain=1.8393:1.8668,variable=\q] plot({pow(\q,2)/(pow(\q,3)-1)},{\q}); 	
\draw[domain=1.8124:1.8233,variable=\q] plot({(pow(\q,4)+\q)/(pow(\q,5)-1)},{\q});	
\draw[domain=1.8124:1.8174,variable=\q] plot({1/\q+(pow(\q,3)+1)/(\q*(pow(\q,5)-1))},{\q});	
\draw[domain=1.8668:1.8706,variable=\q] plot({(pow(\q,5)+\q+1)/(pow(\q^6-1))},{\q});	

\draw (0.618,1.8885)--(0.6204,1.8843)--(0.6228,1.8802)--(0.6252,1.8760)--(0.6276,1.8719)--(0.63,1.8681);

\draw (0.6518,1.8362)--(0.6548,1.8332)--(0.6578,1.8305);
\draw (0.6588,1.8316)--(0.6628,1.8270)--(0.6658,1.8239);
\draw (0.6668,1.8314)--(0.6698,1.8267);

\draw[fill] (0.6832,1.8092) circle[radius=0.001];
\draw (0.6842,1.8109)--(0.6882,1.8050)--(0.6932,1.8012);
\draw (0.6942,1.8020)--(0.6982,1.7973);
\draw (0.6992,1.8020)--(0.7032,1.7965)--(0.7072,1.7911)--(0.7108,1.7872)--(0.7138,1.7851);
\draw (0.7148,1.8028)--(0.7168,1.8002)--(0.7188,1.7977)--(0.7208,1.7956)--(0.7228,1.7930)--(0.7248,1.7903)--(0.7268,1.7877);

\draw (0.7615,1.8091)--(0.7645,1.8060)--(0.7675,1.8026)--(0.7705,1.7997)--(0.7735,1.7978)--(0.7765,1.7940)
	--(0.7795,1.7902)--(0.7825,1.7872)--(0.7845,1.7847);
\draw (0.7855,1.7893)--(0.7872,1.7872);

\draw (0.382,1.8042)--(0.385,1.8010)--(0.389,1.7966);
\draw (0.390,1.7988)--(0.393,1.7941)--(0.396,1.7892)--(0.400,1.7849);
\draw (0.401,1.7963)--(0.404,1.7916)--(0.408,1.7857);

\draw (0.4302,1.7967)--(0.4342,1.7920)--(0.4372,1.7877)--(0.4392,1.7852)--(0.4412,1.7862);

\draw[fill] (0.7275,1.7859) circle[radius=0.001];
\draw[fill] (0.7900,1.7855) circle[radius=0.001];
\draw[fill] (0.2464,1.7873) circle[radius=0.001];
\draw[fill] (0.2474,1.7855) circle[radius=0.001];


\draw[domain=1.618:1.7549,variable=\q] plot({1/\q+1/(pow(\q,2)-1)},{\q}); 
\draw[domain=1.618:1.7063,variable=\q] plot({1/\q+1/pow(\q,2)+1/(pow(\q,3)-\q)},{\q}); 
\draw[domain=1.618:1.6648,variable=\q] plot({1/\q+1/pow(\q,2)+1/pow(\q,3)+1/(pow(\q,4)-pow(\q,2))},{\q}); 
\draw[domain=1.618:1.6441,variable=\q] plot({1/\q+1/pow(\q,2)+1/pow(\q,3)+1/pow(\q,4)+1/(pow(\q,5)-pow(\q,3))},{\q}); 
\draw[domain=1.618:1.6331,variable=\q] plot({1/\q+1/pow(\q,2)+1/pow(\q,3)+1/pow(\q,4)+1/pow(\q,5)+1/(pow(\q,6)-pow(\q,4))},{\q}); 
\draw[domain=1.618:1.6269,variable=\q] plot({1/\q+1/pow(\q,2)+1/pow(\q,3)+1/pow(\q,4)+1/pow(\q,5)+1/pow(\q,6)+1/(pow(\q,7)-pow(\q,5))},{\q}); 
\draw[domain=1.618:1.6234,variable=\q] plot({1/\q+1/pow(\q,2)+1/pow(\q,3)+1/pow(\q,4)+1/pow(\q,5)+1/pow(\q,6)+pow(\q,-7)+1/(pow(\q,8)-pow(\q,6))},{\q}); 
\draw[domain=1.618:1.6213,variable=\q] plot({1/\q+1/pow(\q,2)+1/pow(\q,3)+1/pow(\q,4)+1/pow(\q,5)+1/pow(\q,6)+pow(\q,-7)+pow(\q,-8)+1/(pow(\q,9)-pow(\q,7))},{\q}); 

\draw[domain=1.7549:1.7846,variable=\q] plot({1/\q+pow(\q,-2)+pow(\q,-4)+(\q+1)/(pow(\q,8)-pow(\q,4))},{\q});  
\draw[domain=1.7549:1.7650,variable=\q] plot({1/\q+pow(\q,-2)+pow(\q,-4)+pow(\q,-6)+(\q+1)/(pow(\q,10)-pow(\q,6))},{\q});  
\draw[domain=1.7549:1.7581,variable=\q] plot({1/\q+pow(\q,-2)+pow(\q,-4)+pow(\q,-6)+pow(\q,-8)+(\q+1)/(pow(\q,12)-pow(\q,8))},{\q});  

\draw[fill] (1.0015,1.7859) circle[radius=0.001];
\draw[fill] (1,1.7872) circle[radius=0.002];

\draw[domain=1.618:1.9,variable=\x] plot(\x,{1+(1/\x)});
\useasboundingbox(-0.45,2);
\end{tikzpicture}
\end{center}
\caption{The graph of $q_s(x)=\inf\ub(x)$.}
\label{fig:1}
\end{figure}

A sizeable part of the paper is devoted to the level sets of $q_s$; that is, the sets
\[
L(q):=\{x>0: q_s(x)=q\}, \qquad 1<q\leq q_{\max},
\]
which we study in Section \ref{sec:level-sets}. From Theorem \ref{thm:cadlag-intro} we can deduce that $L(q)$ has both a smallest and a largest element for each $q\in(1,q_{\max}]$, though we do not know whether the level sets are closed. However, we can prove the following.

\begin{theorem} \label{thm:finite-level-sets-intro}
The set $L(q)$ is finite for Lebesgue almost every $q$. Specifically, $L(q)$ is finite whenever $q\not\in\overline{\ub}$.
\end{theorem}

On the other hand, at least some of the level sets of $q_s$ are infinite. Specifically, this holds for the Komornik-Loreti constant $q_{KL}\approx 1.78723$. Recall that Komornik and Loreti \cite{Komornik-Loreti-1998} showed that $q_s(1)=\min\ub(1)=q_{KL}$, so it is interesting to ask for how many other points this is the case. The answer may be surprising.

\begin{theorem} \label{thm:infinite-level-set-qKL-intro}
There are infinitely many points $x$ such that $q_{s}(x)=\min\ub(x)=q_{KL}$. Moreover, the level set $L(q_{KL})$ has infinitely many right- and infinitely many left accumulation points, which are themselves members of $L(q_{KL})$.
\end{theorem}

Theorem \ref{thm:infinite-level-set-qKL-intro} corrects a result in \cite[Theorem 1.1]{Kong_2016} which states that $q_{s}(x)=\min\ub(x)=q_{KL}$ for only four values of $x$. (In fact, for three of the four values reported in \cite{Kong_2016} the value of $q_s(x)$ is strictly below $q_{KL}$.) We will do substantially more and present an infinite set of specific points in $L(q_{KL})$. However, we do not know whether $L(q_{KL})$ is countable.

The level set $L(q_{KL})$ is not alone in being infinite; in fact we also prove the following:

\begin{theorem} \label{thm:other-infinite-level-sets-intro}
The function $q_s$ has infinitely many infinite level sets.
\end{theorem}

In Section \ref{sec:cascades} we examine the graph of $q_s$ in more detail. As a special case of Theorem \ref{thm:Komornik-Loreti-cascades}, we obtain a complete description of the graph of $q_s$ over the interval $[1,\infty)$, which may be summarized as follows (see Figure \ref{fig:1}). There is a strictly increasing sequence $(\hat{q}_n)$ of bases (to be defined in Section \ref{subsec:KL}) such that $\hat q_1=q_G$ and $\hat q_n\nearrow q_{KL}$;
a strictly decreasing sequence $(\xi_m)$ of points in $(1,q_G]$ such that $\xi_1=q_G$ and $\xi_m\searrow 1$;
and for each $m$, a strictly increasing sequence $(\xi_{m,j})$ of points in $(\xi_{m+1},\xi_m)$ with $\xi_{m,j}\nearrow \xi_m$;
such that the following hold:
\begin{itemize}
	\item On $[\xi_{m+1},\xi_m)$, $q_s(x)$ takes values in $(\hat{q}_m,\hat{q}_{m+1}]$.
  \item $q_s$ is strictly decreasing and convex on each interval $[\xi_{m+1}, \xi_{m,1})$ or $[\xi_{m,j}, \xi_{m,j+1})$ with $j\ge 1$.
  \item For each $m$, the sequence $q_s(\xi_{m,j})$ is strictly decreasing in $j$ and converges to $\hat q_m$.
  \item $\lim_{x\nearrow \xi_{m,j}}q_s(x)=\hat q_m<q_s(\xi_{m,j})$ for all $j\ge 1$, and $\lim_{x\nearrow \xi_m}q_s(x)=\hat q_m=q_s(\xi_m)$.
\end{itemize}
We call this repetitive behavior, in which the graph gradually works its way down from $q_{KL}$ at $x=1$ to $q_G$ at $x=q_G$ with repeated upward ``bounces", a {\em Komornik-Loreti cascade}. As shown in Theorem \ref{thm:Komornik-Loreti-cascades}, this type of behavior also occurs at smaller scales in infinitely many other places in the graph.

In Section \ref{sec:maximum}, we compute the maximum value $q_{\max}$ of the function $q_s(x)$. We show that this maximum is uniquely attained at $x=1/q_G$, and $q_{\max}$ is an algebraic integer of degree 26.



Finally, we end the paper with a list of open problems.

\section{Preliminaries} \label{sec:prelim}

In this section we review some well-known properties of unique expansions. First we need some notions from symbolic dynamics (cf.~\cite{Lind_Marcus_1995}). Let $\set{0, 1}^\N$ be the set of infinite sequences of zeros and ones. Denote by $\si$ the left shift on $\set{0,1}^\N$ such that $\si((c_i))=(c_{i+1})$. 
By a word $\sc=c_1\ldots c_n$ we mean a finite string of digits with $c_i\in\set{0,1}$. We denote by $\sc^\f=\sc\sc\ldots\in\set{0,1}^\N$ the periodic sequence with periodic block $\sc$. Throughout the paper we will use the lexicographical ordering `$\prec, \lle, \succ$' and `$\lge$' between sequences and words in the traditional way. For example, for two sequences $(c_i), (d_i)\in\set{0,1}^\N$ we write $(c_i)\prec (d_i)$ if $c_1<d_1$, or there exists $n>1$ such that $c_1\ldots c_{n-1}=d_1\ldots d_{n-1}$ and $c_n<d_n$. Moreover, for two words $\sc, \sd$ we say $\sc\prec \sd$ if $\sc 0^\f\prec \sd 0^\f$. For a sequence $(c_i)\in\set{0, 1}^\N$ we denote by $\overline{(c_i)}:=(1-c_1)(1-c_2)\ldots\in\set{0,1}^\N$ its \emph{reflection}. Similarly, for a word $\sc=c_1\ldots c_n$ we write $\overline{\sc}:=(1-c_1)\ldots(1-c_n)$. If $c_n=0$, then we write $\sc^+:=c_1\ldots c_{n-1}1$; and if $c_n=1$, we write $\sc^-:=c_1\ldots c_{n-1}0$. So $\overline{\sc}, \sc^+$ and $\sc^-$ are all words with digits in $\set{0,1}$.

\subsection{Quasi-greedy expansions and unique expansions}

For $q\in(1,2]$ and $x\in[0, 1/(q-1)]$, let
\[
\Phi_x(q)=x_1(q)x_2(q)\ldots\;\in\set{0,1}^\N
\]
be the quasi-greedy $q$-expansion of $x$; that is, the lexicographically largest $q$-expansion of $x$ not ending with $0^\f$. For the special case $x=1$, we let $\al(q)=\al_1(q)\al_2(q)\ldots=\Phi_1(q)$ denote the quasi-greedy $q$-expansion of $1$. The following characterizations of the map $q\mapsto\Phi_x(q)$ for a fixed $x>0$, and of the map $x\mapsto\Phi_x(q)$ for a fixed $q\in(1,2]$ can be deduced easily from \cite[Lemma 2.3]{DeVries-Komornik-2011}. Observe that for a fixed $x>0$ the largest base $q\in(1,2]$ for which $x$ has a $q$-expansion is $q_x:=\min\set{2, 1+1/x}$.

\begin{lemma} \label{lem:char-quasi-expansion}\mbox{}

\begin{itemize}
\item[{\rm(i)}]
Let $x>0$. Then the map $q\mapsto\Phi_x(q)$ is strictly increasing in $(1, q_x]$, and is left continuous with respect to the order topology.
\item[{\rm(ii)}]
Let $q\in(1,2]$. Then the map $x\mapsto\Phi_x(q)$ is strictly increasing in $[0, 1/(q-1)]$, and is left continuous with respect to the order topology.

\item[{\rm(iii)}]
For each $q\in(1,2]$ and $x\in[0, 1/(q-1)]$, the sequence $\Phi_x(q)=(x_i(q))$ satisfies
\[
x_{n+1}(q)x_{n+2}(q)\ldots\lle \al(q)\quad\textrm{whenever}\quad x_n(q)=0.
\]
\end{itemize}
\end{lemma}

Taking $x=1$ in Lemma \ref{lem:char-quasi-expansion} (i) we see that the map $q\mapsto \Phi_1(q)=\al(q)$ is strictly increasing and left continuous in $(1, q_1]=(1,2]$. Furthermore, by Lemma \ref{lem:char-quasi-expansion} (iii) it follows that for each $q\in(1,2]$ the sequence $\al(q)=\al_1(q)\al_2(q)\ldots$ satisfies
\[
\al_{n+1}(q)\al_{n+2}(q)\ldots\lle \al(q)\quad\forall n\ge 1.
\]

For $q\in(1, 2]$ let $\u_q$ be the set of $x\in[0, 1/(q-1)]$ having a unique $q$-expansion, and let $\us_q$ be the set of all unique $q$-expansions. In view of (\ref{eq:expansion-base-q}), it is clear that $(d_i)\in\us_q$ if and only if $(d_i)_q\in\u_q$. The following characterization of $\us_q$ is well-known (see \cite{Erdos_Joo_Komornik_1990}).

\begin{lemma}\label{lem:unique-expansion}
Let $q\in(1,2]$. Then $(d_i)\in\us_q$ if and only if
\[
\begin{cases}
d_{n+1}d_{n+2}\ldots \prec \al(q)&\textrm{whenever}\ d_n=0,\\
d_{n+1}d_{n+2}\ldots\succ \overline{\al(q)}&\textrm{whenever}\ d_n=1.
\end{cases}
\]
\end{lemma}

Note by Lemma \ref{lem:char-quasi-expansion} (i) that the map $q\mapsto\al(q)$ is increasing. Then by Lemma \ref{lem:unique-expansion} it follows that $\us_p\subseteq\us_q$ whenever $p<q$.

Recall that $\ub(x)$ is the set of all bases $q\in(1,q_x]$ such that $x$ has a unique $q$-expansion. So, $q\in\ub(x)$ if and only if $\Phi_x(q)\in\us_q$. In particular, for $x=1$ the following characterizations of $\ub(1)=\ub$ and its closure $\overline{\ub}$ can be deduced from Lemma \ref{lem:unique-expansion} (see \cite{Komornik_Loreti_2007} for a detailed proof).

\begin{lemma}\label{lem:univoque-bases}\mbox{}

\begin{itemize}
\item[{\rm(i)}] $q\in\ub$ if and only if the quasi-greedy expansion $\al(q)$ satisfies
\[
\overline{\al(q)}\prec\si^n(\al(q))\prec \al(q)\quad\forall ~n\ge 1.
\]

\item[{\rm(ii)}] $q\in\overline{\ub}$ if and only if the quasi-greedy expansion $\al(q)$ satisfies
\[
\overline{\al(q)}\prec \si^n(\al(q))\lle \al(q)\quad \forall ~n\ge 1.
\]
\end{itemize}
\end{lemma}

\subsection{The Komornik-Loreti constant and de Vries-Komornik numbers} \label{subsec:KL}

An important role in this article is played by the classical Thue-Morse sequence $(\tau_i)_{i=0}^\infty$ (see \cite{Allouche_Shallit_1999}). It is defined by $\tau_i=s_i\mod 2$, where $s_i$ is the number of 1's in the binary representation of $i$. Thus, it satisfies the recursions $\tau_{2i}=\tau_i$ and $\tau_{2i+1}=1-\tau_i$, for $i=0,1,2,\dots$. Recall from \cite{Komornik-Loreti-1998} that the smallest base in $\ub$ is the {\em Komornik-Loreti constant} $q_{KL}\approx 1.78723$, defined as the unique root in $(1,2)$ of the equation
\[
1=(\tau_i)_q=\sum_{i=1}^\f\frac{\tau_i}{q^i}.
\]
The truncated Thue-Morse sequence $(\tau_i)_{i=1}^\infty=1101\,0011\,00101101\dots$ has some useful properties, e.g. for each $n\in\N$ we have
\begin{equation} \label{eq:TM-recursion}
\tau_1\ldots\tau_{2^n}=\tau_1\ldots \tau_{2^{n-1}}\overline{\tau_1\ldots\tau_{2^{n-1}}}^+,
\end{equation}
and
\begin{equation} \label{eq:property-tau}
 \overline{\tau_1\ldots\tau_{2^n-i}}\prec\tau_{i+1}\ldots\tau_{2^n}\lle\tau_1\ldots\tau_{2^n-i}\quad\forall~0\le i<2^n.
\end{equation}
By (\ref{eq:property-tau}) and Lemma \ref{lem:char-quasi-expansion} (iii) it follows that $\al(q_{KL})=\tau_1\tau_2\ldots$ is the quasi-greedy $q_{KL}$-expansion of $1$. In fact, by Lemma \ref{lem:unique-expansion} one can verify that $(\tau_i)_{i=1}^\f$ is the unique expansion of $1$ in base $q_{KL}$, i.e., $(\tau_i)_{i=1}^\f\in\us_{q_{KL}}$.

For each $n\in\N$, we define a base $\hat q_n$ via the quasi-greedy expansion of $1$ by
\begin{equation}\label{eq:qn}
\al(\hat q_n)=(\tau_1\ldots \tau_{2^{n-1}}\overline{\tau_1\ldots \tau_{2^{n-1}}})^\f=(\tau_1\ldots \tau_{2^n}^-)^\f.
\end{equation}
Note that $\hat{q}_n$ also satisfies the equation
\[
(\tau_1\dots\tau_{2^n})_{\hat{q}_n}=1,
\]
where  for a finite word $\sc$ and $q\in(1,2]$, we write $(\sc)_q$ to abbreviate $(\sc 0^\infty)_q$.
In particular, $\hat q_1=q_G=(1+\sqrt{5})/2\approx 1.618$ is the golden ratio. A few other useful values are $\hat{q}_2\approx 1.7549$ and $\hat{q}_3\approx 1.7846$. By Lemma \ref{lem:char-quasi-expansion} (i) it follows that the sequence $(\hat{q}_n)$ is strictly increasing with limit $q_{KL}$.

To end this section we introduce the de Vries-Komornik numbers, which are based on the following notion.
A word $a_1\dots a_m\in\{0,1\}^m$ is {\em admissible} if
\begin{equation} \label{eq:admissible word}
\overline{a_1\ldots a_{m-i}}\lle a_{i+1}\ldots a_m\prec a_1\ldots a_{m-i}\quad\forall ~1\le i<m.
\end{equation}

Note that the complement of $\overline{\ub}$ in $(1,2]$ is a countable union of open intervals, i.e., $(1,2]\setminus\overline{\ub}=\bigcup_{n=1}^\f(\hat q_0^{(n)}, \hat q_c^{(n)})$. Here each left endpoint $\hat q_0^{(n)}$ is algebraic, and each right endpoint $\hat q_c^{(n)}$, called a \emph{de Vries-Komornik number}, is transcendental (cf.~\cite{Kong_Li_2015}). We also have $\hat q_0^{(n)}\in(\overline{\ub}\setminus\ub)\cup\set{1}$ and $\hat q_c^{(n)}\in\ub$ for all $n\ge 1$. Furthermore, for a connected component $(\hat q_0^{(n)}, \hat q_c^{(n)})$ there exists an admissible word $\sa=\sa_n=a_1\ldots a_m$ such that $\al(\hat q_0^{(n)})=\sa^\f$ and $\al(\hat q_c^{(n)})=(\theta_i(\sa))$, where the sequence $(\theta_i):=(\theta_i(\sa))$ is defined recursively by
\[
\theta_1\ldots\theta_m=\sa^+,\qquad\textrm{and}\qquad \theta_{2^n m+1}\ldots \theta_{2^{n+1}m}=\overline{\theta_1\ldots \theta_{2^n m}}^{\,+}\quad\forall n\ge 0.
\]
We call $(\hat q_0^{(n)}, \hat q_c^{(n)})$ a {\em basic interval} generated by $\sa$, and denote it by $(\hat q_0(\sa), \hat q_c(\sa))$ to emphasize its dependence on $\sa$. Observe that the first basic interval is $(\hat q_0^{(1)}, \hat q_c^{(1)})=(1, q_{KL})$, which is generated by the word $\sa=0$. Here we set $\al(1):=0^\f$. With this convention, $q_{KL}$ is a de Vries-Komornik number generated by the word $\sa=0$, i.e. $q_{KL}=\hat{q}_c(0)$.


Following \cite{DeVries_Komornik_2008}, we introduce the set
\begin{equation} \label{def:V}
\vb:=\set{q\in(1,2]: \overline{\al(q)}\lle \si^n(\al(q))\lle \al(q)~\forall n\ge 0}.
\end{equation}
It was shown in \cite{DeVries_Komornik_2008} that for any basic interval $(\hat q_0(\sa), \hat q_c(\sa))$ we have $(\hat q_0(\sa), \hat q_c(\sa))\cap\vb=\set{\hat q_n(\sa): n\in\N}$, where $\hat{q}_n(\sa)$ is given by
\[
\al(\hat q_n(\sa))=(\theta_1(\sa)\ldots \theta_{2^{n}m}(\sa)^-)^\f.
\]
By Lemma \ref{lem:char-quasi-expansion} this implies that $\hat q_0(\sa)<\hat q_1(\sa)<\cdots<\hat q_n(\sa)<\hat q_{n+1}(\sa)<\cdots<\hat q_c(\sa)$. So, $(\hat q_0(\sa), \hat q_c(\sa))\setminus\vb=\bigcup_{n=0}^\f\big(\hat q_n(\sa), \hat q_{n+1}(\sa)\big)$. The sequence $(\hat{q}_n(\sa))_{n=1}^\infty$ can be thought of as a local analog of the sequence $(\hat{q}_n)$.

\section{An algorithm to determine $\inf\ub(x)$ for $x\in(0, 1)$} \label{sec:algorithm}

In this section we will introduce an efficient algorithm to calculate the explicit value of $q_s(x)$ for $x\in(0,1)$. Recall that $q_G=(1+\sqrt{5})/2$, and observe that $(0,1)=\bigcup_{k=1}^\f[q_G^{-k}, q_G^{-k+1})$.
The following result forms the theoretical foundation of our algorithm.

\begin{proposition} \label{prop:lower-bound}
Let $x\in[q_G^{-k}, q_G^{-k+1})$ for some $k\in\N$, and let $q^*(x)$ be the unique root in $(1,2)$ of
\[q^k(q-1)=\frac{1}{x}.\]
Then
\begin{enumerate}[{\rm(i)}]
\item $q^*(x)>q_G$, and $q^*(x)\notin\ub(x)$;
\item $q_s(x)\geq q^*(x)$, with equality if and only if $q^*(x)\in\overline{\ub}$.
\end{enumerate}
\end{proposition}

The lower bound $q^*(x)$ is illustrated in Figure \ref{fig:2} below.

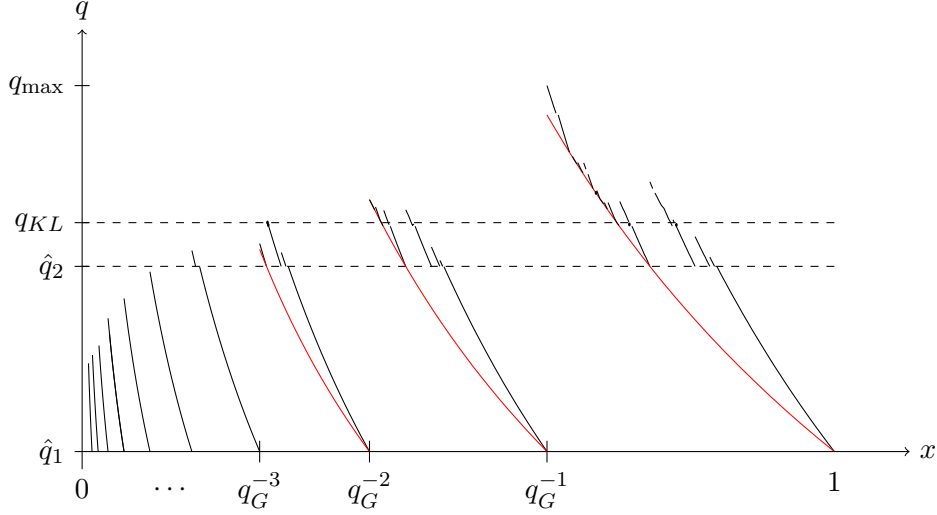
\begin{figure}[h!]
\begin{center}
\begin{tikzpicture}[xscale=10,yscale=18]

\draw [->] (-0.01,1.618) node[anchor=east] {$\hat{q}_1$}  -- (1.1,1.618) node[anchor=west] {$x$};
\draw [->] (0,1.605) node[anchor=north] {$0$} -- (0,1.93) node[anchor=south] {$q$};

\draw [dashed] (1,1.7549)--(0,1.7549);
\draw (-0.01,1.7549) node[anchor=east] {$\hat{q}_2$} -- (0.01,1.7549);
\draw [dashed] (1,1.7872)--(0,1.7872);
\draw (-0.01,1.7872) node[anchor=east] {$q_{KL}$} -- (0.01,1.7872);
\draw (1,1.61) node[anchor=north] {$1$};
\draw (-0.01,1.8885) node[anchor=east] {$q_{\max}$} -- (0.01,1.8885);
\draw (0.618,1.61) node[anchor=north] {$q_G^{-1}$} -- (0.618,1.626);
\draw (0.382,1.61) node[anchor=north] {$q_G^{-2}$} -- (0.382,1.626);
\draw (0.236,1.61) node[anchor=north] {$q_G^{-3}$} -- (0.236,1.626);
\draw (0.12,1.59) node {$\cdots$};


\draw[domain=1.618:1.7549,variable=\q] plot({\q/(pow(\q,2)-1)},{\q});  
\draw[domain=1.618:1.7549,variable=\q] plot({1/(pow(\q,2)-1)},{\q});	
\draw[domain=1.618:1.7549,variable=\q] plot({1/(\q*(pow(\q,2)-1))},{\q});	
\draw[domain=1.618:1.7549,variable=\q] plot({1/(\q*\q*(pow(\q,2)-1))},{\q});	
\draw[domain=1.618:1.7510,variable=\q] plot({1/(\q*\q*\q*(pow(\q,2)-1))},{\q});	
\draw[domain=1.618:1.7313,variable=\q] plot({1/(pow(\q,4)*(pow(\q,2)-1))},{\q});	
\draw[domain=1.618:1.7167,variable=\q] plot({1/(pow(\q,5)*(pow(\q,2)-1))},{\q});	
\draw[domain=1.618:1.7055,variable=\q] plot({1/(pow(\q,5)*(pow(\q,2)-1))},{\q});	
\draw[domain=1.618:1.6966,variable=\q] plot({1/(pow(\q,6)*(pow(\q,2)-1))},{\q});	
\draw[domain=1.618:1.6894,variable=\q] plot({1/(pow(\q,7)*(pow(\q,2)-1))},{\q});	
\draw[domain=1.618:1.6833,variable=\q] plot({1/(pow(\q,8)*(pow(\q,2)-1))},{\q});	

\draw[domain=1.7549:1.7846,variable=\q] plot({1/\q+(\q+1)/(pow(\q,5)-\q)},{\q});	
\draw[domain=1.7549:1.7846,variable=\q] plot({1/pow(\q,2)+(\q+1)/(pow(\q,6)-\q*\q)},{\q});	
\draw[domain=1.7549:1.7718,variable=\q] plot({1/pow(\q,3)+(\q+1)/(pow(\q,7)-pow(\q,3))},{\q});	
\draw[domain=1.7549:1.7846,variable=\q] plot({1/\q+1/pow(\q,3)+(\q+1)/(pow(\q,7)-pow(\q,3))},{\q});	
\draw[domain=1.7549:1.7846,variable=\q] plot({1/pow(\q,2)+1/pow(\q,4)+(\q+1)/(pow(\q,8)-pow(\q,4))},{\q}); 
\draw[domain=1.7549:1.7846,variable=\q] plot({1/pow(\q,3)+1/pow(\q,5)+(\q+1)/(pow(\q,9)-pow(\q,5))},{\q});	
\draw[domain=1.7549:1.7668,variable=\q] plot({1/pow(\q,4)+1/pow(\q,6)+(\q+1)/(pow(\q,10)-pow(\q,6))},{\q});	
\draw[domain=1.7549:1.7769,variable=\q] plot({1/\q+1/pow(\q,3)+1/pow(\q,5)+(\q+1)/(pow(\q,9)-pow(\q,5))},{\q});  
\draw[domain=1.7549:1.7692,variable=\q] plot({1/pow(\q,2)+1/pow(\q,4)+1/pow(\q,6)+(\q+1)/(pow(\q,10)-pow(\q,6))},{\q});  
\draw[domain=1.7549:1.7655,variable=\q] plot({1/pow(\q,3)+1/pow(\q,5)+1/pow(\q,7)+(\q+1)/(pow(\q,11)-pow(\q,7))},{\q});  
\draw[domain=1.7549:1.7618,variable=\q] plot({1/\q+1/pow(\q,3)+1/pow(\q,5)+1/pow(\q,7)+(\q+1)/(pow(\q,11)-pow(\q,7))},{\q});  
\draw[domain=1.7549:1.7594,variable=\q] plot({pow(\q,-2)+pow(\q,-4)+pow(\q,-6)+pow(\q,-8)+(\q+1)/(pow(\q,12)-pow(\q,8))},{\q});  

\draw[red, domain=1.8668:1.6180,variable=\q] plot({1/(\q*(\q-1))},{\q});
\draw[red, domain=1.8042:1.6180,variable=\q] plot({1/(pow(\q,2)*(\q-1))},{\q});
\draw[red, domain=1.7674:1.6180,variable=\q] plot({1/(pow(\q,3)*(\q-1))},{\q});


\draw[domain=1.8393:1.8668,variable=\q] plot({pow(\q,2)/(pow(\q,3)-1)},{\q}); 	
\draw[domain=1.8124:1.8233,variable=\q] plot({(pow(\q,4)+\q)/(pow(\q,5)-1)},{\q});	
\draw[domain=1.8124:1.8174,variable=\q] plot({1/\q+(pow(\q,3)+1)/(\q*(pow(\q,5)-1))},{\q});	
\draw[domain=1.8668:1.8706,variable=\q] plot({(pow(\q,5)+\q+1)/(pow(\q^6-1))},{\q});	

\draw (0.618,1.8885)--(0.6204,1.8843)--(0.6228,1.8802)--(0.6252,1.8760)--(0.6276,1.8719)--(0.63,1.8681);

\draw (0.6518,1.8362)--(0.6548,1.8332)--(0.6578,1.8305);
\draw (0.6588,1.8316)--(0.6628,1.8270)--(0.6658,1.8239);
\draw (0.6668,1.8314)--(0.6698,1.8267);

\draw[fill] (0.6832,1.8092) circle[radius=0.001];
\draw (0.6842,1.8109)--(0.6882,1.8050)--(0.6932,1.8012);
\draw (0.6942,1.8020)--(0.6982,1.7973);
\draw (0.6992,1.8020)--(0.7032,1.7965)--(0.7072,1.7911)--(0.7108,1.7872)--(0.7138,1.7851);
\draw (0.7148,1.8028)--(0.7168,1.8002)--(0.7188,1.7977)--(0.7208,1.7956)--(0.7228,1.7930)--(0.7248,1.7903)--(0.7268,1.7877);

\draw (0.7615,1.8091)--(0.7645,1.8060)--(0.7675,1.8026)--(0.7705,1.7997)--(0.7735,1.7978)--(0.7765,1.7940)
	--(0.7795,1.7902)--(0.7825,1.7872)--(0.7845,1.7847);
\draw (0.7855,1.7893)--(0.7872,1.7872);

\draw (0.382,1.8042)--(0.385,1.8010)--(0.389,1.7966);
\draw (0.390,1.7988)--(0.393,1.7941)--(0.396,1.7892)--(0.400,1.7849);
\draw (0.401,1.7963)--(0.404,1.7916)--(0.408,1.7857);

\draw (0.4302,1.7967)--(0.4342,1.7920)--(0.4372,1.7877)--(0.4392,1.7852)--(0.4412,1.7862);

\draw[fill] (0.7275,1.7859) circle[radius=0.001];
\draw[fill] (0.7900,1.7855) circle[radius=0.001];
\draw[fill] (0.2464,1.7873) circle[radius=0.001];
\draw[fill] (0.2474,1.7855) circle[radius=0.001];

\useasboundingbox(-0.25,1.95);
\end{tikzpicture}
\end{center}
\caption{Graph of $q_s(x)$ for $0<x<1$ with the lower bound $q^*(x)$ from Proposition \ref{prop:lower-bound} in red, plotted for $q_G^{-3}<x<1$.}
\label{fig:2}
\end{figure}

\begin{proof}
We give a detailed proof for the case $k=1$; for general $k$, the result is obtained by preceding all expansions below by the prefix $0^{k-1}$.

Since $x<1$ and $q_G(q_G-1)=1$, by the definition of $q^*(x)$ it follows that $q^*(x)>q_G$.
Now put $q=q^*(x)$, and note that $x$ has two expansions in base $q$:
\[
x=(01^\infty)_q=\big(1\,\overline{\alpha(q)}\big)_q.
\]
This follows since
\[
x=\frac{1}{q(q-1)}=\sum_{n=2}^\infty \frac{1}{q^n}=\frac{1}{q}+\frac{1}{q}\left(\frac{1}{q-1}-1\right).
\]
Thus $q^*(x)\not\in \ub(x)$, completing the proof of (i).

We now turn to the proof of (ii).  Write $\alpha(q)=\alpha_1\alpha_2\dots$. We split the proof that $q_s(x) \ge q$ into two cases.

\medskip

{\em Case A.} Assume first that $q\not\in\overline{\ub}$. Then by Lemma \ref{lem:univoque-bases} (ii) there is a smallest integer $n\geq 0$ such that
\begin{equation}
\overline{\alpha_n\alpha_{n+1}\dots}\succeq \alpha(q).
\label{eq:large-reflected-tail}
\end{equation}
We can see that $n\geq 3$, since $q>q_G$ implies $\alpha_1=\alpha_2=1$.
Note $\alpha_{n-1}=1$ by the minimality of $n$. Construct the sequence
\[
(c_i):=(1\,\overline{\alpha_1\dots\alpha_{n-2}})^\infty=1(\overline{\al_1\ldots\al_{n-1}}^+)^\f.
\]
We claim that
\begin{equation} \label{eq:c_i-claim}
y:=(c_i)_q>x \qquad\mbox{and}\qquad (c_i)\in\us_q.
\end{equation}
{Observe that} $y=\lim_{k\to\infty} \big((1\,\overline{\alpha_1\dots\alpha_{n-2}})^k 01^\infty\big)_q$, and
\begin{align}
\begin{split}
x&=(1\,\overline{\alpha(q)})_q<(1\,\overline{\alpha_1\dots\alpha_{n-1}}\,1^\infty)_q\\
&=(1\,\overline{\alpha_1\dots\alpha_{n-2}}\,01^\infty)_q=\big(1\,\overline{\alpha_1\dots\alpha_{n-2}}\,1\overline{\alpha(q)}\big)_q\\
&<\big((1\,\overline{\alpha_1\dots\alpha_{n-2}})^2 01^\infty\big)_q<\dots<\big((1\,\overline{\alpha_1\dots\alpha_{n-2}})^k 01^\infty\big)_q.
\end{split}
\label{eq:x-y-calculation}
\end{align}
Thus, $x<y$, proving the first half of \eqref{eq:c_i-claim}.

The proof that $(c_i)\in\us_q$ is more involved. By Lemma \ref{lem:unique-expansion} it suffices to show that
\begin{equation}
\overline{\alpha(q)}\prec \sigma^m((c_i))\prec \alpha(q) \qquad\mbox{for all $m\geq 1$}.
\label{eq:c-sandwich}
\end{equation}
That $\overline{\alpha(q)}\prec \sigma^{m}((c_i))$ for all $m\geq 1$ follows immediately since for $1\leq i\leq n-2$,
\[
\overline{\alpha_i\dots\alpha_{n-2}}\,1\succ \overline{\alpha_i\dots\alpha_{n-2}}\,0=\overline{\alpha_i\dots\alpha_{n-1}}\succeq \overline{\alpha_1\dots\alpha_{n-i}},
\]
where the last inequality is a consequence of Lemma \ref{lem:char-quasi-expansion} (iii).

To establish the second inequality in (\ref{eq:c-sandwich}), it is sufficient to show that
\begin{equation}
\overline{\alpha_j\dots\alpha_{n-1}}^+\, \overline{\alpha_1\dots\alpha_{j-1}}\prec \alpha_1\dots\alpha_{n-1}\quad\textrm{for all }1\le j\le n-1,
\label{eq:j-inequality}
\end{equation}
because every tail of the sequence $(c_i)$ begins with $\overline{\alpha_j\dots\alpha_{n-1}}^+\, \overline{\alpha_1\dots\alpha_{j-1}}$ for some $j\in\{1,\dots,n-1\}$.
Since $n\geq 3$ and $\al_1=1$, \eqref{eq:j-inequality} is trivial for $j=1$, so we assume $2\le j\le n-1$.
{We claim that
\begin{equation} \label{eq:mar-10-1}
\overline{\al_j\ldots\al_{n-1}}^+\lle\al_1\ldots\al_{n-j}.
\end{equation}
Suppose otherwise; then $\overline{\al_j\ldots\al_{n-1}}\succeq \al_1\ldots \al_{n-j}$. Using (\ref{eq:large-reflected-tail}) and Lemma \ref{lem:char-quasi-expansion} (iii) this implies
\[
\overline{\al_j\al_{j+1}\ldots}\lge \al_1\ldots \al_{n-j}\al_1\al_2\ldots\lge \al_1\al_2\ldots,
\]
contradicting the minimality of $n$. Hence we have \eqref{eq:mar-10-1}, and by the same argument (replacing $j$ with $n-j+1$), we also have $\overline{\al_{n-j+1}\ldots\al_{n-1}}\prec\al_1\ldots \al_{j-1}$.
Taking the reflection on both sides yields}
\[
\overline{\al_1\ldots \al_{j-1}}\prec \al_{n-j+1}\ldots\al_{n-1}.
\]
This, together with (\ref{eq:mar-10-1}), proves \eqref{eq:j-inequality}, and hence \eqref{eq:c-sandwich}. As a result, $(c_i)=\Phi_y(q)\in\us_q$. This completes the proof of \eqref{eq:c_i-claim}.

We can now prove that $q_s(x)\geq q$. Let $1<p<q$, and suppose $x\in\u_p$. Let $(d_i)=\Phi_x(p)$ be the unique expansion of $x$ in base $p$. Then by Lemma \ref{lem:char-quasi-expansion} and {\eqref{eq:c_i-claim}} it follows that
\begin{equation} \label{eq:mar-10-2}
(d_i)=\Phi_x(p)\prec \Phi_x(q)\prec \Phi_y(q)=(1\overline{\al_1\ldots\al_{n-2}})^\f.
\end{equation}
We will obtain a contradiction by showing that $(d_i)=(1\,\overline{\alpha_1\dots\alpha_{n-2}})^\infty$. We will show by induction that $(d_i)$ begins with the word $(1\,\overline{\alpha_1\dots\alpha_{n-2}})^k 1$ for all $k\geq 0$.
Take first $k=0$. If $d_1=0$, then $x=(d_i)_p\le 1/p$, and hence $p\le 1/x\le q_G$. But then the only two unique expansions in base $p$ are $0^\f$ and $1^\f$, and clearly $0^\f$ is not an expansion of $x$, since $x\ge q_G^{-1}$. Thus $d_1=1$, and the claim holds for $k=0$.

Now suppose the claim holds for some integer $k\ge 0$, i.e., $(d_i)$ begins with $(1\,\overline{\alpha_1\dots\alpha_{n-2}})^k 1$. Since $x\in\u_p$, by Lemma \ref{lem:unique-expansion} we have $\sigma^{k(n-1)+1}((d_i))\succ \overline{\alpha(p)}\succ\overline{\alpha(q)}$, and by assumption, $d_{k(n-1)+1}=1$, so
\begin{equation}
\sigma^{k(n-1)}((d_i))\succ 1\,\overline{\alpha(q)}.
\label{eq:k-step-lower-estimate}
\end{equation}
On the other hand, note by (\ref{eq:mar-10-2}) that  $(d_i)\prec (1\,\overline{\alpha_1\dots\alpha_{n-2}})^\infty$. By the induction hypothesis it follows that $\sigma^{k(n-1)}((d_i))\prec (1\,\overline{\alpha_1\dots\alpha_{n-2}})^\infty$. Therefore, $\sigma^{k(n-1)}((d_i))$ must begin with either $1\,\overline{\alpha_1\dots\alpha_{n-2}}\,0=1\,\overline{\alpha_1\dots\alpha_{n-1}}$ or $1\,\overline{\alpha_1\dots\alpha_{n-2}}\,1$. But the former is impossible, since in view of \eqref{eq:k-step-lower-estimate} and (\ref{eq:large-reflected-tail}) it would imply
\[\sigma^{(k+1)(n-1)+1}((d_i))\succ\overline{\alpha_n\alpha_{n+1}\dots}\succeq\alpha(q)\succ\alpha(p),\] contradicting that $(d_i)\in\us_p$. Therefore, $(d_i)$ begins with $(1\,\overline{\alpha_1\dots\alpha_{n-2}})^{k+1}1$.

By induction it follows that $(d_i)=(1\,\overline{\alpha_1\dots\alpha_{n-2}})^\infty=\Phi_y(q)$, which is impossible by (\ref{eq:mar-10-2}). As a result, $q_s(x)\geq q$. We will show later, in Subsection \ref{subsec:algorithm}, that the inequality is strict in this case.

\medskip

{\em Case B.} Assume next that $q\in\overline{\ub}$. Then by Lemma \ref{lem:univoque-bases} (ii) we have
\[
\overline{\al(q)}\lle \sigma^n(\overline{\alpha(q)})\prec \alpha(q) \qquad\mbox{for all {$n\geq 1$}},
\]
so $1\,\overline{\alpha(q)}$ is not only an expansion of $x$, but in fact the quasi-greedy expansion of $x$ in base $q$:
\[
\Phi_x(q)=1\,\overline{\alpha(q)}.
\]
Furthermore, by \cite[Lemma 4.1]{Komornik_Loreti_2007} there are infinitely many integers $m$ such that $\alpha_m=1$ and $\alpha_1\dots\alpha_m^-$ is admissible; that is,
\begin{equation}
\overline{\alpha_1\dots\alpha_{m-i}}\preceq \alpha_{i+1}\dots \alpha_m^-\prec \alpha_1\dots\alpha_{m-i} \qquad\mbox{for all ${0}\leq i<m$}.
\label{eq:admissible}
\end{equation}
For any such $m$, we can define the sequence
\[(c_i):=(1\,\overline{\alpha_1\dots\alpha_{m-1}})^\infty=1(\overline{\al_1\ldots\al_m}^+)^\f.\]
We claim that
\begin{equation}\label{eq:mar-10-3}
\overline{\al(q)}\prec \si^n((c_i))\prec \al(q)\quad\forall\, n\ge 0.
\end{equation}
Note that $\overline{(c_i)}=0(\alpha_1\dots \alpha_{m}^-)^\infty$, and hence (\ref{eq:admissible}) immediately gives the first inequality in (\ref{eq:mar-10-3}). For the second inequality we observe from  \eqref{eq:admissible} that
\[
\al_{i+1}\ldots\al_m^-\al_1\ldots \al_i\succ\overline{\al_1\ldots \al_m}\quad\forall\, 0\le i<m.
\]
Taking the reflection on both sides gives
\[
\overline{\al_{i+1}\ldots\al_m}^+\,\overline{\al_1\ldots \al_i}\prec \al_1\ldots \al_m\quad\forall\, 0\le i<m.
\]
This 
proves the second inequality in (\ref{eq:mar-10-3}). From (\ref{eq:mar-10-3}) and Lemma \ref{lem:unique-expansion}, it follows that $(c_i)\in\us_q$.

Note that $(c_i)=1(\overline{\al_1\ldots\al_m}^+)^\f\succ 1\overline{\al(q)}=\Phi_x(q)$. Then $(c_i)$ is an expansion of $x$ for some base $r_m>q$. Since $\us_q\subseteq\us_{r_m}$, it follows that $r_m\in\ub(x)$ and $(c_i)=\Phi_x(r_m)$. Letting $m\to\f$ along a subsequence satisfying \eqref{eq:admissible}, we have
\begin{equation} \label{eq:mar-10-4}
\Phi_x(r_m)=1(\overline{\al_1\ldots\al_m}^+)^\f\searrow 1\,\overline{\alpha(q)}=\Phi_x(q).
\end{equation}
{We claim now that $r_m\searrow q$. By \eqref{eq:mar-10-4} and Lemma \ref{lem:char-quasi-expansion}, the sequence $(r_m)$ is decreasing and bounded below by $q$. Suppose, by way of contradiction, that $\lim_{m\to\infty}r_m=r>q$. Then, again by Lemma \ref{lem:char-quasi-expansion}, $\Phi_x(q)\prec \Phi_x(r)\lle \lim_{m\to\infty}\Phi_x(r_m)=\Phi_x(q)$, a contradiction. Hence, $r_m\searrow q$ and so $q_s(x)\le q$.}

Finally we prove that $q_s(x)\ge q$. Suppose $p<q$ and $x\in\u_p$. Let $(d_i)=\Phi_x(p)$ be the unique $p$-expansion of $x$. Then by Lemma \ref{lem:char-quasi-expansion} (i) we have
\begin{equation}\label{eq:mar-10-5}
(d_i)=\Phi_x(p)\prec \Phi_x(q)=1\overline{\al(q)}.
\end{equation}
Using $x\ge q_G^{-1}$ and the same argument as before we can deduce that $d_1=1$. Since $(d_i)\in\us_p$, by Lemma \ref{lem:unique-expansion} and Lemma \ref{lem:char-quasi-expansion} (i) it follows that
\[
d_2d_3\ldots\succ\overline{\al(p)}\succ\overline{\al(q)}.
\]
So, $(d_i)\succ 1\overline{\al(q)}$, contradicting (\ref{eq:mar-10-5}). Therefore, $q_s(x)=q$.
\end{proof}

\begin{remark} \label{rem:algorithm-1}\mbox{}

\begin{itemize}
\item[{\rm(i)}] In Case A in the above proof, since $y>x$ and $y=(c_i)_q\in \u_q$, there is a base $r>q$ such that $(c_i)$ is an expansion of $x$ in base $r$, and since $\us_q\subset\us_r$, it follows that $x\in \u_r$, so $r\in\ub(x)$. The larger the integer $n$ is, the closer $r$ is to $q$. Thus the proof also gives an implicit upper bound for $q_s(x)$. Hence, this method provides a starting point for an iterative procedure to approximate $q_s(x)$ arbitrarily closely; we will describe it in subsection \ref{subsec:algorithm} below.

\item[{\rm(ii)}] For each $k\in\N$ we have  $q^*(x)\searrow q_G$ as $x\nearrow 1/q_G^{k-1}$. Therefore, $q_s(x)>q_G$ for all $x\in(0, 1)$.

\item[{\rm(iii)}] Note  that
\[
x:=\frac{1}{q_{KL}(q_{KL}-1)} \in[q_G^{-1}, 1)\quad\mbox{and} \quad y:=\frac{1}{q_{KL}^2(q_{KL}-1)} \in[q_G^{-2}, q_G^{-1}).
\]
Since $q_{KL}\in\ub$, Proposition \ref{prop:lower-bound} implies that
$q_s(x)=q_{KL}\notin\ub(x)$ and $q_s(y)=q_{KL}\notin\ub(y)$.
However, for $k\geq 3$, $1/(q_{KL}^k(q_{KL}-1))$ does not lie in the interval $[1/{q_G^k},1/q_G^{k-1})$.
\end{itemize}
\end{remark}

\subsection{Description of the algorithm} \label{subsec:algorithm}

Motivated by Proposition \ref{prop:lower-bound} we introduce the following algorithm to calculate $q_s(x)=\inf\ub(x)$ for $x\in[q_G^{-1}, 1)$. We emphasize that this algorithm can be easily generalized to calculate $q_s(x)$ for all $x\in(0,1)$; see Remark \ref{rem:three-cases} below.

\medskip

{\bf Step $1$.} Take $x\in[q_G^{-1}, 1)$. Let $q_1=q^*(x)$, i.e.,
\begin{equation}\label{eq:mar-11-1}
x=(1\overline{\al(q_1)})_{q_1}=(01^\f)_{q_1}.
\end{equation} If $q_1\in\overline{\ub}$, then by Proposition \ref{prop:lower-bound} we have ${q_s(x)}=q_1$ and STOP.
Otherwise, let
\[n_1:=\min\set{n\in\N: \overline{\al_{n}(q_1)\al_{n+1}(q_1)\ldots}\lge \al(q_1)}\] be the number $n$ defined as in (\ref{eq:large-reflected-tail}). Then $\al_{n_1-1}(q_1)=1$. Put $B_1:=1\,\overline{\alpha_1(q_1)\dots\alpha_{n_1-2}(q_1)}$, and let $r_1>q_1$ be the base such that
\begin{equation}\label{eq:mar-11-2}
\big((B_1)^\infty)_{r_1}=x.
\end{equation}
From the work in Case A in the proof of Proposition \ref{prop:lower-bound} it follows that $r_1\in\ub(x)$ and $q_1\leq {q_s(x)}\leq r_1$.

\medskip

{\bf Step $2$.} Let $q_2$ be the base such that
\begin{equation}\label{eq:mar-11-3}
\big(B_1 1\,\overline{\alpha(q_2)}\big)_{q_2}=\big(B_1 01^\infty\big)_{q_2}=x.
\end{equation}

{\em Claim 1:} $q_2\in(q_1,r_1)$, and $q_s(x)\geq q_2$. (The proof of this and the next two claims will be deferred until we have finished describing the algorithm.)

If $q_2\in\overline{\ub}$, then the same argument as in the proof of  Proposition \ref{prop:lower-bound} yields ${q_s(x)}=q_2$ and STOP.
Otherwise, let
\[
n_2:=\min\{n\in\N:\overline{\alpha_n(q_2)\alpha_{n+1}(q_2)\dots}\succeq \alpha(q_2)\},
\]
and set $B_2:=1\,\overline{\alpha_1(q_2)\dots\alpha_{n_2-2}(q_2)}$. In view of the definition of $n_2$ we also define \[m_2:=\min\{m>n_2:\overline{\al_{n_2}(q_2)\ldots\alpha_m(q_2)}\succ\al_1(q_2)\ldots\alpha_{m-n_2+1}(q_2)\},\] or $m_2:=\infty$ if no such $m$ exists. Let $r_2>q_2$ be the base such that
\[
(B_1 B_2^\infty)_{r_2}=x.
\]

{\em Claim 2:} $r_2\in\ub(x)$.

{\em Claim 3:} $r_2\leq r_1$.

In view of the above claims, in this step we conclude that $q_1<q_2\le {q_s(x)}\le r_2\le r_1$, and $r_2\in\ub(x)$.

\medskip

{\bf Step $k$.} Let $k\geq 3$, and assume we have constructed bases $q_1<q_2<\dots<q_{k-1}\le r_{k-1}\le\dots\le r_2\le r_1$ and finite words $B_1,B_2,\dots,B_{k-1}$ such that
\[
(B_1\dots B_{j-1}B_{j}^\infty)_{r_j}=x, \qquad j=2,\dots,k-1.
\]
We then define $q_k$ so that
\[
\big(B_1\dots B_{k-1} 1\,\overline{\alpha(q_k)}\big)_{q_k}=\big(B_1\dots B_{k-1}01^\f\big)_{q_k}=x.
\]
If $q_k\in\overline{\ub}$, then by the same argument as in the proof of Proposition \ref{prop:lower-bound} we obtain ${q_s(x)}=q_k$ and STOP.
Otherwise, let
\begin{equation} \label{eq:def-of-n_k}
n_k:=\min\{n:\overline{\alpha_n(q_k)\alpha_{n+1}(q_k)\dots}\succeq \alpha(q_k)\},
\end{equation}
and let
\begin{equation} \label{eq:def-of-m_k}
m_k:=\min\{m>{n_k}:\overline{\al_{n_k}\ldots\alpha_m(q_k)}\succ\al_1(q_k)\ldots\alpha_{m-{n_k}+1}(q_k)\},
\end{equation}
or $m_k:=\f$ if no such $m$ exists.
Define a block $B_k:=1\,\overline{\alpha_1(q_k)\dots\alpha_{n_k-2}(q_k)}$, and choose $r_k>q_k$ so that
\[
(B_1\dots B_{k-1}B_k^\infty)_{r_k}=x.
\]
We can show in the same way as in Step 2 that $q_k\leq {q_s(x)}\leq r_k$, and $r_k\in\ub(x)$. Furthermore, $q_{k+1}>q_k$ and $r_{k+1}\leq r_k$ for each $k$.

\medskip

{\bf Conclusion.} Now, if at any point it so happens that
\begin{equation} \label{eq:condition-finite}
\alpha_1(q_k)\dots\alpha_{m_k}(q_k)=\alpha_1(r_k)\dots\alpha_{m_k}(r_k)\quad\textrm{with}\quad m_k<\f,
\end{equation}
then the first $m_k$ digits of $\alpha(q)$ must equal this block for all $q\in(q_k,r_k)$, so we have $B_j=B_k$ for all $j\geq k$.  This means that $q_s(x)=r_k=\min \ub(x)$, and the smallest unique expansion of $x$ is $B_1\dots B_{k-1}B_k^\infty$.

{But it is also possible that \eqref{eq:condition-finite} never happens. In that case, the process never stops, and we obtain an increasing sequence $(q_k)$ and a nonincreasing sequence $(r_k)$. Since $\Phi_x(q_k)$ and $\Phi_x(r_k)$ both begin with $B_1\dots B_{k-1}$, the distance between these sequences converges to zero, and hence $r_k-q_k\to 0$. Thus, $q_k$ and $r_k$ have the same limit, call it $q_\f=q_\infty(x)$.

{\em Claim 4:}  $q_s(x)=q_\f=\min\ub(x)$, and the smallest unique expansion of $x$ is $B_1B_2B_3\ldots$.

\begin{remark} \label{rem:three-cases}
We can easily extend the algorithm to all $x\in(0,1)$ as follows: If $q_G^{-k}\leq x<q_G^{-k+1}$, we let $q_1=q^*(x)$ where $q^*(x)$ is defined as in Proposition \ref{prop:lower-bound}. In the rest of the algorithm, we then simply precede all expansions by the fixed prefix $0^{k-1}$. The algorithm can also be extended to $x\geq 1$. However, there is no need to do so because in Theorem \ref{thm:Komornik-Loreti-cascades} we will give an explicit description of the function $q_s(x)$ on $(1,\infty)$.
\end{remark}

\begin{example} \label{ex:algorithm-example}
We illustrate the algorithm using the point $x=2/3$. Here $q_1$ is the positive root of $q(q-1)=1/x=3/2$, which gives $q_1=(1+\sqrt{7})/2\approx 1.8228757$. Then a slight modification of R\'enyi's greedy algorithm \cite{Renyi_1957} gives
\[
\al(q_1)=1101100100101100011000\dots,
\]
which implies $n_1=6$ and $m_1=12$. Thus, $B_1=10010$, and then $r_1$ is the base such that $\big((10010)^\f)_{r_1}=2/3$. This gives $r_1=1/\lambda_1$, where $\lambda_1$ is the positive root of $(\lambda+\lambda^4)/(1-\lambda^5)=2/3$. Numerically, $\lambda_1\approx .5457142$ and $r_1\approx 1.8324610$. This results in the quasi-greedy expansion
\[
\al(r_1)=1101101010001101100101000\dots.
\]
Here we see that $\al_1(q_1)\dots\al_{m_1}(q_1)=110110010010\neq 110110101000=\al_1(r_1)\dots\al_{m_1}(r_1)$, so we continue to the next step.

Setting $(B_1 01^\f)_{q_2}=(1001001^\f)_{q_2}=2/3$ gives the equation
\[
\frac{1}{q_2}+\frac{1}{q_2^4}+\frac{1}{q_2^6(q_2-1)}=\frac23,
\]
so $q_2\approx 1.83161197$. Then
\[
\al(q_2)=1101101001101001101010001001\dots,
\]
so that $n_2=22$ and $m_2=24$. Thus $B_2=100100101100101100101$, and $r_2$ is the base such that $(B_1 B_2^\f)_{r_2}=2/3$. This gives $r_2\approx 1.83161199$, and then
\[
\al(r_2)=1101101001101001101010001010\dots.
\]
At this point we see that $\al_1(q_2)\dots\al_{m_2}(q_2)=\al_1(r_2)\dots\al_{m_2}(r_2)$, so we conclude that $q_s(2/3)=r_2$, i.e. the unique base $q$ such that
\[
\big(10010(100100101100101100101)^\infty\big)_q=\frac23.
\]
Note that {\em in this example}, round-off errors in the numerical computation of $q_1,q_2,r_1$ and $r_2$ impact only the later digits of the quasi-greedy expansions $\al(q_1), \al(r_1)$ etc., and those digits play no role in the determination of $q_s(x)$ or the minimal expansion of $x$. However, round-off errors can lead to inaccurate outcomes in cases where one or more $q_i$'s are very close to discontinuity points of the map $q\mapsto\al(q)$.
\end{example}

\subsection{Proofs supporting the algorithm}
We now prove the above four claims.

\begin{proof}[Proof of Claim 1]
Note that $\al_{n_1-1}(q_1)=1$, so (\ref{eq:mar-11-1}) and (\ref{eq:mar-11-3}) give
\[
(B_101^\f)_{q_2}=x=(1\overline{\al(q_1)})_{q_1}=(B_1\,0\overline{\al_{n_1}(q_1)\al_{n_1+1}(q_1)\ldots})_{q_1}<(B_1 01^\f)_{q_1}.
\]
This implies $q_2>q_1$. Next we prove $q_2<r_1$. Since $r_1>q_1$, \eqref{eq:mar-11-1} gives
\[
(B_1^\f)_{r_1}=x=(01^\f)_{q_1}>(01^\f)_{r_1}.
\]
Therefore, by (\ref{eq:mar-11-2}) and (\ref{eq:mar-11-3}),
\[
(B_1 01^\f)_{r_1}<(B_1^\f)_{r_1}=x=(B_1 01^\f)_{q_2},
\]
so $q_2<r_1$.
For the last statement, i.e. ${q_s(x)}\ge q_2$, we observe that
both $\Phi_x(q_1)$ and $\Phi_x(r_1)$ begin with $B_1$, so by Lemma \ref{lem:char-quasi-expansion} (i), $\Phi_x(p)$ also begins with $B_1$ for any base $p\in(q_1,{r_1})$. Thus, by the same argument as in the proof of Proposition \ref{prop:lower-bound}, we conclude that ${q_s(x)}\ge q_2$.
\end{proof}

\begin{proof}[Proof of Claim 2]
The proof of Proposition \ref{prop:lower-bound} shows that
  \[
  \overline{\al_1(q_1)\ldots \al_{n_1-i}(q_1)}\lle \al_i(q_1)\ldots \al_{n_1-1}^-(q_1)\prec \al_1(q_1)\ldots \al_{n_1-i}(q_1)
  \]
  for all $1\le i<n_1$.
  Since $q_1<q_2$, by the above inequalities and the same argument as in the proof of Proposition \ref{prop:lower-bound} it follows that for each $1\le i<n_1$,
  \[
  \overline{\al_1(q_2)\ldots \al_{n_1-1}(q_2)}\prec \al_i(q_1)\ldots \al_{n_1-1}(q_1)^-\al_1(q_2)\ldots\al_{i-1}(q_2)\prec \al_1(q_2)\ldots\al_{n_1-1}(q_2),
  \]
  and for each $1\le j<n_2$,
  \[
  \overline{\al_1(q_2)\ldots \al_{n_2-1}(q_2)}\prec \al_j(q_2)\ldots\al_{n_2-1}(q_2)^-\al_1(q_2)\ldots \al_{j-1}(q_2)\prec\al_1(q_2)\ldots\al_{n_2-1}(q_2).
  \]
  This implies that
  \[
  \overline{\al(q_2)}\prec \si^n(B_1B_2^\f)\prec \al(q_2)\quad\forall n\ge {1}.
  \]
  Since $r_2>q_2$, we conclude that $B_1B_2^\f\in\us_{r_2}$, and thus $r_2\in\ub(x)$.
\end{proof}

\begin{proof}[Proof of Claim 3]
Note that $r_1, r_2\in\ub(x)$. Then showing $r_2\leq r_1$ is equivalent to showing $\Phi_x(r_2)\preceq\Phi_x(r_1)$, i.e. $B_1 B_2^\infty\preceq B_1^\infty$, or equivalently,
\begin{equation}
\big(1\,\overline{\alpha_1(q_2)\dots\alpha_{n_2-2}(q_2)}\big)^\infty =B_2^\infty\preceq B_1^\infty=\big(1\,\overline{\alpha_1(q_1)\dots\alpha_{n_1-2}(q_1)}\big)^\infty.
\label{eq:block-comparison}
\end{equation}
We consider three cases.

\bigskip
{\em Case 1.} If $n_1=n_2$, then $\alpha_1(q_2)\dots\alpha_{n_2-2}(q_2)=\alpha_1(q_2)\dots\alpha_{n_1-2}(q_2)\succeq \alpha_1(q_1)\dots\alpha_{n_1-2}(q_1)$ since $q_2>q_1$, and so \eqref{eq:block-comparison} holds, possibly with equality.

\bigskip
{\em Case 2.} If $n_2>n_1$, then since $\alpha_{n_1-1}(q_1)=1$, we have
\begin{align*}
1\,\overline{\alpha_1(q_2)\dots\alpha_{n_2-2}(q_2)}&\preceq 1\,\overline{\alpha_1(q_1)\dots\alpha_{n_2-2}(q_1)}\\
&=1\,\overline{\alpha_1(q_1)\dots\alpha_{n_1-2}(q_1)}\,0\,\overline{\alpha_{n_1}(q_1)\dots\alpha_{n_2-2}(q_1)}\\
&\prec 1\,\overline{\alpha_1(q_1)\dots\alpha_{n_1-2}(q_1)}\,1\,\overline{\alpha_1(q_1)\dots\alpha_{n_2-n_1-1}(q_1)},
\end{align*}
and \eqref{eq:block-comparison} follows with strict inequality.

\bigskip
{\em Case 3.} {If $n_2<n_1$, then by the definitions of $n_1$ and $n_2$ it follows that
\[
\overline{\alpha_{n_2}(q_2)\alpha_{n_2+1}(q_2)\dots}\succeq \alpha(q_2)\succ\alpha(q_1)\succ \overline{\alpha_{n_2}(q_1)\alpha_{n_2+1}(q_1)\dots}.
\]
On the other hand, note that $\overline{\alpha(q_2)}\prec\overline{\alpha(q_1)}$, and hence
\begin{equation}\label{eq:may-9-1}
\overline{\alpha_1(q_2)\dots\alpha_{n_2-1}(q_2)}\prec\overline{\alpha_1(q_1)\dots\alpha_{n_2-1}(q_1)}.
\end{equation}
Note that $\al_{n_2-1}(q_2)=1$. Then (\ref{eq:may-9-1}) implies
\[
(\overline{\al_1(q_2)\ldots\al_{n_2-2}(q_2)}\;1)^\f\lle (\overline{\al_1(q_1)\ldots \al_{n_2-1}(q_1)})^\f.
\]
So to prove \eqref{eq:block-comparison} we only need to prove
\[
(\overline{\al_1(q_1)\ldots\al_{n_2-1}(q_1)})^\f\lle (\overline{\al_1(q_1)\ldots \al_{n_1-1}(q_1)}^+)^\f,
\]
or equivalently,
\begin{equation}\label{eq:may-14-1}
(\al_1(q_1)\ldots\al_{n_1-1}(q_1)^-)^\f\lle(\al_1(q_1)\ldots\al_{n_2-1}(q_1))^\f.
\end{equation}
Since $n_2<n_1$, we can write $n_1-1=\ell (n_2-1)+{j}$ with $\ell\ge 1$ and $0<{j}\le n_2-1$. Then by using $\si^n(\al(q_1))\lle \al(q_1)$ for all $n\ge 0$ it follows that
\begin{align*}
(\al_1(q_1)\ldots\al_{n_1-1}(q_1)^-)^\f&\lle \big((\al_1(q_1)\ldots\al_{n_2-1}(q_1))^\ell{\al_1(q_1)\ldots \al_{j}(q_1)^-}\big)^\f\\
&\prec (\al_1(q_1)\ldots \al_{n_2-1}(q_1))^\f,
\end{align*}
proving (\ref{eq:may-14-1}).
}
\end{proof}

\begin{proof}[Proof of Claim 4]
The proof is similar to that of Claim 2. {By Lemma \ref{lem:unique-expansion}} we must show that
\begin{equation} \label{eq:B-sequence-bounds}
\overline{\al(q_\f)}\prec \si^n(B_1B_2\ldots)\prec \al(q_\f)\quad\forall ~{n\ge 1}.
\end{equation}
Observe that
\[
B_1B_2\ldots=1\overline{\al_1(q_1)\ldots \al_{n_1-1}(q_1)^-\;\al_1(q_2)\ldots \al_{n_2-1}(q_2)^-\ldots},
\]
so it follows immediately using Lemma \ref{lem:char-quasi-expansion} (iii) that {for any $n\ge 1$ we have}
\[\overline{\sigma^n(B_1 B_2\dots)}\prec \al(q_k)\quad \textrm{for some }k,\]
 and since $q_k<q_\f$, this gives the first inequality in \eqref{eq:B-sequence-bounds}.

For the second inequality, it suffices to show that for each $k\in\N$ and $1\le i<n_k$,
\begin{equation} \label{eq:qk-segment}
\overline{\al_i(q_k)\dots\al_{n_k-1}(q_k)^-\al_1(q_{k+1})\dots\al_{n_{k+1}-2}(q_{k+1})}
	\prec \al_1(q_{k+1})\dots\al_{n_k+n_{k+1}-i-2}(q_{k+1}).
\end{equation}
The argument is the same for each $k$, so we present it for $k=1$. That is, we show
\begin{equation} \label{eq:qk-segment-1}
\overline{\al_i(q_1)\dots\al_{n_1-1}(q_1)^-\al_1(q_{2})\dots\al_{n_{2}-2}(q_{2})}
	\prec \al_1(q_{2})\dots\al_{n_1+n_{2}-i-2}(q_{2})\quad {\forall\,1\le i<n_1}.
\end{equation}
As in the proof of Proposition \ref{prop:lower-bound} (cf. \eqref{eq:j-inequality}), we have
\begin{equation} \label{eq:j-inequality-version}
\overline{\al_i(q_1)\dots\al_{n_1-1}(q_1)^-\al_1(q_1)\dots\al_{i-1}(q_1)} \prec \al_1(q_1)\dots\al_{n_1-1}(q_1).
\end{equation}
Observe that $\al(q_1)\prec\al(q_2)$, and hence, $\overline{\al_1(q_2)\dots\al_{i-1}(q_2)}\preceq \overline{\al_1(q_1)\dots\al_{i-1}(q_1)}$.
So, if $i\le n_2-1$, we are done. On the other hand, if $i\geq n_2$, then it follows that $n_2<n_1$, so as in Case 3 of the proof of Claim 3, we have
\begin{equation} \label{eq:q1-and-q2-relation}
\al_1(q_1)\dots\al_{n_2-1}(q_1) \prec \al_1(q_2)\dots\al_{n_2-1}(q_2).
\end{equation}
Now notice that $i<n_1$ implies $n_1+n_2-i-2\ge n_2-1$, so \eqref{eq:q1-and-q2-relation} also gives
\[
\al_1(q_1)\dots\al_{n_1+n_2-i-2}(q_1)\prec \al_1(q_2)\dots\al_{n_1+n_2-i-2}(q_2).
\]
Thus, from \eqref{eq:j-inequality-version} we obtain
\begin{align*}
\overline{\al_i(q_1)\dots\al_{n_1-1}(q_1)^-\al_1(q_2)\dots\al_{n_2-2}(q_2)}
&\preceq \overline{\al_i(q_1)\dots\al_{n_1-1}(q_1)^-\al_1(q_1)\dots\al_{n_2-2}(q_1)}\\
&\preceq \al_1(q_1)\dots\al_{n_1+n_2-i-2}(q_1)\\
&\prec \al_1(q_2)\dots\al_{n_1+n_2-i-2}(q_2),
\end{align*}
proving \eqref{eq:qk-segment-1}.
\end{proof}

\section{Classification of points in $(0,1)$} \label{sec:classification}

From Claims 1-4 in the previous section we see that the algorithm always converges to the correct value $q_s(x)$. In fact, we have the following three possibilities:
\begin{enumerate}[I.]
\item The algorithm stops after a finite number of steps with a $q_k\in\overline{\ub}$. Then $q_s(x)=q_k\not\in\ub(x)$.
\item The algorithm never stops, but \eqref{eq:condition-finite} holds for some $k\in\N$. Then $q_s(x)=r_k=\min\ub(x)$.
\item The algorithm never stops and \eqref{eq:condition-finite} fails for every $k\in\N$. Then $q_s(x)=q_\f=\min\ub(x)$.
\end{enumerate}
We say $x\in(0,1)$ is a point of  {type} I,II or III according to which of these cases applies.
{Denote by $X_I, X_{II}$ and $X_{III}$ the sets of points of types I, II and III, respectively. Then
$(0,1)=X_I\cup X_{II}\cup X_{III}.$

\begin{proposition} \label{prop:type-I} \mbox{}

\begin{enumerate}[{\rm(i)}]
\item Let $x\in(0,1)$. If $q_s(x)\not\in\overline{\ub}$, then $q_s(x)=\min\ub(x)$.
\item $X_I$ is a Lebesgue null set.
\end{enumerate}
\end{proposition}

\begin{proof}
The first statement follows directly from the classification  of the three types. For the second statement we note from the algorithm that any point $x\in X_I$ must be of the form
$x=(d_1\ldots d_m 01^\f)_q$ {for some} $q\in\overline{\ub}.$
   So, \begin{equation}\label{eq:may-26-1}
   X_I\subseteq\bigcup_{m=1}^\f\bigcup_{d_1\ldots d_m\in\set{0,1}^m}X_{d_1\ldots d_m},
   \end{equation}
   where $X_{d_1\ldots d_m}:=\set{(d_1\ldots d_m 01^\f)_q: q\in\overline{\ub}}.$
   Observe that for each $d_1\ldots d_m\in\set{0,1}^m$ with $m\in\N$, the map $q\mapsto (d_1\ldots d_m 01^\f)_q$ is bi-Lipschitz from $\overline{\ub}$ to $X_{d_1\ldots d_m}$.
  Since $\overline{\ub}$ has zero Lebesgue measure (cf.~\cite{Erdos_Joo_Komornik_1990}), it follows that $X_{d_1\ldots d_m}$ has zero Lebesgue measure as well. Hence, by (\ref{eq:may-26-1}) we conclude that $X_I$ also has zero Lebesgue measure.
 \end{proof}

 \begin{remark}
   In Section \ref{sec:continuity} we show that $q_{\max}:=\max\set{q_s(x): x>0}$ exists, and $q_{\max}=q_s(q_G^{-1})\approx 1.88845$. Then by the same argument as in the proof of Proposition \ref{prop:type-I} we obtain
	\begin{equation*}
	\dim_H X_I=\dim_H(\overline\ub\cap(1,q_{\max}])
	=\max_{q\le q_{\max}}\dim_H\u_q<1,
	\end{equation*}
	where in the second equality we used \cite[Theorem 1.3]{Kalle-Kong-Li-Lv-2019}. Using the formula for the dimension of $\u_q$ from \cite{Komornik-Kong-Li-17} and a computer calculation, we estimate that $\dim_H X_I\approx 0.8546$.
 \end{remark}
 }

For $x\in X_{II}$ we can get $q_s(x)$ after finitely many steps although the algorithm never stops, as we saw in Example \ref{ex:algorithm-example}. In many cases, the number of steps is quite small; we present a few more examples, leaving the details to the interested reader.

\begin{example} \label{ex:1} \mbox{}

\begin{enumerate}[{\rm(i)}]
\item For $x=0.72$, we have $n_1=5$, $n_2=7$ and $n_3=17$, and $q_s(x)$ is the base $q\approx 1.7967$ such that
\[
\big(1001\,100101(1001010110100101)^\infty\big)_q=0.72.
\]
\item For $x=0.84$, we have $n_1=n_2=n_3=3$, $n_4=5$, and $q_s(x)$ is the base $q\approx 1.7557$ such that
\[
\big(10\,10\,10(1001)^\infty\big)_q=\big(1(01)^3(0011)^\infty\big)_q=0.84.
\]
\item For $x=0.39$ we have $k=2$ in Proposition \ref{prop:lower-bound}, so here the algorithm begins by letting $q_1(x)$ be the unique root in $(1,2)$ of $q^2(q-1)=(0.39)^{-1}$, and we must precede the expansions by the digit $0$. The algorithm gives $n_1=7$ and $n_2=9$, and $q_s(x)$ is the base $q\approx 1.7988$ such that
\[
\big(0100101(10010101)^\f\big)_q=0.39.
\]
\item Similarly, for $x=1/2$ we find that $q_s(x)$ is the base $q$ such that $\big((01)^\f\big)_q=1/2$, which gives $q_s(1/2)=\sqrt{3}$.
\end{enumerate}
Note in all these examples, $x$ is of type II and $q_s(x)=\min\ub(x)$ {is algebraic, because $x$ was chosen rational. If we increase $x$ by a sufficiently small amount in any of these examples, the minimal unique expansion of $x$ does not change in view of Proposition \ref{prop:finite-alg-stable} below, and $q_s(x)$ becomes algebraic over the field $\mathbb{Q}(x)$}. In Section \ref{subsec:infinite-qKL-level-set} we will construct examples of points of type III. Observe that points of type I are the easiest to come by: simply take $x=1/q(q-1)$ for any $q\in\overline{\ub}$.
\end{example}

\begin{remark}
In each of the above examples, the sequence $(n_i)$ is increasing. We suspect, but have been unable to prove, that this is always the case.
\end{remark}

Next we give a characterization of $X_{II}$, i.e., the set of points such that
 the equality \eqref{eq:condition-finite} occurs in the algorithm. Recall the definition of the set $\vb$ from \eqref{def:V}. By Lemma \ref{lem:univoque-bases} it follows that $\overline{\ub}\subseteq\vb$, and $\vb\setminus\overline{\ub}$ is countable. Since $\overline{\ub}$ is a Lebesgue null set of full Hausdorff dimension, so is $\vb$.

\begin{proposition} \label{th:qs-not-in-V}
  Let $x\in(0,1)$. Then $x\in X_{II}$ if and only if $q_s(x)\notin\vb$.
\end{proposition}

\begin{proof}
Assume first that $x\in X_{II}$, and fix $k\in\N$ such that \eqref{eq:condition-finite} holds. We will show that $[q_k, r_k]\cap \vb=\emptyset$, and hence $q_s(x)\not\in\vb$.

{Observe by Remark \ref{rem:algorithm-1} (ii) that $q_1=q^*(x)>q_G$ for each $x\in(0,1)$. 
This, along with the fact that the sequence $(q_k)$ is increasing, implies that $\al_1(q_k)=\al_2(q_k)=1$.}

	Recall the definition \eqref{eq:def-of-n_k} of $n_k$. Denote $(\al_i)=\al(q_k)$. By the same argument as in the proof of (\ref{eq:c-sandwich}) it follows that
  \begin{equation}
    \label{eq:kong-jan-1}
    \overline{\al_1\ldots \al_{n_k-1}}\prec\al_j\ldots\al_{n_k-1}^-\al_1\ldots \al_{j-1}\prec \al_1\ldots \al_{n_k-1}\qquad\forall\, 1\le j<n_k,
  \end{equation}
  where we have used $\al_1=\al_2=1$. Then (\ref{eq:kong-jan-1}) gives
  \[
  \overline{(\al_1\ldots \al_{n_k-1}^-)^\f}\lle\si^j((\al_1\ldots \al_{n_k-1}^-)^\f)\lle (\al_1\ldots\al_{n_k-1}^-)^\f\qquad\forall j\ge 0.
  \]
  So there exists $q_L\in\vb$ such that $\al(q_L)=(\al_1\ldots \al_{n_k-1}^-)^\f$. Clearly, $\al(q_L)\prec (\al_i)=\al(q_k)$. By Lemma \ref{lem:char-quasi-expansion} (i) this gives $q_L<q_k$.

  Note that $q_L\in\vb\setminus{\ub}$. Then the smallest base $q_R\in \vb$ larger than $q_L$ admits the quasi-greedy expansion $\al(q_R)=(\al_1\ldots\al_{n_k-1}\overline{\al_1\ldots \al_{n_k-1}})^\f$. So once we show that
  \begin{equation}
    \label{eq:kong-jan-2}
    \al(r_k)\prec (\al_1\ldots\al_{n_k-1}\overline{\al_1\ldots \al_{n_k-1}})^\f,
  \end{equation}
	it will follow that $[q_k, r_k]\cap \vb=\emptyset$. From the definition of $m_k$ in \eqref{eq:def-of-m_k} we see that
\[
\al_{n_k}\ldots \al_{m_k}\prec \overline{\al_1\ldots \al_{m_k-n_k+1}}.
\]
If $m_k-n_k+1\le n_k-1$, then (\ref{eq:kong-jan-2}) follows directly by   our assumption that $\al_1(r_k)\ldots \al_{m_k}(r_k)=\al_1\ldots \al_{m_k}$. If $m_k-n_k+1>n_k-1$, then write $m_k-n_k+1=\ell(n_k-1)+j$ with $\ell\ge 1$ and $1\le j\le n_k-1$.
By the minimality of $m_k$ it follows that
\[
\al_{n_k}\ldots \al_{(\ell+1)(n_k-1)}=\overline{\al_1\ldots \al_{ \ell (n_k-1)}}
\]
and
\[
\al_{(\ell+1)(n_k-1)+1}\ldots \al_{m_k}\prec \overline{\al_{ \ell (n_k-1)+1}\ldots \al_{\ell(n_k-1)+j}}.
\]
So if $\ell$ is odd, say $\ell=2i-1$, then we have
\[
\al_{1}\ldots\al_{(\ell+1)(n_k-1)}=\big(\al_1\ldots \al_{n_k-1}\overline{\al_1\ldots \al_{n_k-1}}\big)^i
\]
and
\[
\al_{(\ell+1)(n_k-1)+1}\ldots \al_{m_k}\prec \al_1\ldots \al_j;
\]
while if $\ell$ is even, say $\ell=2i$, we have
\[
\al_{1}\ldots\al_{(\ell+1)(n_k-1)}=\big(\al_1\ldots \al_{n_k-1}\overline{\al_1\ldots \al_{n_k-1}}\big)^{i}\al_1\ldots \al_{n_k-1}
\]
and
\[
\al_{(\ell+1)(n_k-1)+1}\ldots\al_{m_k}\prec \overline{\al_1\ldots \al_j}.
\]
Since $\al_1(r_k)\ldots \al_{m_k}(r_k)=\al_1\ldots \al_{m_k}$, this again proves (\ref{eq:kong-jan-2}).

For the converse, assume $q_s(x)\not\in\vb$.
Note that $x\in(0,1)$ implies $q_s(x)>q_G$, so we have  $q_s(x)\in(q_G, 2]\setminus\vb$.
  Let $(q_L, q_R)$ be the connected component of $(q_G, 2]\setminus\vb$ which contains $q_s(x)$. Then there exists an admissible word $\al_1\ldots\al_m^-$ such that
  \begin{equation}\label{eq:kong-jan-3}
  \al(q_L)=(\al_1\ldots \al_m^-)^\f\qquad\textrm{and}\qquad \al(q_R)=(\al_1\ldots \al_m\overline{\al_1\ldots \al_m})^\f.
  \end{equation}
  The algorithm produces a sequence of bases $(q_k)$ strictly increasing to $q_s(x)$ and a sequence of bases $(r_k)$ decreasing to $q_s(x)$. So there exists a smallest integer $K_1$ such that $q_s(x)\in(q_k, r_k]\subset(q_L, q_R)$ for any $k\ge K_1$. By (\ref{eq:kong-jan-3}) and Lemma \ref{lem:char-quasi-expansion} (i) this implies
  \begin{equation}
    \label{eq:kong-jan-4}
    (\al_1\ldots \al_m^-)^\f\prec \al(q_k)\prec \al(q_s(x))\lle\al(r_k)\prec (\al_1\ldots\al_m\overline{\al_1\ldots\al_m})^\f.
  \end{equation}
 Recall the definition \eqref{eq:def-of-n_k} of $n_k$. We claim that
  \begin{equation}
    \label{eq:claim-nk}
    n_k=m+1\quad\forall k\ge K_1.
  \end{equation}

  Since $q_L\in\vb$ and $\al(q_L)=(\al_1\ldots\al_m^-)^\f$, we have
  \begin{equation}
    \label{eq:kong-jan-5}
    \al_i\ldots \al_m\succ\al_i\ldots \al_m^-\lge\overline{\al_1\ldots\al_{m-i+1}}\quad\forall 1\le i\le m.
  \end{equation}
  Note by (\ref{eq:kong-jan-4}) that $\al_1(q_k)\ldots\al_m(q_k)=\al_1\ldots \al_m$. Then (\ref{eq:kong-jan-5}) yields that
  \[
  \overline{\al_{i}(q_k)\ldots\al_m(q_k)}\prec \al_1(q_k)\ldots\al_{m-i+1}(q_k)
  \]
  for all $1\le i\le m$. This implies $n_k\ge m+1$. On the other hand, using $\al_1(q_k)\ldots\al_m(q_k)=\al_1\ldots \al_m$ and (\ref{eq:kong-jan-4}) it follows that
  \[
  \al_{m+1}(q_k)\al_{m+2}(q_k)\ldots\prec (\overline{\al_1\ldots\al_m}\al_1\ldots\al_m)^\f=\overline{\al(q_R)}\prec \overline{\al(q_k)}.
  \]
  So $n_k\le m+1$. This proves (\ref{eq:claim-nk}).

  Next, we recall the definition \eqref{eq:def-of-n_k} of $m_k$. We claim there exists $M\in\N$ such that
  \begin{equation}
    \label{eq:claim-mk}
    m_k\le M\quad\forall k\ge K_1.
  \end{equation}
  Since $q_k<q_s(x)<q_R$ and $n_k=m+1$, by (\ref{eq:kong-jan-4}) there exists $M\in\N$ depending {only} on $q_s=q_s(x)$ such that
  \begin{align*}
    \al_{m+1}(q_s)\ldots \al_M(q_s)&\prec \al_{m+1}(q_R)\ldots\al_M(q_R)\\
		&=\overline{\al_1(q_R)\ldots\al_{M-m}(q_R)}\\
		&\lle\overline{\al_1(q_s)\ldots\al_{M-m}(q_s)}.
  \end{align*}
  This implies
  \begin{align*}
  \al_{m+1}(q_k)\ldots \al_M(q_k)&\lle \al_{m+1}(q_s)\ldots\al_M(q_s)\\
	&\prec\overline{\al_1(q_s)\ldots\al_{M-m}(q_s)}\\
	&\lle \overline{\al_1(q_k)\ldots \al_{M-m}(q_k)},
  \end{align*}
  so (\ref{eq:claim-mk}) holds.

  Note that $q_k$ strictly increases to $q_s(x)$ and $r_k$ decreases to $q_s(x)$ when $k\to\f$. Hence there exists $K_2>K_1$ such that
  \[
  \al_1(q_k)\ldots \al_M(q_k)=\al_1(r_k)\ldots\al_M(r_k)\quad\forall k\ge K_2.
  \]
	In view of \eqref{eq:claim-mk}, this implies \eqref{eq:condition-finite}.
  \end{proof}
	
Motivated by the work of de Vries and Komornik in \cite{DeVries-Komornik-2011} we introduce the set
\[
\mathbb U:=\{(x,q)\in (0,\infty)\times(1,2): x\in\u_q\}.
\]
We call $\mathbb U$ the {\em master univoque set}. Recall that the set-valued function $q\mapsto \us_q$ is increasing. Hence, if $(x_0,q_0)\in\mathbb U$, then the unique expansion $(d_i)$ of $x_0$ in base $q_0$ lies in $\us_q$ for every $q\in[q_0,2]$, and thus $(d_i)$ is the unique expansion in base $q$ of some number $x=x(q)\leq x_0$. The equation
\begin{equation} \label{eq:unique-expansion-of-xq}
x=x(q)=\sum_{i=1}^\infty \frac{d_i}{q^i}
\end{equation}
defines $x$ as a strictly decreasing, strictly convex function of $q$, with $x(q_0)=x_0$. Consequently, \eqref{eq:unique-expansion-of-xq} implicitly defines $q$ as a strictly decreasing, strictly convex function of $x$ corresponding to the fixed sequence $(d_i)\in\us_{q_0}$. So the graph $\{(x(q),q): q\in[q_0,2]\}$ is an arc crossing the point $(x_0, q_0)$ in the plane; we shall denote it by $C(x_0,q_0)$. The set $\mathbb U$ consists of an uncountable collection of {pairwise disjoint} such arcs, which we denote by $\mathscr{C}$. So, if $q_s(x)=\min\ub(x)$, then $C(x, q_s(x))$ is the lowest arc intersecting the fiber $\{x\}\times(1,2]$; see Figure \ref{fig:1}.

{Now we show that the map $x\mapsto \Phi_x(q_s(x))$ is locally stable in $(0,1)$ assuming $q_s(x)\notin\overline{\ub}$. In other words, the function $q_s(x)=\inf\ub(x)$ is locally convex and strictly decreasing with the assumption that $q_s(x)\notin\overline{\ub}$. Recall from Section 2 that the connected components of $(1,2]\setminus\vb$ are of the form $(\hat q_n(\sa), \hat q_{n+1}(\sa))$, where $n\in\N{\cup\set{0}}$ and $\sa$ is an admissible word.

\begin{proposition} \label{prop:finite-alg-stable}
 Let $(q_n,q_{n+1})=(\hat q_n(\sa), \hat q_{n+1}(\sa))$ be a connected component of $(1,2]\setminus\vb$.
 If $q_s(x_0)\in(q_n,q_{n+1}]$ for some $x_0\in(0,1)$, then the map
 \[
 x\mapsto \Phi_x(q_s(x))
 \]
 is constant in the interval $[x_0, x^+)$, where   $x^+:=(\Phi_{x_0}(q_s(x_0)))_{q_n}$.
\end{proposition}

Roughly speaking, Proposition \ref{prop:finite-alg-stable} says that if $q_s(x_0)\in(\hat q_n(\sa),\hat q_{n+1}(\sa)]$ for some $n\ge 0$, then the  graph of $q_s$ includes the part of the arc $C(x_0, q_s(x_0))\in\mathscr C$ for  $x\in[x_0, x^+)$.

\begin{proof}
Note by Proposition \ref{prop:type-I} (i) that $q_s(x_0)=\min\ub(x_0)\in(q_n,q_{n+1}]$. Let $(d_i)=\Phi_{x_0}(q_s(x_0))$ be the unique expansion of $x_0$ in base $q_s(x_0)$. Take $x\in[x_0, x^+)=[x_0, (d_i)_{q_n})$. Then the equation
$x=(d_i)_q$ determines a unique $q\in(q_n, q_{n+1}]$. Since $\us_q=\us_{q_s(x_0)}$, it follows that $(d_i)$ is also the unique expansion of $x$ in base $q$, i.e., $q\in\ub(x)$. To finish the proof it suffices to show that $q_s(x)\ge q$.

Suppose on the contrary that there exists $p<q$ such that $p\in\ub(x)$. Since the arcs in $\mathscr C$ are pairwise disjoint, it follows that the arc $C(x, p)$ is underneath the arc $C(x_0, q_s(x_0))$.  This means there must exist a $q'<q_s(x_0)$ such that $q'\in\ub(x_0)$, leading to a contradiction with the definition of $q_s(x_0)$. This completes the proof.
\end{proof}
}

By Propositions \ref{th:qs-not-in-V} and \ref{prop:finite-alg-stable} it follows that if $x\in X_{II}$, then $X_{II}$ contains a (right) neighborhood of $x$. Note by Example \ref{ex:1} that $X_{II}\ne\emptyset$. So, $X_{II}$ has positive Lebesgue measure.

\section{Continuity properties of $\inf\ub(x)$} \label{sec:continuity}


The purpose of this section is to prove Theorem \ref{thm:cadlag-intro}. First we establish the right continuity of $q_s$. In the algorithm from Section \ref{sec:algorithm}, we write $q_i(x)$, $r_i(x)$ and $B_i(x)$ to indicate the dependence on $x$. First, we need a technical lemma.




\begin{lemma} \label{lem:left-and-right-stability}
Let $x_0\in(0,1)$, and $k\in\N$.
\begin{enumerate}[{\rm(i)}]
\item There is a $\delta_k>0$ such that the maps $x\mapsto q_i(x)$ and $x\mapsto r_i(x)$ are right-continuous at $x_0$ and strictly decreasing on $[x_0,x_0+\delta_k]$ for $i=1,2,\dots,k$, and the blocks $B_i=B_i(x)$, $i=1,2,\dots,k$ are independent of $x$ for $x\in [x_0,x_0+\delta_k]$.
\item If $q_i\not\in\vb$ for $i=1,2,\dots,k$, then there is a $\delta_k'>0$ such that the maps $x\mapsto q_i(x)$, $i=1,2,\dots,k+1$ and $x\mapsto r_i(x)$, $i=1,2,\dots,k$ are left-continuous at $x_0$ and strictly decreasing on $[x_0-\delta_k',x_0]$, and the blocks $B_i=B_i(x)$, $i=1,2,\dots,k$ are independent of $x$ for $x\in [x_0-\delta_k',x_0]$.
\end{enumerate}
\end{lemma}

\begin{proof}
We  first prove (i)  by induction on $k$.
First take $k=1$. Recall that $q_1(x)$ is determined by the equation
$(01^\f)_{q_1(x)}=x$, so $q_1(x)$ is continuous and strictly decreasing in $x$. Furthermore, by Lemma \ref{lem:char-quasi-expansion} (i) the map $q\mapsto\alpha(q)$ is left-continuous and increasing. Hence, there exists $\de_1>0$ such that for each $x\in[x_0, x_0+\de_1]$, the quasi-greedy expansion $\alpha(q_1(x))$ agrees with $\alpha(q_1(x_0))$ in its first $n_1(x_0)$ digits, where
\[
n_1(x_0)=\min\set{n: \overline{\al_n(q_1(x_0))\al_{n+1}(q_1(x_{0}))\ldots}\lge \al(q_1(x_0))}.
\]
Note that the map $x\mapsto \al(q_1(x))$, and therefore also the map $x\mapsto \al_n(q_1(x))\al_{n+1}(q_1(x))\dots$, is decreasing in $[x_0,x_0+\de_1]$. So for any $x\in[x_0,x_0+\de_1]$ we have $n_1(x)=n_1(x_0)=:n_1$, and
\[
B_1(x)=1\overline{\al_1(q_1(x))\ldots\al_{n_1-2}(q_1(x))}=1\overline{\al_1(q_1(x_0))\ldots\al_{n_1-2}(q_1(x_0))}=B_1(x_0)=:B_1.
\]
Recall next that $r_1(x)\in(q_1(x),2]$ is determined by
$
(B_1^\f)_{r_1(x)}=x.
$
Thus, the map $x\mapsto r_1(x)$ is also strictly decreasing on $[x_0, x_0+\de_1]$ and right-continuous at $x_0$.

Proceeding by induction, let $k\in\N$ and suppose that the maps $x\mapsto q_k(x)$ and $x\mapsto r_k(x)$ are strictly decreasing on $[x_0, x_0+\de_k]$ and right-continuous at $x_0$ for some sufficiently small $\de_k>0$, and the blocks $B_1(x), B_2(x), \ldots, B_k(x)$ are independent of $x$ in the interval $[x_0, x_0+\de_k]$.  Write $B_j:=B_j(x)$ for $1\le j\le k$.  For $x\in[x_0, x_0+\de_k]$ the base $q_{k+1}(x)$ is determined by
\[
(B_1\ldots B_k 01^\f)_{q_{k+1}(x)}=x.
\]
Thus, the map $x\mapsto q_{k+1}(x)$ is strictly decreasing on $[x_0, x_0+\de_k]$ and right-continuous at $x_0$. Since the map $q\mapsto \al(q)$ is left continuous and increasing, there exists a $\de_{k+1}\in(0, \de_k)$ such that $\al(q_{k+1}(x))$ agrees with $\alpha(q_{k+1}(x_0))$ in its first $n_{k+1}(x_0)$ digits for any $x\in[x_0,x_0+\de_{k+1}]$, where
\[
n_{k+1}(x_0)=\min\set{n: \overline{\al_n(q_{k+1}(x_0))\al_{n+1}(q_{k+1}(x_{0}))\ldots}\lge \al(q_{k+1}(x_0))}.
\]
So for any $x\in[x_0,x_0+\de_{k+1}]$ we have $n_{k+1}(x)=n_{k+1}(x_0)=:n_{k+1}$ and
\begin{align*}
B_{k+1}(x)&=1\overline{\al_1(q_{k+1}(x))\ldots\al_{n_{k+1}-2}(q_{k+1}(x))}\\
&=1\overline{\al_1(q_{k+1}(x_0))\ldots\al_{n_{k+1}-2}(q_{k+1}(x_0))}=B_{k+1}(x_0)=:B_{k+1}.
\end{align*}
Since $r_{k+1}(x)\in (q_{k+1}(x), r_k(x)]$ is determined by
$$(B_1\ldots B_k B_{k+1}^\f)_{r_{k+1}(x)}=x,$$
this also implies that the map $x\mapsto r_{k+1}(x)$ is strictly decreasing on $[x_0, x_0+\de_{k+1}]$ and right-continuous at $x_0$. This proves (i).

The proof of (ii) is similar, using that the map $q\mapsto\al(q)$ is continuous (and hence right-continuous) in every connected component of $(1,2]\backslash\vb$.
\end{proof}

\begin{proposition} \label{prop:right-continuity}
The function $q_s$ is right continuous in $(0,\f)$.
\end{proposition}

\begin{proof}
Here we establish the right continuity of $q_s$ on $(0,1)$. For $x\geq 1$, the result will follow from Theorem \ref{thm:Komornik-Loreti-cascades} below, though it may also be deduced in a similar way from a straightforward extension of our algorithm to $[1,\infty)$. (We emphasize that Theorem \ref{thm:Komornik-Loreti-cascades} does not use the right continuity of $q_s$.)  We focus on the interval $[q_G^{-1}, 1)$; the result follows very similarly for other parts of $(0,1)$. Let $q_G^{-1}\le x_0<1$. We consider two cases.

\emph{ Case 1.} $q_i(x_0)\notin\overline{\ub}$ for each $i\in\N$, {i.e., $x_0$ is not a point of type I}. In this case, the algorithm never stops and produces a sequence $(q_i(x_0))$ strictly increasing to $q_s(x_0)$ and a sequence $(r_i(x_0))$ decreasing to $q_s(x_0)$ as $i\to\f$.
By Lemma \ref{lem:left-and-right-stability} (i), there exists for each $i\in\N$ a sufficiently small $\de_i>0$ such that the maps $x\mapsto q_i(x)$ and $x\mapsto r_i(x)$ are decreasing in $[x_0, x_0+\de_i]$ and right-continuous at $x_0$.

Now take $\ep>0$. Since $\lim_{i\to\f}q_i(x_0)=\lim_{i\to\f}r_i(x_0)=q_s(x_0)$, the right continuity of $q_i(x)$ and $r_i(x)$ at $x_0$ implies that we can find $N\in\N$ sufficiently large and $\delta>0$ sufficiently small so that
\begin{equation}\label{eq:mar-12-11}
|q_N(x)-q_s(x_0)|<\ep \qquad\mbox{and} \qquad |r_N(x)-q_s(x_0)|<\ep
\end{equation}
for all $x\in[x_0, x_0+\de]$. Note that $q_N(x)<q_s(x)\le r_N(x)$ for any $x\in(0,1)$. By (\ref{eq:mar-12-11}) we conclude that
\[
q_s(x_0)-\ep\le q_N(x)\le q_s(x)\le r_N(x)\le q_s(x_0)+\ep
\]
for all $x\in[x_0, x_0+\de]$. This proves the right continuity of $q_s$ at $x_0$.

\emph{Case 2.} $q_i(x_0)\in\overline{\ub}$ for some $i\in\N$, i.e., $x_0$ is of type I. Without loss of generality, assume this holds for $i=1$. (Otherwise, we can simply replace $q_1(x)$ by $q_i(x)$ in what follows.) By Proposition \ref{prop:lower-bound}, $q_s(x_0)=q_1(x_0)$. Write $(\al_i):=\al(q_1(x_0))$. Since $q_1(x_0)\in\overline{\ub}$, there exist by \cite[Lemma 4.1]{Komornik_Loreti_2007} arbitrarily large integers $m$ such that
\begin{equation} \label{eq:alpha-bar-inequality}
\overline{\alpha_{j+1}\dots \alpha_m}\prec \alpha_1\dots \alpha_{m-j} \qquad \forall ~ 1\leq j<m.
\end{equation}
Fix $\ep>0$, and choose $m\in\N$ so large that \eqref{eq:alpha-bar-inequality} holds and
\begin{equation}
\frac{1}{q_G^{m-2}(q_G-1)}<\frac{\ep}{2}.
\label{eq:small-fraction-m}
\end{equation}
Since $q_1(x)$ is continuous and decreasing in $x$ and $\alpha(q)$ is left-continuous in $q$, we can choose $\de_m>0$ small enough such that for any $x\in[x_0, x_0+\de_m]$,
 \begin{equation}\label{eq:mar-13-1}
\alpha_1(q_1(x))\dots\alpha_m(q_1(x))=\alpha_1\dots\alpha_m.
\end{equation}
Now fix $x\in[x_0, x_0+\de_m]$, and set $q_1:=q_1(x)$, $r_1:=r_1(x)$ and $n_1:=n_1(x)$.
If $q_1\in\overline{\ub}$, then $q_s(x)=q_1$. Otherwise, \eqref{eq:alpha-bar-inequality} and (\ref{eq:mar-13-1}) imply
\begin{equation}\label{eq:mar-13-2}
n_1=\min\set{n\geq 1: \overline{\alpha_n(q_1)\alpha_{n+1}(q_1)\dots}\succeq \al(q_1)}>m.
\end{equation}
Note that $r_1>q_1$, and $q_1$ and $r_1$ satisfy the equations
\[
\big(1\,\overline{\alpha_1(q_1)\alpha_2(q_1)\dots}\big)_{q_1}=x=\big((1\,\overline{\alpha_1(q_1)\dots\alpha_{n_1-2}(q_1)})^\infty\big)_{r_1}.
\]
Then by (\ref{eq:mar-13-1}) and (\ref{eq:mar-13-2}) it follows that
\[
(1\,\overline{\al_1\ldots \al_{m-1}})_{q_1}\le x\le (1\,\overline{\al_1\ldots \al_{m-1}}\,1^\f)_{r_1}=(1\,\overline{\al_1\ldots \al_{m-1}})_{r_1}+r_1^{-m}(1^\f)_{r_1},
\]
which implies
\begin{align*}
\frac{1}{r_1^{m}(r_1-1)}=r_1^{-m}(1^\f)_{r_1} &\ge (1\,\overline{\al_1\ldots \al_{m-1}})_{q_1}-(1\,\overline{\al_1\ldots \al_{m-1}})_{r_1}\\
&\ge \frac{1}{q_1}-\frac{1}{r_1}\ge \frac{r_1-q_1}{r_1^2}.
\end{align*}
Hence
\begin{equation*}
0<r_1-q_1\le \frac{1}{r_1^{m-2}(r_1-1)}\le \frac{1}{q_G^{m-2}(q_G-1)}<\frac{\ep}{2},
\end{equation*}
where the third inequality follows since $r_1>q_1\ge q_G$, and the last inequality follows by \eqref{eq:small-fraction-m}.

To summarize, we have for each $x\in[x_0, x_0+\de_m]$ that either $q_s(x)=q_1(x)$ or $0<r_1(x)-q_1(x)<\ep/2$. Since $q_1(x)$ is continuous, there exists $\de\in(0, \de_m)$ such that for any $x\in[x_0, x_0+\de]$,
\[|q_1(x)-q_s(x_0)|=|q_1(x)-q_1(x_0)|<\ep/2.\]
 Therefore, for any $x\in[x_0, x_0+\de]$ we obtain either
 \[|q_s(x)-q_s(x_0)|=|q_1(x)-q_1(x_0)|<\ep/2,\]
  or
\[
q_s(x_0)-\ep/2\le q_1(x)\le  q_s(x)\le r_1(x)\le q_1(x)+\ep/2\le q_s(x_0)+\ep.
\]
This gives the right-continuity of $q_s$ at $x_0$.
%
%
\end{proof}

Recall that $\mathscr{C}$ is the collection of pairwise disjoint arcs from the master univoque set $\mathbb U$. For an arc $C=C(x_0,q_0)\in\mathscr{C}$ with the corresponding unique expansion $(d_i)=\Phi_{x_0}(q_0)\in\us_{q_0}$, we denote by $\big(\frac{dx}{dq}\big)_C$ the derivative of $x$ with respect to $q$ in \eqref{eq:unique-expansion-of-xq}; that is,
\begin{equation} \label{eq:derivative-of-x}
\left(\frac{dx}{dq}\right)_C:=\frac{d}{dq}\left(\sum_{i=1}^\infty \frac{d_i}{q^i}\right)=-\sum_{i=1}^\infty \frac{i d_i}{q^{i+1}}.
\end{equation}
We denote the derivative of the inverse function by $\big(\frac{dq}{dx}\big)_C$.

\begin{lemma} \label{lem:derivative-bound}
Let $C=C(x,q)$ be an arc in $\mathscr{C}$. If $q_G^{-k}<x<q_G^{-k+1}$ for $k\in\N$, then
\[
\left|\left(\frac{dx}{dq}\right)_C\right|\geq \frac{1}{2^{k+1}}, \qquad\mbox{and so} \qquad \left|\left(\frac{dq}{dx}\right)_C\right|\leq 2^{k+1}.
\]
And if $x>1$, then
\[
\left|\left(\frac{dx}{dq}\right)_C\right|\geq \frac1 {4}, \qquad\mbox{and so} \qquad \left|\left(\frac{dq}{dx}\right)_C\right|\leq {4}.
\]
\end{lemma}

\begin{proof}
The first statement follows from \eqref{eq:derivative-of-x} since any unique expansion of $x\in(q_G^{-k},q_G^{-k+1})$ must begin with $0^{k-1}1$, and $q\leq 2$. The second statement follows similarly, since a unique expansion of $x>1$ must begin with $1$.
\end{proof}

\begin{lemma} \label{lem:left-hand-limits}
The function $q_s$ has at each point $x>0$ a left-hand limit.
\end{lemma}

\begin{proof}
In view of Theorem \ref{thm:Komornik-Loreti-cascades} it suffices to consider $x\in(0,1)$.

Fix $x_0\in(0,1)$, and suppose $q_s$ does {\em not} have a left-hand limit at $x_0$. Then there are numbers $q^+$ and $q^-$ and two sequences $(x_n)$ and $(y_m)$ converging to $x_0$ from the left such that
\begin{equation}\label{eq:mar-12-0}
\lim_{n\to\infty} q_s(x_n)=q^+>q^-=\lim_{m\to\infty} q_s(y_m).
\end{equation}
Write $\Delta :=q^+-q^-$. Furthermore, since $x_0\in(0,1)$, let $k\in\N$ be such that $q_G^{-k}<x_0\leq q_G^{-k+1}$. Choose $\delta>0$ so small that $x_0-\delta>q_G^{-k}$. By Lemma \ref{lem:derivative-bound}, if $(x,q)$ lies on an arc $C\in\mathscr{C}$ and $x_0-\delta<x<x_0$, then $\big|\big(\frac{dq}{dx}\big)_C\big|\leq K:=2^{k+1}$. By (\ref{eq:mar-12-0}) we can find an integer $N$ so that
\begin{equation}\label{eq:mar-12-1}
x_0>x_N>\max\set{x_0-\de, x_0-\frac{\Delta}{2K}}\quad\textrm{and}\quad q_s(x_N)>q^+-\frac{\Delta}{4}.
\end{equation}
Also, by (\ref{eq:mar-12-0}) we can find a large enough integer $M$ so that
\begin{equation}\label{eq:mar-12-2}
x_0>y_M>x_N\quad\textrm{and}\quad q_s(y_M)<q^-+\frac{\Delta}{4}.
\end{equation}
%
If $q_s(y_M)\in \ub(y_M)$, set $\tilde q:=q_s(y_M)$. Otherwise, we can find $\ep>0$ such that $q_s(y_M)+\ep<q^-+\Delta /4$ and $q_s(y_M)+\ep\in\ub(y_M)$; we then set $\tilde q:=q_s(y_M)+\ep$. In either case we have 
\begin{equation}\label{eq:mar-12-3}
\tilde q\in\ub(y_M)\quad\textrm{and}\quad \tilde q<q^-+\frac{\Delta}{4}.
\end{equation}
The point $(y_M, \tilde q)$ lies on an arc $C\in\mathscr{C}$. By (\ref{eq:mar-12-1}) and (\ref{eq:mar-12-3}),
\[
q_s(x_N)-\tilde q>q^+-\frac{\Delta}{4}-\left(q^-+\frac{\Delta}{4}\right)=\frac{\Delta}{2}.
\]
Thus,
\[
{\left|\frac{\tilde q-q_s(x_N)}{y_M-x_N}\right|=}\frac{q_s(x_N)-\tilde q}{y_M-x_N}>\frac{q_s(x_N)-\tilde q}{x_0-x_N}>\frac{\Delta /2}{x_0-x_N}>K,
\]
where the last inequality follows by (\ref{eq:mar-12-1}).
Observe that the slope of the arc $C=C(y_M, \tilde q)$ never exceeds $K$ in absolute value. We see that $C$ passes underneath the point $(x_N,q_s(x_N))$. But this means that there is a base $p<q_s(x_N)$ such that $p\in\ub(x_N)$, contradicting the definition of $q_s(x_N)$. As a result, $q_s$ has a left-hand limit at $x_0$.
\end{proof}

\begin{lemma} \label{lem:no-down-jumps}
The function $q_s$ has no downward jumps; that is,
$$\lim_{x\to x_0^-}q_s(x)\leq \lim_{x\to x_0^+}q_s(x) \qquad\forall\,x_0>0.$$
\end{lemma}

\begin{proof}
Suppose by way of contradiction that
\begin{equation}\label{eq:may-13-1}
q_+:=\lim_{x\to x_0^-}q_s(x)>q_-:=\lim_{x\to x_0^+}q_s(x)
\end{equation}
for some $x_0$. By right-continuity, $q_s(x_0)=\inf\ub(x_0)=q_-$. If $q_-\in\ub(x_0)$, then let $C\in\mathscr{C}$ be the arc containing the point $(x_0,q_-)$. Otherwise, there exists $p\in(q_-, q_+)$ sufficiently close to $q_-$ such that $p\in\ub(x_0)$, and in this case let $C\in\mathscr{C}$ be the arc containing $(x_0, p)$. Thus,  by (\ref{eq:may-13-1}) it follows that  for all $x<x_0$ sufficiently close to $x_0$, the point $(x, q_s(x))$ is above the arc $C$, contradicting the definition of $q_s(x)$. 
\end{proof}

\begin{proof}[Proof of Theorem \ref{thm:cadlag-intro}]
The theorem follows from Proposition \ref{prop:right-continuity}, Lemma \ref{lem:left-hand-limits} and Lemma \ref{lem:no-down-jumps}.
\end{proof}

\begin{remark}
We observe that by Theorem \ref{thm:cadlag-intro}, the function $q_s$ has the same continuity properties as the sample paths of a spectrally positive L\'evy process.
\end{remark}

As a direct consequence of Theorem \ref{thm:cadlag-intro} we have the following:

\begin{corollary} \label{cor:smallest-and-largest}\mbox{}

\begin{enumerate}[{\rm(i)}]
\item $q_s$ attains a maximum value $q_{\max}:=\max q_s$. 
\item The function $q_s$ has the intermediate value property; that is, for any $q\in(1,q_{\max}]$, there exists $x>0$ such that $q_s(x)=q$.
\item Each level set $L(q)$ with $1<q\le q_{\max}$
has both a smallest and a largest element.
\end{enumerate}
\end{corollary}

\begin{proof}
All three statements follow easily from Theorem \ref{thm:cadlag-intro} and the fact that $\lim_{x\to\infty}q_s(x)=1=\inf_x q_s(x)$.
\end{proof}

\begin{remark} \label{rem:qmax}
In Section \ref{sec:maximum} we will show that the maximum value $q_{\max}$ is uniquely attained at $x=1/q_G$ and we will give an exact calculation of $q_{\max}$, showing $q_{\max}=q_s(1/q_G)\approx 1.88845$.
\end{remark}


To end this section, we show that the function $q_s$ can jump only when it reaches a level $q\in\vb\backslash\overline{\ub}$.

\begin{proposition} \label{prop:jumps}
Let $q_k=q_k(x)$ be the $k^{\mathrm{th}}$ lower bound for $q_s(x)$ in the algorithm. Then $q_s$ has a jump at $x$ if and only if $q_k(x)\in\vb\backslash\overline{\ub}$ for some $k$. Moreover, if this is the case and $k_0$ is the {\em smallest} such $k$, then
\[
\lim_{y\nearrow x}q_s(y)=q_{k_0}(x)<q_s(x).
\]
\end{proposition}

\begin{proof}
Assume first that $q_k(x)\in\vb\backslash\overline{\ub}$ for some $k$, and let $k_0$ be the smallest such $k$. Since $q_{k_0}\not\in\overline{\ub}$, and the sequence $(q_k)$ is strictly increasing to $q_s(x)$, we have $q_s(x)>q_{k_0}(x)$. Hence, it suffices to show that $\lim_{y\nearrow x}q_s(y)=q_{k_0}(x)$.

Since $q_j(x)\not\in\vb$ for all $j<k_0$, there is by Lemma \ref{lem:left-and-right-stability} (ii) a $\delta>0$ such that $q_1,\dots,q_{k_0}$ are decreasing on $[x-\delta,x]$ and left-continuous at $x$, and the blocks $B_1,\dots,B_{k_0-1}$ are constant on the interval $(x-\delta,x]$. Choosing $\delta$ even smaller if necessary, we may assume that $q_1(y),\dots,q_{k_0}(y)\not\in\overline{\ub}$ for all $y\in(x-\delta,x)$. Hence

\[
\liminf_{y\nearrow x} q_s(y)\geq \lim_{y\nearrow x}q_{k_0}(y)=q_{k_0}(x).
\]
On the other hand, since $q_{k_0}:=q_{k_0}(x)\in\vb\backslash\overline{\ub}$, we can write $\alpha(q_{k_0})=(a_1\dots a_m\overline{a_1\dots a_m})^\f$ where $a_1\dots a_m^-$ is some admissible word. Consider the expansion
\[
\sc:=B_1(x)\dots B_{k_0-1}(x)1\,\overline{\alpha(q_{k_0}(x))}
=B_1(x)\dots B_{k_0-1}(x)1(\overline{a_1\dots a_m}a_1\dots a_m)^\f.
\]
Note that $(\sc)_{q_{k_0}}=x$.  Hence, for $y\in(x-\delta,x)$, $\sc$ is an expansion of $y$ in some base $q(y)>q_{k_0}$, and $q(y)\searrow q_{k_0}$ as $y\nearrow x$. Finally, by the nature of $\alpha(q_{k_0})$, the expansion $\sc$ of $y$ is unique. Therefore,
\[
\limsup_{y\nearrow x} q_s(y)\leq \limsup_{y\nearrow x}q(y)=q_{k_0}(x).
\]
This completes the first half of the proof.

For the converse, assume that $q_k(x)\not\in\vb\backslash\overline{\ub}$ for every $k$. Since $q_s$ has no downward jumps, it suffices to show that
\begin{equation} \label{eq:left-lower-semi-continuous}
\lim_{y\nearrow x} q_s(y)\geq q_s(x).
\end{equation}
We consider two cases.

\medskip
{\em Case 1.} $q_k(x)\in\overline{\ub}$ for some $k$ (in other words, $x$ is of type I). Let $k_0$ be the smallest such $k$. Since $q_1(x),\dots,q_{k_0-1}(x)\not\in\vb$, we have by Lemma \ref{lem:left-and-right-stability} (ii) that the function $y\mapsto q_{k_0}(y)$ is strictly decreasing on $(x-\delta,x]$. Hence $q_s(y)\geq q_{k_0}(y)>q_{k_0}(x)=q_s(x)$ for all $y\in(x-\delta,x)$, which gives \eqref{eq:left-lower-semi-continuous}.

\medskip
{\em Case 2.} $q_k(x)\not\in\overline{\ub}$ for every $k$. Let $\ep>0$. Since $q_k(x)\nearrow q_s(x)$, we can choose $k$ so large that $q_k(x)>q_s(x)-\ep$. Again by Lemma \ref{lem:left-and-right-stability} (ii), the map $y\mapsto q_k(y)$ is decreasing in $(x-\delta,x]$ for some sufficiently small $\delta>0$. So $q_s(y)\geq q_k(y)>q_k(x)>q_s(x)-\ep$ for all $y\in(x-\delta,x)$, and hence $\lim_{y\nearrow x} q_s(y)\geq q_s(x)-\ep$. Letting $\ep\to 0$ we again obtain \eqref{eq:left-lower-semi-continuous},  completing the proof.
\end{proof}

\begin{corollary} \label{cor:decreasing-here}
If $x\in X_I\cup X_{II}$, then there exists $\delta>0$ such that $q_s(y)\neq q_s(x)$ for all $y\in(x-\delta,x)$.
\end{corollary}

\begin{proof}
Let $x\in X_I\cup X_{II}$. If $q_k(x)\in\vb\backslash\overline{\ub}$ for some $k$, then $q_s$ jumps at $x$ by Proposition \ref{prop:jumps}, and since $q_s$ has a left-hand limit at $x$, the result follows.

On the other hand, suppose $q_k(x)\not\in\vb\backslash\overline{\ub}$ for every $k$. Then we are in Case 1 or Case 2 of the above proof. If we are in Case 1, then the above proof shows that there is $\delta>0$ such that $q_s(y)>q_s(x)$ for all $y\in(x-\delta,x)$. If we are in Case 2, then $x\in X_{II}$ and there is an integer $k$ such that \eqref{eq:condition-finite} holds. By Lemma \ref{lem:left-and-right-stability} (ii), there is $\delta>0$ such that the blocks $B_1=B_1(y),\dots,B_k=B_{k}(y)$ are independent of $y$ for $y\in(x-\delta,x]$. This means that for all $y\in(x-\delta,x]$, $q_s(y)=r_k(y)$ is the unique base $r$ such that $\big(B_1\dots B_{k-1}B_k^\f\big)_r=y$, so again $q_s(y)>q_s(x)$ for all $y\in(x-\delta,x)$.
\end{proof}

From Corollary \ref{cor:decreasing-here}, together with Proposition \ref{prop:accumulation-points} below, we obtain some interesting information about the behavior of the level sets around points of each of the three types:
\begin{itemize}
\item No point of type I is a left accumulation point of its level set. However, a point of type I can be a right accumulation point of its level set; e.g. $x=1/q_{KL}(q_{KL}-1)$ (cf.~Example \ref{ex:accumulation-points}).
\item Every point of type II is an isolated point of its level set. This follows from Corollary \ref{cor:decreasing-here} and Propositions \ref{th:qs-not-in-V} and \ref{prop:finite-alg-stable}.
\item Some (in fact infinitely many) points of type III are left accumulation points of their level set.
\end{itemize}
We do not know if a point of type III can be a right accumulation point of its level set.

\section{Level sets of $\inf\ub(x)$} \label{sec:level-sets}

In this section, we investigate the level sets $L(q)$ of $q_s$ and prove Theorems \ref{thm:finite-level-sets-intro}, \ref{thm:infinite-level-set-qKL-intro} and \ref{thm:other-infinite-level-sets-intro}. 



\begin{proof}[Proof of Theorem \ref{thm:finite-level-sets-intro}]
Fix $q\in(1,q_{\max}]\setminus\overline{\ub}$. Then there is a connected component $(q_n,q_{n+1})$ of $(1,2]\setminus\vb$ such that $q_n<q\le q_{n+1}$.
By Corollary \ref{cor:smallest-and-largest}, $L(q)$ has a smallest element $x_0$, and $x_0>0$ since $q_s(0)=1<q$.

Suppose $x\in L(q)$. By Proposition  \ref{prop:type-I} (i), $q_s(x)=q=\min\ub(x)$. Hence the point $(x,q)$ lies on an arc $C=C(x,q)\in\mathscr{C}$. Let $(d_i)=\Phi_x(q)$ be the unique expansion of $x$ in base $q$. Since $(d_i)\in\us_p$ for all $p\in(q_n,q_{n+1}]$, the arc $C$ extends down and to the right to a point $(x',q_n)$ with $x'>x$.
Furthermore, by Proposition \ref{prop:finite-alg-stable} the entire part of the arc $C$ between $(x,q)$ and $(x',q_n)$, with the exception of the endpoint $(x',q_n)$, is contained in the graph of $q_s$.

Since $x_0>0$, we can choose $k\in\N$ so large that $x_0>q_G^{-k}$. By Lemma \ref{lem:derivative-bound}, the slope of $C$ is bounded in absolute value by a constant $M:=2^{k+1}$. It follows that
\[
x'-x\ge\frac{q-q_n}{M}>0.
\]
Observe that for each $y\in(x,x')$ we have $q_s(y)<q_s(x)=q$. Therefore any two points of $L(q)$ must lie at least a distance $(q-q_n)/M$ apart. Moreover, Corollary \ref{cor:smallest-and-largest} implies $L(q)$ is bounded. Hence $L(q)$ is finite, completing the proof.
\end{proof}

\subsection{The infinite level set $L(q_{KL})$} \label{subsec:infinite-qKL-level-set}
Below, we will prove something considerably more detailed than Theorem \ref{thm:infinite-level-set-qKL-intro}.
First we define a subset $S\subset\N$ recursively as follows:
\begin{enumerate}
\item $2,3,4\in S$;
\item For each $k\in\N_{\geq 2}$ and $2^k<n\leq 2^{k+1}$,  $n\in S$ if and only if
\[
{n>3\cdot 2^{k-1} \qquad\textrm{and} \qquad n-3\cdot 2^{k-1}\in S.}
\]
\end{enumerate}
Thus, $S=\{2,3,4,8,14,15,16,26,27,28,32,50,51,52,56,62,63,64,\dots\}$.

\begin{lemma}\label{lem:S}\mbox{}

\begin{enumerate}[{\rm(i)}]
\item $2^k\in S$ for any $k\in\N$.
\item $2^{2k}-2, 2^{2k}-1\in S$ for any $k\in\N$.
\item The set $S$ has density zero in $\N$.
\end{enumerate}
\end{lemma}

\begin{proof}
Statement (i) follows by induction since $2, 2^2\in S$ and $2^{k+1}=3\cdot 2^{k-1}+2^{k-1}$ for $k\geq 2$. Similarly, (ii) follows inductively since $2^2-2, 2^2-1\in S$ and
\[
2^{2(k+1)}-j=3\cdot 2^{2k}+(2^{2k}-j)\quad\textrm{for }j=1,2.
\]

For (iii), let $N_k:=\#\{n\leq 2^k: n\in S\}$ for $k\in\N$. One checks easily that $N_k=N_{k-1}+N_{k-2}$ for $k\geq 3$, with $N_1=1$ and $N_2=3$. Thus, $N_k$ is the $k$th {\em Lucas number}, given explicitly by
\[
N_k=\left(\frac{1+\sqrt{5}}{2}\right)^k+\left(\frac{1-\sqrt{5}}{2}\right)^k, \qquad k\in\N.
\]
As a result,
\[
\limsup_{n\to\infty}\frac{\#\{m\leq n: m\in S\}}{n}=\lim_{k\to\infty}\frac{N_k}{2^k}=0,
\]
completing the proof.
\end{proof}

{%

For each $n\in\N$, we define the points
\begin{gather*}
x_n:=(\tau_n\tau_{n+1}\dots)_{q_{KL}},\quad
x_n':=(0\tau_n\tau_{n+1}\dots)_{q_{KL}},\quad
x_n'':=(00\tau_n\tau_{n+1}\dots)_{q_{KL}}.
\end{gather*}

\begin{theorem} \label{thm:set-with-qKL-as-min}
The level set $L(q_{KL})$ is infinite. Furthermore,
\begin{enumerate}[{\rm(i)}]
\item $q_s(x_n)=q_{KL}$ if and only if $n\in S\cup\{1\}$. 
\item Assume $\tau_n=1$. Then $q_s(x_n')=q_{KL}$ if and only if $n\in S$.
\item $q_s(x_n'')=q_{KL}$ if and only if $n\in S$ and $\tau_n\dots\tau_{n+6}=1010011$.
\item In each case above ($x=x_n, x_n'$ or $x_n''$), if $q_s(x)=q_{KL}$, then $q_{KL}=\min\ub(x)$.
\end{enumerate}
\end{theorem}

 \begin{remark}\mbox{}
  \begin{enumerate}[{\rm(i)}]
  \item {The points $x_n,x_n'$ and $x_n''$ which belong to the level set $L(q_{KL})$ are all of type III.} 
   \item Observe that $x_n'=x_{n-1}$ if $\tau_{n-1}=0$, and similarly, $x_n''=x_{n-1}'$ if $\tau_{n-1}=0$. {Note that if $\tau_n=1$, then $\tau_{n}\ldots \tau_{n+5}\lge 100101$.}
It follows that the points $x_n$ with $\tau_n=1$ lie in the interval $[1/q_G,1)$; the points $x_n'$ with $\tau_n=1$ lie in $[1/q_G^2,1/q_G)$; and the points $x_n''$ with $\tau_n\dots \tau_{n+6}=1010011$ lie in $[1/q_G^3,1/q_G^2)$. Thus, Theorem \ref{thm:set-with-qKL-as-min} implies that the intersection of the level set $L(q_{KL})$ with each of these intervals is infinite.
	\item On the other hand, $L(q_{KL})$ does not intersect $(0,1/q_G^3)$: For $x$ in this interval, we can find an integer $k\geq 3$ so that $1/q_G^{k+1}\leq x<1/q_G^k$, and then $0^k(10)^\f$ is a unique expansion of $x$ in some base $q\in(q_G,q_{KL})$, since $\big(0^k(10)^\f\big)_{q_G}=q_G^{-k}>x$, and
	\[
	\big(0^k(10)^\f\big)_{q_{KL}}=\frac{1}{q_{KL}^{k-1}(q_{KL}^2-1)}<\frac{1}{q_G^{k+1}}.
	\]
As a consequence, if we set $x=0^k\tau_n\tau_{n+1}\dots$ with $k\geq 3$ and any $n\in\N$, then $x\leq q_{KL}^{-k}<q_G^{-3}$ and so $q_s(x)<q_{KL}$.
\item By Lemma \ref{lem:S} (iii) and Theorem \ref{thm:set-with-qKL-as-min} (i) it follows that the implication
\[
x=(\tau_n\tau_{n+1}\ldots)_{q_{KL}}\quad\Longrightarrow\quad \min\ub(x)=q_{KL}
\] fails for almost all $n\in\N$ in the sense of density.
\end{enumerate}
\end{remark}

Before proving the theorem, we first develop a series of useful estimates, which we will need also in Section \ref{sec:cascades}. Recall from (\ref{eq:qn}) the definition of the sequence $(\hat q_m)$, and recall that $\hat q_m\nearrow q_{KL}$ as $m\to\f$.

\begin{lemma} \label{lem:calculation}
For every $m\in\N$, we have the identities
\begin{gather}
(\tau_1\dots\tau_{2^m})_{\hat{q}_{m+1}}=1-\hat{q}_{m+1}^{-2^m} \frac{2-\hat{q}_{m+1}}{\hat{q}_{m+1}-1}, \label{eq:identity2} \\
(\overline{\tau_1\dots\tau_{2^m}}^{\,+})_{\hat{q}_{m+1}}=\frac{2-\hat{q}_{m+1}}{\hat{q}_{m+1}-1}, \label{eq:identity1} \\
(\overline{\tau_1\dots\tau_{2^{m-1}}}^{\,+})_{\hat{q}_{m+1}}=\frac{2-\hat{q}_{m+1}}{\hat{q}_{m+1}-1}\left(\frac{1}{1-\hat{q}_{m+1}^{-2^{m-1}}}-\hat{q}_{m+1}^{-2^{m-1}}\right). \label{eq:identity3}
\end{gather}
\end{lemma}

\begin{proof}
Let $x=(\tau_1\dots\tau_{2^m})_{\hat{q}_{m+1}}$ and $y=(\overline{\tau_1\dots\tau_{2^m}}^{\,+})_{\hat{q}_{m+1}}$. On the one hand,
\[
x+y-\hat{q}_{m+1}^{-2^m}=\left(1^{2^m}\right)_{\hat{q}_{m+1}}=\sum_{i=1}^{2^m}\frac{1}{\hat{q}_{m+1}^i}=\frac{1-\hat{q}_{m+1}^{-2^m}}{\hat{q}_{m+1}-1},
\]
while on the other hand, since $\tau_{2^m+1}\dots\tau_{2^{m+1}}=\overline{\tau_1\dots\tau_{2^m}}^{\,+}$,
\[
1=(\tau_1\dots\tau_{2^{m+1}})_{\hat{q}_{m+1}}=x+\hat{q}_{m+1}^{-2^m}y.
\]
Solving these two linear equations in $x$ and $y$ gives \eqref{eq:identity1} and \eqref{eq:identity2}.
To derive \eqref{eq:identity3}, set $z=(\overline{\tau_1\dots\tau_{2^{m-1}}})_{\hat{q}_{m+1}}$. By \eqref{eq:identity1},
\[
\frac{2-\hat{q}_{m+1}}{\hat{q}_{m+1}-1}=(\overline{\tau_1\dots\tau_{2^m}}^{\,+})_{\hat{q}_{m+1}}=z+\hat{q}_{m+1}^{-2^{m-1}}\big(\overline{\tau_{2^{m-1}+1}\dots\tau_{2^m}}^{\,+}\big)_{\hat{q}_{m+1}}.
\]
Here
\[
\big(\overline{\tau_{2^{m-1}+1}\dots\tau_{2^m}}^{\,+}\big)_{\hat{q}_{m+1}}=\big(\tau_1\dots\tau_{2^{m-1}})_{\hat{q}_{m+1}}
=\left(1^{2^{m-1}}\right)_{\hat{q}_{m+1}}-z=\frac{1-\hat{q}_{m+1}^{-2^{m-1}}}{\hat{q}_{m+1}-1}-z.
\]
Combining these equations gives
\[
z=\frac{2-\hat{q}_{m+1}}{(\hat{q}_{m+1}-1)\big(1-\hat{q}_{m+1}^{-2^{m-1}}\big)}-\frac{\hat{q}_{m+1}^{-2^{m-1}}}{\hat{q}_{m+1}-1}.
\]
Adding $\hat{q}_{m+1}^{-2^{m-1}}$ and rearranging yields \eqref{eq:identity3}.
\end{proof}

In the next lemma, the value of the constant is critical.

\begin{lemma} \label{lem:qm-difference}
For all $m\geq 1$, we have
\[
\hat{q}_{m+1}-\hat{q}_m\geq (.265)\hat{q}_{m+1}^{-2^m}{>\hat q_{m+1}^{-(2^m+3)}}.
\]
\end{lemma}

\begin{proof}
The second inequality follows {since $\hat{q}_m^{-3}\leq\hat q_1^{-3}\approx .2361<.265$}. So we only prove the first inequality.
Set $\lambda_m:=1/\hat{q}_m$, and note by (\ref{eq:qn}) that $\lambda_m$ is the unique positive zero of the polynomial
\[
g_m(\lambda):=\sum_{i=1}^{2^m}\tau_i\lambda^i-1.
\]
On the other hand, Lemma \ref{lem:calculation} gives
\begin{equation}
g_m(\lambda_{m+1})=(\tau_1\dots\tau_{2^m})_{\hat{q}_{m+1}}-1
=-\hat{q}_{m+1}^{-2^m} \frac{2-\hat{q}_{m+1}}{\hat{q}_{m+1}-1}
=-\lambda_{m+1}^{2^m} \frac{2\lambda_{m+1}-1}{1-\lambda_{m+1}}.
\label{eq:fm-at-lambda-m-plus-1}
\end{equation}
Now we estimate $g_m'(\lambda)$ for $\lambda\in[\lambda_{m+1},\lambda_m]$. We claim that
\begin{equation}
g_m'(\lambda)\leq 3.24 \qquad \forall~m\geq 3 \quad\textrm{and}\quad \la\in[\la_{m+1},\lambda_m].
\label{eq:f-prime-bound}
\end{equation}
First, observe that
\begin{equation*}
\sum_{{i=2^m+1}}^\f i\tau_i\lambda^{i-1}<\sum_{{i=2^m+1}}^\f i\lambda^{i-1}
=\frac{\lambda^{2^m}}{(1-\lambda)^2}+\frac{{2^m}\lambda^{2^m}}{1-\lambda}<\frac{2^m\lambda^{2^m}}{(1-\lambda)^2}.
\end{equation*}
Thus, for $m\geq 3$ and $\lambda\in[\lambda_{m+1},\la_m]$,
\begin{equation*}
g_m'(\lambda)=\sum_{i=1}^{2^m}i\tau_i\lambda^{i-1}<\sum_{i=1}^\infty i\tau_i\lambda^{i-1 }
<\sum_{i=1}^{2^5}i\tau_i\lambda_3^{i-1}+\frac{2^5\lambda_3^{2^5}}{(1-\lambda_3)^2}
\approx 3.23712<3.24,
\end{equation*}
proving \eqref{eq:f-prime-bound}. It follows from $g_m(\lambda_m)=0$, \eqref{eq:fm-at-lambda-m-plus-1} and \eqref{eq:f-prime-bound} that
\[
\lambda_m-\lambda_{m+1}\geq \frac{|g_m(\lambda_{m+1})|}{3.24}=\lambda_{m+1}^{2^m} \frac{2\lambda_{m+1}-1}{3.24(1-\lambda_{m+1})}, \qquad m\geq 3.
\]
Hence, using that $\hat{q}_m=1/\lambda_m$ for all $m$,
\begin{align*}
\hat{q}_{m+1}-\hat{q}_m&\geq \hat{q}_{m+1}^{-2^m}\hat{q}_m\hat{q}_{m+1} \frac{2-\hat{q}_{m+1}}{3.24(\hat{q}_{m+1}-1)}>(.265)\hat{q}_{m+1}^{-2^m}, \qquad m\geq 3.
\end{align*}
Here the last inequality follows since, for $m\geq 3$,
\[
\hat{q}_m\hat{q}_{m+1}\frac{2-\hat{q}_{m+1}}{3.24(\hat{q}_{m+1}-1)}\geq \hat{q}_3^2\frac{2-q_{KL}}{3.24(q_{KL}-1)}\approx .26567>.265.
\]
Finally, for $m=1,2$ the statement of the lemma is easily checked by direct calculation.
\end{proof}

\begin{corollary} \label{cor:inequality}
We have
\begin{equation}
q_{KL}-\hat{q}_m>(.265)q_{KL}^{-2^m}{>q_{KL}^{-(2^m+3)}} \qquad\forall m\in\N,
\end{equation}
and
\begin{equation}
q_{KL}-\hat{q}_m>q_{KL}^{-3\cdot 2^{m-1}} \qquad\forall m\geq 2.
\end{equation}
\end{corollary}

\begin{proof}
The first inequality follows directly from Lemma \ref{lem:qm-difference}, as
\[
q_{KL}-\hat{q}_m>\hat{q}_{m+1}-\hat{q}_m>(.265)\hat{q}_{m+1}^{-2^m}>(.265)q_{KL}^{-2^m}>q_{KL}^{-(2^m+3)}.
\]
The second inequality follows from the first one when $m\geq 3$ by simple algebra; and when $m=2$ it follows by direct calculation.
\end{proof}

Define the polynomial functions
\begin{equation}
f_n(q):=q^n-\sum_{i=1}^n \tau_i q^{n-i}=q^n\left(1-\sum_{i=1}^n\frac{\tau_i}{q^i}\right), \qquad n\in\N, \quad q\in\R.
\label{eq:f_n}
\end{equation}
Observe that $x_n=(\tau_n\tau_{n+1}\ldots)_{q_{KL}}=f_{n-1}(q_{KL})$.

\begin{lemma} \label{lem:polynomial-inequality}
For each $n\in\N$ and $1\leq p<q_{KL}$, we have
\[
f_n(q_{KL})-f_n(p)\geq q_{KL}-p.
\]
\end{lemma}

\begin{proof}
Since $f_1(q)=q-1$, the statement holds for $n=1$ with equality. Assume from here on that $n\geq 2$.
For brevity, set $q:=q_{KL}$, and note that $\sum_{i=1}^n\tau_i q^{-i}<1$ for each $n$. Multiplying by $q^n$ gives
\begin{equation}
q^n>\sum_{i=1}^n \tau_i q^{n-i},
\label{eq:leading-power-dominates}
\end{equation}
i.e. $f_n(q)>0$, for every $n\in\N$. Factoring out $q-p$, we obtain
\begin{align*}
f_n(q)-f_n(p)&=q^n-p^n-\sum_{i=1}^{n-1}\tau_i(q^{n-i}-p^{n-i})\\
&=(q-p)\left[\sum_{j=0}^{n-1}q^{n-1-j}p^j-\sum_{i=1}^{n-1}\tau_i\sum_{j=0}^{n-i-1}q^{n-i-1-j}p^j\right].
\end{align*}
Interchanging summations and applying \eqref{eq:leading-power-dominates} with $n-j-1$ in place of $n$, we find that
\begin{equation*}
\sum_{i=1}^{n-1}\tau_i\sum_{j=0}^{n-i-1}q^{n-i-1-j}p^j = \sum_{j=0}^{n-2}\sum_{i=1}^{n-j-1}\tau_i q^{n-i-1-j}p^j
<\sum_{j=0}^{n-2}q^{n-j-1}p^j.
\end{equation*}
Hence, $f_n(q)-f_n(p)>(q-p)p^{n-1}\geq q-p$.
\end{proof}

Recall from Section \ref{sec:prelim} that $(\tau_i)_{i=0}^\f$ is defined by
\begin{equation} \label{eq:tau-n}
\tau_i=s_i\!\!\mod 2,
\end{equation}
where $s_i$ is the number of 1's in the binary representation of $i$.

\begin{lemma} \label{lem:tau-n}
{Let $n\in S$. Then
\begin{enumerate}[{\rm(i)}]
\item $\tau_n\tau_{n+1}\in\set{01, 10}$ ; and
\item $x_n\in(q_G^{-2}, 1)$.
\end{enumerate}
}
\end{lemma}

\begin{proof}
First we prove inductively that for any $n\in S\cap[1, 2^k]$ we have $\tau_n\tau_{n+1}\in\set{01, 10}$. Clearly, this holds for $k=1$ and $k=2$ since $\tau_1\ldots \tau_5=11010$.
Suppose it holds for $n\in S\cap[1, 2^k]$ with $k\ge 2$, and consider $n\in S\cap (2^k, 2^{k+1}]$. By our construction of $S$ we can write
\[
n=2^k+2^{k-1}+j\quad\textrm{with}\quad j\in[1, 2^{k-1}]\cap S.
\]
We consider three cases.

Case I. $j\le 2^{k-1}-2$. Then by (\ref{eq:tau-n}) we have $\tau_n\tau_{n+1}=\tau_j\tau_{j+1}$, so the induction hypothesis yields $\tau_n\tau_{n+1}\in\set{01, 10}$.

Case II. $j=2^{k-1}-1$. Since $\tau_{2^n}=1$ for any $n\in\N$, it follows that $\tau_n\tau_{n+1}=\tau_j 1=\tau_j\tau_{j+1}$. Again, we have $\tau_n\tau_{n+1}\in\set{01, 10}$ by the induction hypothesis.

Case III. $j=2^{k-1}$. Then $\tau_{n}\tau_{n+1}=\tau_{2^{k+1}}\tau_{2^{k+1}+1}=10$.

This proves (i).
Next we show (ii). Since $x_n=(\tau_n\tau_{n+1}\ldots )_{q_{KL}}\in\u_{q_{KL}}$ and $\tau_n\tau_{n+1}\ldots\prec \tau_1\tau_2\ldots$, we have $x_n<1$. If $\tau_n=1$, then $x_n>q_{KL}^{-1}>q_G^{-2}$. Otherwise, by part (i),
 $\tau_{n}\tau_{n+1}=01$, and using $(\tau_i)\in\us_{q_{KL}}$ we conclude that $\tau_{n+2}\tau_{n+3}\ldots \lge \overline{\tau_1\tau_2\ldots}=001011\ldots$, so
\[
x_n=(\tau_n\tau_{n+1}\ldots)_{q_{KL}}>(01001011)_{q_{KL}}>q_G^{-2}.
\]
This completes the proof.
\end{proof}

\begin{remark}
The proof of Lemma \ref{lem:tau-n} implies that for any $n\in S\cap (2^k, 2^{k+1}]$ with $k\ge 2$ we have
$\tau_n\ldots\tau_{2^{k+1}}=\tau_{n-3\cdot 2^{k-1}}\ldots\tau_{2^{k-1}}.$
\end{remark}

\begin{lemma} \label{lem:xn-distance}
For $n,m\in\N$ with $n\neq m$, define
\[
x_n\wedge x_m:=\min\{j\geq 0: \tau_{n+j}\neq\tau_{m+j}\}.
\]
Then $|x_n-x_m|<q_{KL}^{-x_n\wedge x_m}$.
\end{lemma}

\begin{proof}
Let $N=x_n\wedge x_m$. Assume without loss of generality that $\tau_{n+N}=0$ and $\tau_{m+N}=1$. Since $\tau_{m+N}\tau_{m+N+1}\dots\prec (\tau_i)$, $x_{m+N}<1$. Likewise,
$\tau_{n+N}\tau_{n+N+1}\dots\succ \overline{(\tau_i)}$, so that
\[
x_{n+N}>\frac{1}{q_{KL}-1}-1.
\]
Hence,
\[
0<x_m-x_n=q_{KL}^{-N}(x_{m+N}-x_{n+N})<q_{KL}^{-N}\left(2-\frac{1}{q_{KL}-1}\right)\approx 0.73\, q_{KL}^{-N}<q_{KL}^{-N},
\]
as was to be shown.
\end{proof}

\begin{lemma} \label{lem:restriction}\mbox{}

\begin{enumerate}[{\rm(i)}]
\item If $x\ge q_G^{-k}$ for some $k\in\N$, then for any $p\in\ub(x)$ the unique expansion $\Phi_x(p)$ can not begin with $0^k$.
\item If $x<q_G^{-1}$, then for any $p\in\ub(x)\cap(1, q_{KL}]$ the unique expansion $\Phi_x(p)$ begins with $0$.
\end{enumerate}
\end{lemma}

\begin{proof}
For (i) suppose $(d_i)$ is the  unique expansion of $x$ in base $p\in(1,2]$ with $d_1\ldots d_k=0^k$. Then $x=(d_1d_2\ldots)_p\le 1/p^k$. But $x\ge q_G^{-k}$, which implies $p\le q_G$. So the only two unique expansions in base $p$ are $0^\f$ and $1^\f$, and thus $(d_i)=0^\f$. This leads to a contradiction since $x>0$.

Next we prove (ii). Suppose on the contrary that $(d_i)$ is the unique expansion of $x$ in some base $p\in(1, q_{KL}]$ with $d_1=1$. Then by Lemma \ref{lem:unique-expansion},
\[x\ge (1\overline{\al(p)})_p=(01^\f)_p=\frac{1}{p(p-1)}\ge \frac{1}{q_{KL}(q_{KL}-1)}\ge\frac{1}{q_G}, \]
leading to a contradiction with $x<q_G^{-1}$.
\end{proof}

\begin{proposition} \label{prop:power-of-two}
Let $n\in S$  such that $2^{k-1}<n\leq 2^k$ for some integer $k$. If $p\in\ub(x_n)\cap(1,q_{KL})$, then the unique expansion $\Phi_{x_n}(p)$ begins with $\tau_n\dots\tau_{2^k}$.
\end{proposition}

\begin{proof}
If $n=2$ or $4$, then
\[
x_n\geq q_{KL}^{-1}+q_{KL}^{-4}\approx .6575>q_G^{-1},
\]
so by Lemma \ref{lem:restriction} (i) a unique expansion of $x_n$ in any base must begin with $1=\tau_n\dots\tau_{2^k}$. If $n=3$, then $x_n<1/q_{KL}<1/q_G$ and
\[
x_n\geq q_{KL}^{-2}+q_{KL}^{-5}+q_{KL}^{-6}\approx .3986>q_G^{-2},
\]
so by Lemma \ref{lem:restriction} a unique expansion of $x_3$ in any base $p<q_{KL}$ must begin with $01=\tau_3\tau_4$.

For the remainder of the proof, we assume that $n\in S\cap(4,\f)$ and hence $k\geq 3$. Let $p<q_{KL}$ and suppose that ${x_n}\in\u_p$. Then by Lemma \ref{lem:char-quasi-expansion} we have $\Phi_{x_n}(p)\prec \tau_n\tau_{n+1}\ldots$. By way of contradiction, suppose there is an integer $l$ with $n\le l\leq 2^k$ such that $\Phi_{x_n}(p)$ begins with $\tau_n\dots\tau_l^-$. We claim that $l\ge n+2$.

If $l=n$, then $\tau_n=1$, and  we can show that $x_n>q_G^{-1}$. So, by Lemma \ref{lem:restriction} (i) it follows that $\Phi_{x_n}(p)$ begins with $1=\tau_n$, leading to a contradiction. So, $l\ge n+1$. Next we prove that $l\ne n+1$. Suppose on the contrary that $\Phi_{x_n}(p)$ begins with $\tau_n\tau_{n+1}^-$. By Lemma \ref{lem:tau-n} it follows that $\tau_n\tau_{n+1}=01$. Note by Lemma \ref{lem:tau-n} that $x_n>q_G^{-2}$, and then by Lemma \ref{lem:restriction} (i) it follows that $\Phi_{x_n}(p)$ can not begin with $00=\tau_n\tau_{n+1}^-$, again leading to a contradction. This proves the claim.

Set $n_1:=n$, $k_1:=k$, and $l_1:=l$. Since $n_1\in S\cap (2^{k_1-1}, 2^{k_1}]$, it must be the case that $3\cdot 2^{k_1-2}<n_1\leq 2^{k_1}$. Let $n_2:=n_1-3\cdot 2^{k_1-2}$, so $1\leq n_2\leq 2^{k_1-2}$ and $n_2\in S$. Observe by (\ref{eq:tau-n}) that
\begin{equation}
\tau_{n_1}\dots \tau_{2^{k_1}}\tau_{2^{k_1}+1}\dots \tau_{2^{k_1}+2^{k_1-2}}=\tau_{n_2}\dots \tau_{2^{k_1-2}}\tau_{2^{k_1-2}+1}\dots\tau_{2^{k_1-1}}^-,
\label{eq:m-n-match}
\end{equation}
so $x_{n_1}\wedge x_{n_2}=2^{k_1-1}-n_2$.

Let $l_2:=l_1-3\cdot 2^{k_1-2}$; then $n_2<l_2\leq 2^{k_1-2}$. Furthermore, $l_2-n_2=l_1{-}n_1$, and $\tau_{n_1}\dots\tau_{l_1}=\tau_{n_2}\dots\tau_{l_2}$ by \eqref{eq:m-n-match}. Let $k_2$ be the integer such that
$2^{k_2-1}<l_2\le 2^{k_2}.$ Then $k_2\le k_1-2$.
 If $l_2=2^{k_2}$ and $n_2\le 2^{k_2-1}$, then we stop. If $l_2=2^{k_2}$ and $n_2>2^{k_2-1}$, then we continue the process. If $l_2\ne 2^{k_2}$, then we claim that $n_2>2^{k_2-1}$. This follows by observing that $\tau_{2^{k_2-1}+1}\ldots\tau_{l_2}=\overline{\tau_1\ldots \tau_{l_2-2^{k_2-1}}}$ is the smallest block of its length that can occur in any sequence of $\us_p$ with $p<q_{KL}$. Therefore, we either  stop with $l_2=2^{k_2}$ and $n_2\le 2^{k_2-1}$, or we continue the process  with $2^{k_2-1}<n_2<l_2\le 2^{k_2}$.

We know $n_2\not\in\{2,3,4\}$ because $l_2-n_2=l_1-n_1\geq 2$. Hence $n_2>4$ (and $k_2\geq 3$), and since also $n_2\in S\cap(2^{k_2-1}, 2^{k_2}]$, we see that $3\cdot 2^{k_2-2}<n_2\leq 2^{k_2}$. We now repeat the above steps: Define $n_3:=n_2-3\cdot 2^{k_2-2}$ and $l_3:=l_2-3\cdot 2^{k_2-2}$. We claim that
\begin{equation}
2^{k_2-1}-n_3<2^{k_1-1}-n_2.
\label{eq:decreasing-common-part}
\end{equation}
To see this, observe that the definitions of $n_2$ and $k_2$ imply $k_2\leq k_1-2$. Therefore,
\[
2^{k_1-1}-2^{k_2-1}\geq 2^{k_2+1}-2^{k_2-1}=3\cdot 2^{k_2-1}>3\cdot 2^{k_2-2}=n_2-n_3,
\]
and \eqref{eq:decreasing-common-part} follows. As a result, $x_{n_2}\wedge x_{n_3}=2^{k_2-1}-n_3$.

Let $k_3$ be the integer such that $2^{k_3-1}<l_3\le 2^{k_3}$. If $l_3=2^{k_3}$ and $n_3\leq 2^{k_3-1}$, then we stop. Otherwise, we continue the process with $2^{k_3-1}<n_3<l_3\le 2^{k_3}$. Continue this process, either there exist $j\ge 1$ and  an integer $\nu$ such that $l_j=2^\nu$ and $n_j\le 2^{\nu-1}$, or we obtain sequences $(n_i)$, $(k_i)$ and $(l_i)$ with $n_i\in S$, satisfying the relationships
\begin{gather*}
3\cdot 2^{k_i-2}<n_i<l_i\leq 2^{k_i}, \qquad n_{i+1}=n_i-3\cdot 2^{k_i-2}, \\
 l_{i+1}=l_i-3\cdot 2^{k_i-2}, \qquad k_{i+1}\leq k_i-2.
\end{gather*}
Moreover, $x_{n_i}\wedge x_{n_{i+1}}=2^{k_i-1}-n_{i+1}$ for each $i$, and this sequence is decreasing.

Since the sequence $(k_i)$ is strictly decreasing, the process must eventually stop at some finite time $j$ (in the worst case, with $n_j=2$ and $l_j=4$). This means there is an integer $\nu\ge 2$ such that $l_j=2^\nu$ and $n_j\leq 2^{\nu-1}$. Hence $\Phi_{x_n}(p)\in\us_p$ contains the word $\tau_{2^{\nu-1}+1}\dots \tau_{2^\nu}^-=\overline{\tau_1\dots\tau_{2^{\nu-1}}}$ after the leading $\tau_{2^\nu-1}=1$. By Lemma \ref{lem:unique-expansion} and (\ref{eq:qn}) this    implies $p>\hat{q}_\nu$. Then it follows that
\begin{equation}
x_{n_1} \leq \big(\tau_{n_1}\dots\tau_{l_1}\big)_p=\big(\tau_{n_j}\dots\tau_{l_j}\big)_p\\
<\big(\tau_{n_j}\dots\tau_{l_j}\big)_{\hat{q}_\nu}=\big(\tau_{n_j}\dots\tau_{2^\nu}\big)_{\hat{q}_\nu}=f_{n_j-1}(\hat{q}_\nu).
\label{eq:xn-upper-estimate}
\end{equation}
On the other hand, $x_{n_j}=\big(\tau_{n_j}\tau_{n_j+1}\dots\big)_{q_{KL}}=f_{n_j-1}(q_{KL})$, so
\begin{equation}
x_{n_j}-x_{n_1}>f_{n_j-1}(q_{KL})-f_{n_j-1}(\hat{q}_\nu)\geq q_{KL}-\hat{q}_\nu>q_{KL}^{-3\cdot 2^{\nu-1}},
\label{eq:x-diff-lower-estimate}
\end{equation}
where the second inequality follows from Lemma \ref{lem:polynomial-inequality} and the third inequality from Corollary \ref{cor:inequality}. (Note that $\nu\geq 2$ since $l_i\geq 4$.) However, since $2^\nu=l_j\leq 2^{k_{j-1}-2}$, we have $\nu\leq k_{j-1}-2$ and hence
\begin{align*}
x_{n_1}\wedge x_{n_j} &= \min\{x_{n_1}\wedge x_{n_2},\dots,x_{n_{j-1}}\wedge x_{n_j}\}\\
&=2^{k_{j-1}-1}-n_j\geq 2^{\nu+1}-2^{\nu-1}=3\cdot 2^{\nu-1},
\end{align*}
so by Lemma \ref{lem:xn-distance},
\[
|x_{n_1}-x_{n_j}|<q_{KL}^{-3\cdot 2^{\nu-1}},
\]
contradicting \eqref{eq:x-diff-lower-estimate}. Therefore, the number $l$ introduced near the beginning of the proof cannot exist, and $\Phi_{x_n}(p)$ must begin with $\tau_n\dots\tau_{2^k}$.
\end{proof}

\begin{lemma} \label{lem:awkward-induction-statement}
Let $n\in\N\setminus S$. Then either
\begin{enumerate}[{\rm(a)}]
\item $\tau_n\tau_{n+1}\dots$ begins with $11$, $011$ or $0011$; or
\item $2^k<n\leq 2^{k+1}$ for some integer $k\geq 3$, and there are integers $j\geq n$ and $l\leq k-2$ such that $j+2^l\leq 2^{k+1}$, $\tau_j=0$, $\tau_{j+2^l}=1$, and
\begin{equation}
\tau_n\dots\tau_j\big(\tau_{j+1}\dots\tau_{j+2^l}^-\big)^\infty=\tau_n\dots\tau_j\big(\tau_1\dots \tau_{2^l}^-\big)^\infty \in \us_p \qquad\forall p>\hat{q}_l.
\label{eq:induction-hypothesis}
\end{equation}
\end{enumerate}
\end{lemma}

\begin{proof}
We prove the lemma by induction. First, note that $\tau_1\tau_2\ldots=110100110010\ldots$, so statement (a) holds for $n=1,5,6$ and $7$. This provides the basis for the induction.

Next, let $n>8$, $n\not\in S$, and assume the statement of the lemma holds for all elements of $\N\backslash S$ smaller than $n$.
Observe that there is an integer $k\geq 3$ such that $2^k<n\leq 2^{k+1}$. There are two possibilities:

\bigskip
{\em Case 1.} $2^k<n\leq 2^k+2^{k-1}$. We claim that (b) holds for $j=2^k+2^{k-1}$ and $l=k-2$. To see this, note that $\tau_{2^k+2^{k-1}}=\overline{\tau_{2^{k-1}}}=0$, and
\[
\tau_{2^k+2^{k-1}+1}\dots \tau_{2^{k+1}}=\tau_1\dots \tau_{2^{k-1}},
\]
so in particular $\tau_{j+2^l}=\tau_{2^l}=1$, and $\tau_{j+1}\dots\tau_{j+2^l}^-=\tau_1\dots\tau_{2^l}^-$. It remains to verify that
\begin{equation}
\tau_n\dots\tau_j\big(\tau_1\dots \tau_{2^l}^-\big)^\infty \in \us_p \qquad\forall p>\hat{q}_l.
\label{eq:tail-in-up}
\end{equation}
To this end, observe that
\[
\tau_{2^k+1}\dots\tau_{2^k+2^{k-1}}=\overline{\tau_1\dots\tau_{2^{k-1}}}=\overline{\tau_1\ldots \tau_{2^{k-2}}}\tau_1\ldots\tau_{2^{k-2}}^-,
\]
so
\[
\tau_{2^k+1}\dots\tau_j\big(\tau_1\dots\tau_{2^l}^-\big)^\infty=\overline{\tau_1\dots\tau_{2^l}}\big(\tau_1\dots\tau_{2^l}^-\big)^\infty.
\]
Note that
$
\overline{\tau_1\ldots \tau_{2^l-i}}\lle \tau_{i+1}\ldots \tau_{2^l}^-\prec \tau_1\ldots \tau_{2^l-i}$ for all $0\le i<2^l.
$
This implies
\[
\overline{\tau_1\ldots\tau_{2^l}}\prec \overline{\tau_{i+1}\ldots \tau_{2^l}}\tau_1\ldots \tau_i\prec \tau_1\ldots \tau_{2^l}\quad\forall ~1\le i<2^l,
\]
and
\[
\overline{\tau_1\ldots \tau_{2^l}}\prec \tau_{i+1}\ldots\tau_{2^l}^-\tau_1\ldots \tau_i\prec \tau_1\ldots \tau_{2^l}\quad\forall ~0\le i<2^l.
\]
So,
\[
\overline{\al(\hat q_l)}=(\overline{\tau_1\ldots \tau_{2^l}}^+)^\f\lle \si^n(\overline{\tau_1\ldots\tau_{2^l}}(\tau_1\ldots \tau_{2^l}^-)^\f)\lle (\tau_1\ldots\tau_{2^l}^-)^\f=\al(\hat q_l)\quad\forall ~n\ge 1.
\]
Therefore, (\ref{eq:tail-in-up}) follows by Lemmas \ref{lem:char-quasi-expansion} and \ref{lem:unique-expansion}.

\bigskip
{\em Case 2.} $2^k+2^{k-1}<n\leq 2^{k+1}$. Let $m:=n-2^k-2^{k-1}=n-3\cdot 2^{k-1}$. Then $m\not\in S$ by the construction of $S$, and $1\leq m\leq 2^{k-1}$. Observe that
\begin{equation}
\tau_n\tau_{n+1}\dots \tau_{2^{k+1}}=\tau_m\tau_{m+1}\dots \tau_{2^{k-1}}.
\label{eq:TM-copy}
\end{equation}
If $\tau_m\tau_{m+1}\dots$ begins with $11$, $011$ or $0011$, then so does $\tau_n\tau_{n+1}\dots$, so (a) holds for $n$.

Otherwise, $m>8$. Let $k'$ be the integer such that $2^{k'}<m\leq 2^{k'+1}$. Then $3\leq k'\leq k-2$. By the induction hypothesis, there are integers $j'\geq m$ and $l\leq k'-2$ such that $j'+2^l\le 2^{k'+1}$, $\tau_{j'}=0$, $\tau_{j'+2^l}=1$, and
\[
\tau_m\dots\tau_{j'}\big(\tau_{j'+1}\dots \tau_{j'+2^l}^-\big)^\infty=\tau_m\dots\tau_{j'}\big(\tau_1\dots\tau_{2^l}^-\big)^\infty \in \us_p \qquad \forall p>\hat{q}_l.
\]
Set $j:=j'+n-m$. Since $j'+2^l\leq 2^{k'+1}\leq 2^{k-1}$, the last equation and \eqref{eq:TM-copy} imply
\[
\tau_n\dots\tau_j\big(\tau_{j+1}\dots\tau_{j+2^l}^-\big)^\infty=\tau_n\dots\tau_j\big(\tau_1\dots\tau_{2^l}^-\big)^\infty \in \us_p \qquad \forall p>\hat{q}_l.
\]
Thus (b) holds for $n$, and the induction is complete.
\end{proof}

\begin{lemma}\label{lem:case-a}
If $\tau_n\tau_{n+1}\ldots$ begins with $0^k 11$ for $k\in\{0,1,2\}$, then there exists $p\in(q_G, q_{KL})$ such that $\Phi_{x_n}(p)=0^k(10)^\f\in\us_p$.
\end{lemma}

\begin{proof}
We prove this for $k=0$; the proof for $k=1$ or $2$ is similar.
Note by Lemma \ref{lem:unique-expansion} that $x_n=(11\tau_{n+2}\tau_{n+3}\ldots)_{q_{KL}}>((10)^\f)_{q_{KL}}$. So there exists $p\in(1,q_{KL})$ such that $x_n=((10)^\f)_p$.
On the other hand, since
$x_n=(\tau_n\tau_{n+1}\ldots)_{q_{KL}}<(\tau_1\tau_2\ldots)_{q_{KL}}=1$, it follows that
\[
((10)^\f)_p=x_n<1=((10)^\f)_{q_G},
\]
and hence $p>q_G$. This proves $\Phi_{x_n}(p)=(10)^\f\in\us_p$.
\end{proof}

\begin{proof}[Proof of Theorem \ref{thm:set-with-qKL-as-min}]
(i) The proof proceeds in two steps.

{\em Step 1}. We first show that if $n\in S$, then $\min \ub(x_n)=q_{KL}$. Since $\overline{(\tau_i)}\prec\tau_n\tau_{n+1}\dots\prec (\tau_i)$ for every $n\geq 2$, by Lemma \ref{lem:unique-expansion} it is clear that $x_n\in \u_{q_{KL}}$, so it remains to show that $x_n\not\in\u_p$ for any $p<q_{KL}$.

Take $p<q_{KL}$, and suppose by way of contradiction that $x_n\in\u_p$. Let $k$ be the integer such that $2^{k-1}<n\leq 2^k$.
By Proposition \ref{prop:power-of-two}, $(d_i):=\Phi_{x_n}(p)$ must begin with $\tau_n\dots\tau_{2^k}$.
Since $\Phi_{x_n}(p)\prec\Phi_{x_n}(q_{KL})=\tau_n\tau_{n+1}\ldots$, there is therefore an integer $j>k$ such that
\begin{equation}
d_1\dots d_{2^{j-1}-n+1}=\tau_n\dots\tau_{2^{j-1}}
\label{eq:match-so-far}
\end{equation}
and
\begin{equation*}
d_{2^{j-1}-n+2}\dots d_{2^j-n+1}\prec \tau_{2^{j-1}+1}\dots\tau_{2^j}.
\end{equation*}
On the other hand, $x_n\in\u_p$ implies
\[
d_{2^{j-1}-n+2}d_{2^{j-1}-n+3}\dots \succ \overline{\alpha(p)}\succ \overline{\alpha(q_{KL})}=\overline{\tau_1\tau_2\dots},
\]
so that
\[
d_{2^{j-1}-n+2}\dots d_{2^j-n+1} \succeq \overline{\tau_1\dots\tau_{2^{j-1}}}=\tau_{2^{j-1}+1}\dots\tau_{2^j}^-,
\]
where the last equality follows since $\tau_{2^{j-1}+1}\dots\tau_{2^j}=\overline{\tau_1\dots\tau_{2^{j-1}}}^{\,+}$.
Hence
\begin{equation}
d_{2^{j-1}-n+2}\dots d_{2^j-n+1}=\tau_{2^{j-1}+1}\dots\tau_{2^j}^-.
\label{eq:mismatch-at-end}
\end{equation}
Combining this with \eqref{eq:match-so-far}, the definition of quasi-greedy expansion implies
\[
\sum_{i=n}^{2^j} \frac{\tau_i}{p^{i-n+1}}\geq x_n.
\]
It follows that
\begin{align*}
\sum_{i=1}^{2^j}\frac{\tau_i}{p^i} &= \sum_{i=1}^{n-1}\frac{\tau_i}{p^i}+\frac{1}{p^{n-1}}\sum_{i=n}^{2^j}\frac{\tau_i}{p^{i-n+1}}
\geq \sum_{i=1}^{n-1}\frac{\tau_i}{p^i}+\frac{x_n}{p^{n-1}}\\
&> \sum_{i=1}^{n-1}\frac{\tau_i}{q_{KL}^i}+\frac{x_n}{q_{KL}^{n-1}}=1=\sum_{i=1}^{2^j}\frac{\tau_i}{(\hat q_j)^i},
\end{align*}
so that $p<\hat{q}_j$.
Note by (\ref{eq:match-so-far}) and (\ref{eq:mismatch-at-end}) that
\begin{equation}\label{eq:match-1}
 d_{2^{j-1}-n+1}\ldots d_{2^j-n+1}=\tau_{2^{j-1}}\ldots\tau_{2^j}^-=1\overline{\tau_1\ldots\tau_{2^{j-1}}}.
\end{equation}
Since $p<\hat q_j$, we have $(d_i)\in\us_p\subseteq\us_{\hat q_j}$. Then by (\ref{eq:match-1}), $\al(\hat q_j)=(\tau_1\ldots \tau_{2^{j-1}}\overline{\tau_1\ldots \tau_{2^{j-1}}})^\f$ and Lemma \ref{lem:unique-expansion} it follows that the sequence $(d_i)$ must end with $(\tau_1\ldots\tau_{2^{j-1}}\overline{\tau_1\ldots\tau_{2^{j-1}}})^\f=\al(\hat q_j)$, leading to a contradiction with $(d_i)\in\us_{\hat q_j}$.
We conclude that $x_n\not\in\u_p$. Therefore, $\min \ub(x_n)=q_{KL}$.

\medskip

{\em Step 2}. Next, we assume $n\not\in S$, and show that $\inf \ub(x_n)<q_{KL}$. Consider the two possibilities in Lemma \ref{lem:awkward-induction-statement}. If $\tau_n\tau_{n+1}\dots$ begins with $11$, $011$ or $0011$, respectively, then by Lemma \ref{lem:case-a} it follows that $(10)^\infty$, $0(10)^\infty$ or $00(10)^\infty$, respectively is a unique expansion of $x_n$ in some base between $q_G$ and $q_{KL}$, and we are done. Otherwise, there are integers $k,j$ and $l$ as in case (b) of the lemma. Let
\[
(d_i):=\tau_n\dots\tau_j\big(\tau_1\dots \tau_{2^l}^-\big)^\infty=\tau_n\dots\tau_j\big(\tau_{j+1}\dots\tau_{j+2^l}^-\big)^\infty.
\]
There is a unique base $p$ such that $((d_i))_p=x_n$. We claim that $\hat{q}_l<p<q_{KL}$. The first inequality follows since
\begin{align*}
(d_i)_{\hat{q}_l} &=\left(\tau_n\dots\tau_j\big(\tau_1\dots\tau_{2^l}^-\big)^\infty\right)_{\hat{q}_l}
=\big(\tau_n\dots \tau_j^+\big)_{\hat{q}_l}\\
&>\big(\tau_n\dots \tau_j^+\big)_{q_{KL}}>\big(\tau_n\tau_{n+1}\dots\big)_{q_{KL}}=x_n=(d_i)_p.
\end{align*}
By Lemma \ref{lem:awkward-induction-statement}, $(d_i)\in \us_p$ and so $p\in\ub(x_n)$. Finally, $\Phi_{x_n}(p)=(d_i)\prec \tau_n\tau_{n+1}\dots=\Phi_{x_n}(q_{KL})$ and hence $p<q_{KL}$. Thus, $\inf\ub(x_n)\leq p<q_{KL}$. This proves (i).

{%
The proofs of (ii) and (iii) are similar to the proof of (i), with some modifications. We provide the details for (iii) only, and leave (ii) to the interested reader.

Assume $n\in S$ and $\tau_n\dots\tau_{n+6}=1010011$. We must first check the analog of Proposition \ref{prop:power-of-two}. Let $k$ be the integer such that $2^{k-1}<n\leq 2^k$. Suppose there is $p<q_{KL}$ such that $x_n''\in\u_p$. We will show that $\Phi_{x_n''}(p)$ begins with $00\tau_n\dots\tau_{2^k}$.

Following the iterative construction in the proof of Proposition \ref{prop:power-of-two}, we obtain after some finite number of steps integers $n_i,l_i$ and $\nu$ such that $l_i=2^\nu$, $n_i\leq 2^{\nu-1}$, and $p>\hat{q}_\nu$. (Note that $\Phi_{x_n''}(p)$ is just $\Phi_{x_n}(p)$ preceded by $00$.)
Preceding the expansions in \eqref{eq:xn-upper-estimate} by $00$ we obtain
\begin{equation*}
x_n''<\frac{f_{n_i-1}(\hat{q}_\nu)}{\hat{q}_\nu^2},
\end{equation*}
while
\[
x_{n_i}''=\frac{x_{n_i}}{q_{KL}^2}=\frac{f_{n_i-1}(q_{KL})}{q_{KL}^2}.
\]
For simplicity, set $m:=n_i-1$. Then
\begin{align*}
x_{n_i}''-x_n''&> \frac{f_m(q_{KL})}{q_{KL}^2}-\frac{f_m(\hat{q}_\nu)}{\hat{q}_\nu^2}\\
&=\frac{f_m(q_{KL})-f_m(\hat{q}_\nu)}{q_{KL}^2}-f_m(\hat{q}_\nu)\left(\frac{1}{\hat{q}_\nu^2}-\frac{1}{q_{KL}^2}\right)\\
&\geq \frac{q_{KL}-\hat{q}_\nu}{q_{KL}^2}\left(1-f_m(\hat{q}_\nu)\cdot\frac{q_{KL}+\hat{q}_\nu}{\hat{q}_\nu^2}\right),
\end{align*}
where the second inequality follows by Lemma \ref{lem:polynomial-inequality}. We now estimate the expression in parentheses. First,
\[
f_m(q_{KL})=(\tau_{m+1}\tau_{m+2}\dots)_{q_{KL}}\leq (\tau_2\tau_3\dots)_{q_{KL}}=q_{KL}-1,
\]
so again using Lemma \ref{lem:polynomial-inequality},
\[
f_m(\hat{q}_\nu)\leq f_m(q_{KL})-(q_{KL}-\hat{q}_\nu)\leq \hat{q}_\nu-1.
\]
Thus,
\begin{align*}
1-f_m(\hat{q}_\nu)\cdot\frac{q_{KL}+\hat{q}_\nu}{\hat{q}_\nu^2} &\geq 1-\frac{(\hat{q}_\nu-1)(q_{KL}+\hat{q}_\nu)}{\hat{q}_\nu^2}
=\frac{q_{KL}+\hat{q}_\nu-q_{KL}\hat{q}_\nu}{\hat{q}_\nu^2}\\
&\geq \frac{2q_{KL}-q_{KL}^2}{q_{KL}^2}=\frac{2}{q_{KL}}-1\approx 0.119.
\end{align*}
Here the second inequality follows since the expression preceding it is decreasing in $\hat{q}_\nu$, and $\hat{q}_\nu\nearrow q_{KL}$ as $\nu\to\infty$. Combining these estimates, we obtain
\begin{equation}
x_{n_i}''-x_n''> (0.119)\frac{q_{KL}-\hat{q}_\nu}{q_{KL}^2}.
\label{eq:xn-difference-lower-estimate}
\end{equation}
Using the more precise first estimate in Corollary \ref{cor:inequality} we get for all $\nu\geq 4$,
\[
x_{n_i}''-x_n''>(0.119) q_{KL}^{-2^\nu-5}>q_{KL}^{-3\cdot 2^{\nu-1}-2}.
\]
On the other hand,
\[
x_{n_i}''-x_n''=\frac{x_{n_i}-x_n}{q_{KL}^2}<q_{KL}^{-3\cdot 2^{\nu-1}-2}
\]
as in the proof of Proposition \ref{prop:power-of-two}. This gives a contradiction if $\nu\geq 4$.

When $\nu\leq 3$, we use the fact that in the sequence $(\tau_i)$, the word $101\,0011$ is always followed by either $0010\,1100$ or $0010\,1101$. This implies that $x_n\wedge x_{n_i}\geq 14$, so by Lemma \ref{lem:xn-distance},
\[
x_{n_i}''-x_n''=\frac{x_{n_i}-x_n}{q_{KL}^2}<q_{KL}^{-16}\approx 0.000092.
\]
For $\nu=2$ or $3$, we can directly compute the right hand side of \eqref{eq:xn-difference-lower-estimate}; using $\hat{q}_2\approx 1.7549$ and $\hat{q}_3\approx 1.7846$, we obtain $x_{n_i}''-x_n''>0.00120$ for $\nu=2$, and $>0.000098$ for $\nu=3$. Thus, for these values too we arrive at a contradiction.

Now that we have proved the analog of Proposition \ref{prop:power-of-two}, the rest of the proof that $\min\ub(x_n'')=q_{KL}$ is a straightforward modification of Step 1 of (i) above. We omit the details.

Conversely, if $n\not\in S$, then the proof of Step 2 above continues to work, provided we precede all expansions by $00$. If $n\in S$ but $\tau_n\dots\tau_{n+6}\neq 101\,0011$, there are two cases:

(i) $\tau_n\tau_{n+1}\dots$ begins with $0$ or $100$. Then
\[
x_n''<(00101)_{q_{KL}}=q_{KL}^{-3}+q_{KL}^{-5}\approx .2300<q_G^{-3},
\]
so $x_n''$ has a unique expansion starting with $000$ in some base $p<q_{KL}$.

(ii) $\tau_n\tau_{n+1}\dots$ begins with $101\,0010$. Then
\[
x_n''<(00101\,0011)_{q_{KL}}=q_{KL}^{-3}+q_{KL}^{-5}+q_{KL}^{-8}+q_{KL}^{-9}\approx .24499.
\]
On the other hand,
\[
\big(001\,(0011)^\infty\big)_{\hat{q}_2}=\big(011\,(0011)^\infty)_{\hat{q}_2}-\frac{1}{\hat{q}_2^2}=\frac{1}{\hat{q}_2}-\frac{1}{\hat{q}_2^2}\approx .24512>x_n'',
\]
so we see that $001\,(0011)^\infty$ is a unique expansion of $x_n''$ in some base $p\in(\hat{q}_2,q_{KL})$.

In both cases, $\inf \ub(x_n'')<q_{KL}$.
}
\end{proof}

\begin{remark} \label{rem:type-III-algebraic}
Since $x_n=f_{n-1}(q_{KL})$, it follows immediately that in this case also, for $n\in S$, $q_s(x_n)=q_{KL}$ is algebraic over the field $\mathbb{Q}(x_n)$, even though $x_n$ is of type III and the minimal unique expansion of $x_n$ is not eventually periodic.
\end{remark}

Theorem \ref{thm:set-with-qKL-as-min} implies that $L({q_{KL}})$ is infinite.
We now show that it has in fact infinitely many accumulation points. This uses the following elementary lemma.

\begin{lemma} \label{lem:TM-inequality}
For any $k\in\N$, we have
\[
\tau_{2^k+1}\tau_{2^k+2}\dots \succ \tau_{2^{k+1}+1}\tau_{2^{k+1}+2}\dots.
\]
\end{lemma}

\begin{proof}
Note by \eqref{eq:TM-recursion} that $\tau_{2^k+1}\dots\tau_{2^{k+1}}=\overline{\tau_1\dots\tau_{2^k}}^{\,+}$ and
$\tau_{2^{k+1}+1}\dots \tau_{2^{k+1}+2^k}=\overline{\tau_1\dots\tau_{2^k}}$. The lemma follows.
\end{proof}

\begin{proposition} \label{prop:accumulation-points} \mbox{}

\begin{enumerate}[{\rm(i)}]
\item The level set $L(q_{KL})$ has infinitely many right accumulation points which lie themselves in $L(q_{KL})$.
\item The level set $L(q_{KL})$ also has infinitely many left accumulation points which lie themselves in $L(q_{KL})$.
\end{enumerate}
\end{proposition}

\begin{proof}
(i) Fix $n\in S$ initially, and let $k:=k_n$ be the integer such that $2^{k-1}<n\leq 2^k$. Observe that $n>2^{k}-2^{k-2}$ by the construction of $S$. Now set
\[
n_i:=2^{k+2i-2}-2^{k}+n, \qquad i\in\N.
\]
Using the definition of $S$ and induction it follows that $n_i\in S$ for each $i$, since $n_{i+1}=3\cdot 2^{k+2i-2}+n_i$. Furthermore,
\begin{align*}
\tau_{n_i}\tau_{n_i+1}\dots&=(\tau_{n_i}\dots\tau_{2^{k+2i-2}})(\tau_{2^{k+2i-2}+1}\tau_{2^{k+2i-2}+2}\dots)\\
&=(\tau_n\dots\tau_{2^k})(\tau_{2^{k+2i-2}+1}\tau_{2^{k+2i-2}+2}\dots),
\end{align*}
so {by Lemma \ref{lem:TM-inequality}} it follows that $\tau_{n_i}\tau_{n_i+1}\dots\succ \tau_{n_{i+1}}\tau_{n_{i+1}+1}\dots$, and hence, $x_{n_i}>x_{n_{i+1}}$. By Theorem \ref{thm:set-with-qKL-as-min}, $x_{n_i}\in L(q_{KL})$ for each $i$. Hence the limit
\[
x_n^*:=\lim_{i\to\infty}x_{n_i}=\big(\tau_n\dots\tau_{2^k}\overline{\tau_1\tau_2\dots}\big)_{q_{KL}}
\]
is a right accumulation point of $L(q_{KL})$, and it lies itself in $L(q_{KL})$ by right continuity of $q_s(x)$ (cf.~Theorem \ref{thm:cadlag-intro}).

We now argue that infinitely many of these accumulation points $x_n^*, n\in S$ are distinct. Observe first that for any $j\leq 2^k$,
\begin{align*}
\tau_j\dots\tau_{2^k}\overline{\tau_1\tau_2\dots}
&=\tau_j\dots\tau_{2^k}\tau_{2^k+1}\dots\tau_{2^{k+1}}^-\overline{\tau_{2^k+1}\tau_{2^k+2}\dots}\\
&\prec \tau_j\tau_{j+1}\dots \prec \tau_1\tau_2\dots.
\end{align*}
Therefore, $\tau_n\dots\tau_{2^k}\overline{\tau_1\tau_2\dots}$ is the quasigreedy expansion of $x_n^*$ in base $q_{KL}$ by Lemma \ref{lem:char-quasi-expansion}(iii). Now suppose $n\in S\cap(2^{k_n-1}, 2^{k_n}]$ and $n'\in S\cap(2^{k_n'-1}, 2^{k_n'}]$ with $2^{k_n}-n\neq 2^{k_n'}-n'$. Since no two tails of the sequence $(\tau_i)$ coincide, it is clear that
\[
\Phi_{x_n^*}(q_{KL})=\tau_n\dots\tau_{2^{k_n}}\overline{\tau_1\tau_2\dots}\neq \tau_{n'}\dots\tau_{2^{k_n'}}\overline{\tau_1\tau_2\dots}=\Phi_{x_{n'}^*}(q_{KL}),
\]
and hence, $x_n^*\neq x_{n'}^*$. It is not hard to see that $2^{k_n}-n$ takes on infinitely many distinct values for $n\in S$. For instance, we may take $n=2^k-2^{k-2}+2$ for any $k\in\N_{\geq 3}$; then $k_n=k$, and $2^{k_n}-n=2^{k-2}-2$. Thus, $L(q_{KL})$ has infinitely many right accumulation points.

(ii) Fix $n\in S$; we show that $x_n$ is a left accumulation point of $L(q_{KL})$. Let $k$ be the integer such that $2^{k-1}<n\leq 2^k$, and set $n_i:=n+3\cdot 2^{k+2i}$, for $i\in\N$. Note that $3\cdot 2^{k+2i}<n_i\leq 3\cdot 2^{k+2i}+2^k\leq 2^{k+2i+2}$ and $n_i-3\cdot 2^{k+2i}=n\in S$, so $n_i\in S$ and hence $x_{n_i}\in L(q_{KL})$. Furthermore,
\begin{equation}
\tau_{n_i}\dots\tau_{2^{k+2i+2}}=\tau_n\dots\tau_{2^{k+2i}}.
\label{eq:LAP-initial-match}
\end{equation}
Since $2^{k+2i+2}-n_i\geq 2^{k+2i}-2^k\to\infty$ as $i\to\infty$, it follows that $\tau_{n_i}\tau_{n_i+1}\dots\to\tau_n\tau_{n+1}\dots$, and hence $x_{n_i}\to x_n$. It remains to verify that $x_{n_i}<x_n$. But this follows since
\[
\tau_{2^{k+2i}+1}\dots\tau_{2^{k+2i+1}}=\overline{\tau_1\dots\tau_{2^{k+2i}}}^{\,+}
\]
and
\[
\tau_{2^{k+2i+2}+1}\dots\tau_{2^{k+2i+2}+2^{k+2i}}=\overline{\tau_1\dots\tau_{2^{k+2i}}},
\]
so combined with \eqref{eq:LAP-initial-match}, we obtain $\tau_{n_i}\tau_{n_i+1}\dots\prec\tau_n\tau_{n+1}\dots$. This implies $x_{n_i}<x_n$, which gives the desired result.
\end{proof}

\begin{example} \label{ex:accumulation-points}
Taking $n=2^k$ for any $k\in\N$ in the above proof we get $x_n^*=(1\,\overline{\tau_1\tau_2\dots})_{q_{KL}}=1/q_{KL}(q_{KL}-1)$, which is the smallest element of $L(q_{KL})$ in the interval $[q_G^{-1},1)$. Similarly, taking $n=2^k-1$ with $k$ even, we obtain $x_n^*=(01\,\overline{\tau_1\tau_2\dots})_{q_{KL}}=1/q_{KL}^2(q_{KL}-1)$, which is the smallest element of $L(q_{KL})$ in the interval $[q_G^{-2},q_G^{-1})$. {On the other hand, the point $1/q_{KL}^3(q_{KL}-1)$ does not lie in $L(q_{KL})$ because it does not lie in $[q_G^{-3},q_G^{-2})$. Considering Theorem \ref{thm:set-with-qKL-as-min} (iii) and a modification of the proof of Proposition \ref{prop:accumulation-points} (i), the smallest element of $L(q_{KL})$ in this last interval (and the smallest overall) appears to be
\[
(001010011\overline{\tau_1\tau_2\ldots})_{q_{KL}}=\frac{1}{q_{KL}^3}+\frac{1}{q_{KL}^5}+\frac{1}{q_{KL}^8}+\frac{1}{q_{KL}^9(q_{KL}-1)}\approx 0.2464427.
\]
}
Observe that the limit points $x_n^*$ are all of type I.
\end{example}

\begin{proof}[Proof of Theorem \ref{thm:infinite-level-set-qKL-intro}]
The theorem follows from Theorem \ref{thm:set-with-qKL-as-min} and Proposition \ref{prop:accumulation-points}.
\end{proof}

\subsection{Other infinite level sets}
We next show that there are in fact infinitely many levels $q>q_{KL}$ such that $L(q)$ is infinite. Recall the de Vries-Komornik numbers $\hat{q}_c(\sa)$ from Section \ref{sec:prelim}.
Let $Q_3\approx 1.839$ be the tribonacci number, that is, the unique positive root of $q^{-1}+q^{-2}+q^{-3}=1$. So, $\al(Q_3)=(110)^\f$. Theorem \ref{thm:other-infinite-level-sets-intro} follows from the following more detailed result.

\begin{theorem} \label{thm:deVries-Komornik}
Let $\hat{q}_c(\sv)\in(q_{KL},Q_3)$ be a de Vries-Komornik number generated by an admissible word $\sv$ of length $m$, and let $(\alpha_i):=\alpha(\hat{q}_c(\sv))$. For $n\in\N$, define
\[
x_n:=x_n(\sv):=\big(\alpha_n\alpha_{n+1}\dots\big)_{\hat{q}_c(\sv)}.
\]
Then $\min\ub(x_n)=\hat{q}_c(\sv)$ for all $n$ of the form $n=m 2^{k-1}$ or $n=m(2^{2k-1}-1)$, where $k\in\N$.
\end{theorem}

\begin{proof}
For brevity, set $q_\sv:=\hat{q}_c(\sv)$.
We first demonstrate that $\min\ub(x_n)=q_\sv$ for $n=m 2^{k-1}$. Since $q_{KL}<q_\sv<Q_3$, we have that $m>3$ and $\sv=a_1\dots a_m$ begins with $110$. Note that $x_n\in\u_{q_\sv}$ for each $n$ because $q_\sv\in\ub$. So, $q_s(x_n)\le q_{\sv}$.

If $n=m 2^{k-1}$, then $\Phi_{x_n}(q_\sv)$ begins with $\alpha_{m 2^{k-1}}\alpha_{m2^{k-1}+1}\dots\alpha_{m2^{k-1}+m-1}=1\overline{\alpha_1\dots\alpha_{m-1}}$; hence $\Phi_{x_n}(q_\sv)$ begins with $1001$. It follows that
\[
x_n\geq \frac{1}{q_\sv}+\frac{1}{q_\sv^4}>\frac{1}{Q_3}+\frac{1}{Q_3^4}\approx .6312>\frac{1}{q_G},
\]
so by Lemma \ref{lem:restriction} (i) any unique expansion of $x_n$ must begin with $1$. The rest of the proof is then analogous to the first part of the proof of Theorem \ref{thm:set-with-qKL-as-min}. We just have to replace the special bases $\hat{q}_j$ with their local analogs $\hat{q}_j(\sv)$; see Section \ref{sec:prelim} for the definition.

The proof for $n=m(2^{2k-1}-1)$ is more complicated and requires reasoning similar to the proof of Proposition \ref{prop:power-of-two}. Suppose there is a base $p<q_\sv$ such that $x\in\u_p$. We must show that $\Phi_{x_n}(p)$ starts with $\alpha_n\dots\alpha_{m2^{2k-1}}=1\overline{\alpha_1\dots\alpha_m}^{\,+}=1\overline{\sv}$. The rest of the proof is then again similar to the first part of the proof of Theorem \ref{thm:set-with-qKL-as-min}. By analogy with Lemma \ref{lem:polynomial-inequality}, we define the polynomial function
\[
f(q):=q^{m-1}-\sum_{i=1}^{m-1}\alpha_i q^{m-1-i}, \qquad q\in\R.
\]
As in Lemma \ref{lem:polynomial-inequality} we can show that
\begin{equation} \label{eq:poly-inequality-m}
f(q_\sv)-f(\hat{q}_1(\sv))\geq q_\sv-\hat{q}_1(\sv).
\end{equation}

As in the case $n=m 2^{k-1}$, $\Phi_{x_n}(q_\sv)$ begins with $1001$ and so $x_n>1/q_G$, and hence $\Phi_{x_n}(p)$ must begin with $1$. Let $(d_i):=\Phi_{x_n}(p)$, and note that $d_2\dots d_{m+1}\succeq \overline{\alpha_1\dots\alpha_m}=\overline{\sv^+}$. So if $\Phi_{x_n}(p)$ does not begin with
$1\overline{\sv}$ it must begin with $1\overline{\sv^+}$, since $p<q_\sv$ implies $\Phi_{x_n}(p)\prec\Phi_{x_n}(q_\sv)$. But this is only possible if $p>\hat{q}_1(\sv)$.

Similar to the proof of Proposition \ref{prop:power-of-two}, we then obtain
\[
x_n\leq (1\overline{\sv})_p<(1\overline{\sv})_{\hat{q}_1(\sv)}=f(\hat{q}_1(\sv)),
\]
noting that $\alpha(\hat{q}_1(\sv))=(\sv^+\overline{\sv^+})^\infty$. On the other hand, $x_m=f(q_\sv)$, so \eqref{eq:poly-inequality-m} gives
\[
x_m-x_n>f(q_\sv)-f(\hat{q}_1(\sv))\geq q_\sv-\hat{q}_1(\sv).
\]
By analogy with Corollary \ref{cor:inequality} we can show that
\[
q_\sv-\hat{q}_1(\sv)\geq \frac{(q_{KL}-1)^2}{q_\sv^{2m+4}}.
\]
(In the proof of \cite[Lemma 3.6]{Allaart-Baker-Kong-17}, we can take $n=2m$, $N=4$, $p_2=q=q_\sv$, $p_1=\hat{q}_1(\sv)$ and $p=q_{KL}$.)
Since $\sv$ is an admissible word starting with $110$ but not equal to $110$ and $q_\sv>q_{KL}$, we have $m\geq 5$, so that we can further estimate
\[
q_\sv-\hat{q}_1(\sv)\geq \frac{q_\sv(q_{KL}-1)^2}{q_\sv^{3m}}>\frac{q_{KL}(q_{KL}-1)^2}{q_\sv^{3m}}>q_\sv^{-3m}.
\]
Hence
\begin{equation} \label{eq:far-apart}
x_m-x_n>q_\sv^{-3m}.
\end{equation}
At the same time, however, we claim that $x_m\wedge x_n=3m$. To see this, note first that $\alpha_n=\alpha_m=1$. Furthermore,
\begin{gather*}
\alpha_{n+1}\dots\alpha_{n+m}=\alpha_{m(2^{2k-1}-1)+1}\dots\alpha_{m\cdot 2^{2k-1}}=\overline{\alpha_1\dots\alpha_m}^{\,+}=\alpha_{m+1}\dots\alpha_{2m}, \\
\alpha_{n+m+1}\dots\alpha_{n+3m}=\alpha_{m\cdot 2^{2k-1}+1}\dots\alpha_{m\cdot 2^{2k-1}+2m}=\overline{\alpha_1\dots\alpha_{2m}},
\end{gather*}
and
\[
\alpha_{2m+1}\dots\alpha_{4m}=\overline{\alpha_1\dots\alpha_{2m}}^{\,+}.
\]
Thus $\alpha_{n}\dots\alpha_{n+3m}=\alpha_m\dots\alpha_{4m}^-$, yielding $x_m\wedge x_n=3m$.
But as in Lemma \ref{lem:xn-distance}, $|x_m-x_n|\leq q_\sv^{-x_m\wedge x_n}$, so this contradicts \eqref{eq:far-apart}. Therefore, $x\not\in\u_p$, and $\min\ub(x)=q_\sv$.
\end{proof}

\begin{remark}\mbox{}

\begin{enumerate}[{\rm(i)}]
\item When $\sv=110$ (so that $q_\sv>Q_3$ and $m=3$), we still have $\min\ub(x_n)=\hat{q}_c(\sv)$ for all $n$ of the form $n=m(2^{2k-1}-1)$, but not for $n=m 2^k$. While the above proof doesn't quite work, we can show by direct calculation that $q_\sv\approx 1.87064$ and $\hat{q}_1(\sv)\approx 1.86675$, so that
\[
q_\sv-\hat{q}_1(\sv)\approx 0.0039>0.00357\approx q_\sv^{-9}=q_\sv^{-3m}.
\]
Hence, we still get a contradiction in the same way, though only barely!

(For $n=m 2^k$ with $k\geq 1$, we can check that $x_n<1/q_G$, so there is a unique expansion of $x_n$ in some base $p<q_\sv$ beginning with $0$.)

\item Recall from Remark \ref{rem:qmax} that $q_{\max}=\max q_s(x)\approx 1.88845$. Hence the range of $q_s(x)$ contains infinitely many de Vries-Komornik numbers $\hat{q}_c(\sa)>1.87064$. We do not know if the level set $L(\hat{q}_c(\sa))$ is infinite for any of these.
    \item {The points $x_n$ with $\min\ub(x_n)=\hat q_c(\sv)$ described in Theorem \ref{thm:deVries-Komornik} are all of type III.}
    \end{enumerate}
\end{remark}

\section{Komornik-Loreti cascades} \label{sec:cascades}

In this section we explore the delicate behavior of the graph of $q_s(x)$ below the line $y=q_{KL}$.
We fix an integer $n\in S\cup\{1\}$, and let $k$ be the integer such that $2^{k-1}<n\leq 2^k$.
Note by Theorem \ref{thm:set-with-qKL-as-min} (i) that $x_n=(\tau_n\tau_{n+1}\dots)_{q_{KL}}$ and $q_s(x_n)=q_{KL}$. Define the sequences
\[
\sc_{n,j}^{(m)}:=\tau_n\dots\tau_{2^{m-1}}(\overline{\tau_1\dots\tau_{2^{m-1}}}^{\,+})^j(\overline{\tau_1\dots\tau_{2^m}}^{\,+})^\infty, \qquad m>k,\ j\in\N
\]
and
\[
\sc_{n,\infty}^{(m)}:=\tau_n\dots\tau_{2^{m-1}}(\overline{\tau_1\dots\tau_{2^{m-1}}}^{\,+})^\infty.
\]
Accordingly, set
\[
\xi_m:=\big(\sc_{n,\infty}^{(m)}\big)_{\hat{q}_m}, \qquad \xi_{m,j}:=\big(\sc_{n,j}^{(m)}\big)_{\hat{q}_m},
\]
where the bases $\hat q_m$ are defined in (\ref{eq:qn}).

\begin{proposition} \label{prop:location}
The sequence $(\xi_m)_{m>k}$ strictly decreases to $x_n$. Furthermore, for each fixed $m>k$, the sequence $(\xi_{m,j})$ strictly increases to $\xi_m$, and $\xi_{m+1}<\xi_{m,1}$. Therefore,
\[
x_n<\cdots<\xi_{m+1}<\xi_{m,1}<\xi_{m,2}<\cdots<\xi_{m,j-1}<\xi_{m,j}\nearrow \xi_m\leq \xi_{k+1}
\]
for all $m>k$ and $j\ge 2$.
\end{proposition}

\begin{proof}
{The first statement follows from the definition of $\xi_m$, noting that
\begin{align*}
\xi_{m+1}
&=\left(\tau_n\ldots \tau_{2^{m-1}}\overline{\tau_1\ldots \tau_{2^{m-1}}}^+(\overline{\tau_1\ldots \tau_{2^{m-1}}}\tau_1\ldots \tau_{2^{m-1}})^\f\right)_{\hat q_{m+1}}\\
&<\left(\tau_n\ldots \tau_{2^{m-1}}(\overline{\tau_1\ldots \tau_{2^{m-1}}}^+)^\f\right)_{\hat q_{m+1}}\\
&<\left(\tau_n\ldots \tau_{2^{m-1}}(\overline{\tau_1\ldots \tau_{2^{m-1}}}^+)^\f\right)_{\hat q_{m}}=\xi_m.
\end{align*}
For the second statement we observe that $\xi_{m,j}=(\tau_n\ldots \tau_{2^{m-1}}(\overline{\tau_1\ldots \tau_{2^{m-1}}}^+)^{j+1})_{\hat q_m}$. This implies $\xi_{m,j+1}>\xi_{m,j}$ for any $j\ge 1$.} Finally, note that
\begin{align*}
\xi_{m,1}&=\left(\tau_n\dots\tau_{2^{m-1}}\overline{\tau_1\dots\tau_{2^{m-1}}}^{\,+}(\overline{\tau_1\dots\tau_{2^m}}^{\,+})^\f\right)_{\hat{q}_m}\\
&=\left(\tau_n\dots\tau_{2^m}(\overline{\tau_1\dots\tau_{2^m}}^{\,+})^\f\right)_{\hat{q}_m}\\
&>\left(\tau_n\dots\tau_{2^m}(\overline{\tau_1\dots\tau_{2^m}}^{\,+})^\f\right)_{\hat{q}_{m+1}}=\xi_{m+1}.
\end{align*}
This completes the proof.
\end{proof}

Now we state our main result in this section.

\begin{theorem} \label{thm:Komornik-Loreti-cascades}
Let $n\in S\cup\set{1}$, {and let $k$ be the integer such that $2^{k-1}<n\leq 2^k$.}
\begin{enumerate}[{\rm(i)}]
\item
For each $x\in(x_n, \xi_{k+1})$ we have $q_s(x)=\min\ub(x)$. Furthermore, the value $q_s(x)$ can be explicitly calculated for each $x\in[\xi_{m+1}, \xi_m)$ with $m>k$.
\begin{itemize}
\item If $x\in[\xi_{m+1},\xi_{m,1})$, then $q_s(x)$ is the unique $q\in (\hat{q}_m,\hat{q}_{m+1}]$ such that
\begin{equation}
\big(\tau_n\dots\tau_{2^m}(\overline{\tau_1\dots\tau_{2^m}}^{\,+})^\infty\big)_q=x.
\label{eq:x-in-special-interval}
\end{equation}
In particular, $q_s(\xi_{m+1})=\hat q_{m+1}$.

\item If $x\in[\xi_{m,j-1}, \xi_{m,j})$ for some $j\ge 2$, then  $q_s(x)$ is the unique $q\in (\hat{q}_m,\hat{q}_{m+1})$ such that
\begin{equation}
\big(\tau_n\dots\tau_{2^{m-1}}(\overline{\tau_1\dots\tau_{2^{m-1}}}^{\,+})^j(\overline{\tau_1\dots\tau_{2^m}}^{\,+})^\infty\big)_q=x.
\label{eq:minimal-expansion1}
\end{equation}
\end{itemize}

\item The function $q_s$ has the following properties on each interval $[\xi_{m+1}, \xi_m)$ with $m>k$:
\begin{itemize}
  \item $q_s$ is strictly decreasing and convex on the interval $[\xi_{m+1}, \xi_{m,1})$ or $[\xi_{m,j-1}, \xi_{m,j})$ with $j\ge 2$.
  \item The sequence $q_s(\xi_{m,j})$ is strictly decreasing in $j$ and converges to $\hat q_m$.
  \item $\lim_{x\nearrow \xi_{m,j}}q_s(x)=\hat q_m<q_s(\xi_{m,j})$ for all $j\ge 1$, and $\lim_{x\nearrow \xi_m}q_s(x)=\hat q_m=q_s(\xi_m)$.
\end{itemize}
\end{enumerate}
\end{theorem}

We call the graph of $q_s(x)$ on the interval $(x_n,\xi_{k+1}]$ a {\em Komornik-Loreti cascade}, because it ``cascades down" (with some bounces along the way!) from the Komornik-Loreti constant $q_{KL}$ at the left endpoint $x_n$. The largest Komornik-Loreti cascade is found above the interval $(1,q_G]$ (see Figure \ref{fig:1}). It is obtained by setting $n=1$, so that $x_n=1$ and $k=0$, and $\xi_{k+1}=\xi_1=(1^\infty)_{\hat{q}_1}=1/(\hat{q}_1-1)=\hat{q}_1=q_G$.

The proof uses the following proposition, which generalizes a crucial inequality from the proof of \cite[Lemma 3.5]{Kong_2016}.

\begin{proposition} \label{prop:Derong-overlap-general}
Let $n\in S\cup\set{1}$, {and let $k$ be the integer such that $2^{k-1}<n\leq 2^k$.} Then for all $m>k$ and all $j\in\N$, we have
\begin{equation}\label{eq:may-31-1}
\big(\sc_{n,j}^{(m)}\big)_{\hat{q}_m}>\big(\sc_{n,j+1}^{(m)}\big)_{\hat{q}_{m+1}}.
\end{equation}
\end{proposition}

\begin{proof}
We need the following estimates, precise to $6$ decimal places:
\begin{equation} \label{eq:q-hat-estimates}
\hat{q}_2\approx 1.754877, \qquad \hat{q}_3\approx 1.784599, \qquad \hat{q}_4\approx 1.787207.
\end{equation}
Observe that
$
\big(\sc_{n,j}^{(m)}\big)_{\hat{q}_m}=\big(\tau_n\dots\tau_{2^{m-1}}(\overline{\tau_1\dots\tau_{2^{m-1}}}^{\,+})^{j+1}\big)_{\hat{q}_m},
$
so that
\begin{align}
\begin{split}
\big(\sc_{n,j}^{(m)}\big)_{\hat{q}_m}-\big(\sc_{n,j+1}^{(m)}\big)_{\hat{q}_{m+1}}
&=\big(\tau_n\dots\tau_{2^{m-1}}(\overline{\tau_1\dots\tau_{2^{m-1}}}^{\,+})^{j+1}\big)_{\hat{q}_m}\\
&	\qquad -\big(\tau_n\dots\tau_{2^{m-1}}(\overline{\tau_1\dots\tau_{2^{m-1}}}^{\,+})^{j+1}\big)_{\hat{q}_{m+1}}\\
& \qquad -\hat{q}_{m+1}^{-(j+2)2^{m-1}+(n-1)}\big((\overline{\tau_1\dots \tau_{2^m}}^{\,+})^\infty\big)_{\hat{q}_{m+1}}.
\end{split}
\label{eq:c-difference-expression}
\end{align}
This expression is clearly increasing in $j$, since $\hat{q}_m<\hat{q}_{m+1}$. Hence it is smallest when $j=1$. Moreover,
we have the inequality
\[
\big(\sc_{n,1}^{(m)}\big)_{\hat{q}_m}-\big(\sc_{n,2}^{(m)}\big)_{\hat{q}_{m+1}}
\geq \hat{q}_{m+1}^{-(2^{m-1}-n)}\left[\big(\sc_{2^{m-1},1}^{(m)}\big)_{\hat{q}_m}-\big(\sc_{2^{m-1},2}^{(m)}\big)_{\hat{q}_{m+1}}\right],
\]
since $\big(\tau_n\dots\tau_{2^{m-1}-1}\big)_{\hat{q}_m}\geq \big(\tau_n\dots\tau_{2^{m-1}-1}\big)_{\hat{q}_{m+1}}$. Hence, it suffices to prove (\ref{eq:may-31-1}) for $n=2^{m-1}$ and $j=1$. Using the expression in \eqref{eq:c-difference-expression} we only need to prove
\begin{equation}\label{eq:may-31-2}
\begin{split}
\big(\sc_{2^{m-1},1}^{(m)}\big)_{\hat{q}_m}-\big(\sc_{2^{m-1},2}^{(m)}\big)_{\hat{q}_{m+1}}
&=\big(1(\overline{\tau_1\dots\tau_{2^{m-1}}}^{\,+})^2\big)_{\hat{q}_m}
-\big(1(\overline{\tau_1\dots\tau_{2^{m-1}}}^{\,+})^2\big)_{\hat{q}_{m+1}}\\
&\qquad -\hat{q}_{m+1}^{-2^m-1}\big((\overline{\tau_1\dots \tau_{2^m}}^{\,+})^\infty\big)_{\hat{q}_{m+1}}>0.
\end{split}
\end{equation}

We first check the inequality (\ref{eq:may-31-2}) for $m=1$ and $m=2$. For $m=1$ we have
\begin{align*}
\big(\sc_{1,1}^{(1)}\big)_{\hat{q}_1}-\big(\sc_{1,2}^{(1)}\big)_{\hat{q}_2}
&>(111)_{\hat{q}_1}-(111)_{\hat{q}_2}-\hat{q}_2^{-3}\big((01)^\infty\big)_{\hat{q}_1}\\
&=\sum_{i=1}^3 \left(\frac{1}{\hat{q}_1^i}-\frac{1}{\hat{q}_2^i}\right)-\frac{1}{\hat{q}_2^3\hat{q}_1}\geq .042,
\end{align*}
where we use \eqref{eq:q-hat-estimates} and the estimate
\begin{equation}
\big((\overline{\tau_1\dots \tau_{2^m}}^{\,+})^\infty\big)_{\hat{q}_{m+1}}<\big((\overline{\tau_1\dots \tau_{2^m}}^{\,+})^\infty\big)_{\hat{q}_m}=\frac{1}{\hat{q}_m-1}-1=\frac{2-\hat{q}_m}{\hat{q}_m-1}.
\label{eq:tail-estimate}
\end{equation}
For $m=2$, we have by \eqref{eq:q-hat-estimates},
\begin{align*}
\big(\sc_{2,1}^{(2)}\big)_{\hat{q}_2}-\big(\sc_{2,2}^{(2)}\big)_{\hat{q}_3}
&>\big(1(01)^2\big)_{\hat{q}_2}-\big(1(01)^2\big)_{\hat{q}_3}-\hat{q}_3^{-5}\big((0011)^\infty\big)_{\hat{q}_2}\\
&=\sum_{i=1}^3\left(\frac{1}{\hat{q}_2^{2i-1}}-\frac{1}{\hat{q}_3^{2i-1}}\right)-\hat{q}_3^{-5}\frac{2-\hat{q}_2}{\hat{q}_2-1}
\geq .054.
\end{align*}
Thus, (\ref{eq:may-31-2}) holds for $m=1, 2$.

For $m\geq 3$, we will prove (\ref{eq:may-31-2}) by considering the further estimation:
\begin{align*}
\big(\sc_{2^{m-1},1}^{(m)}\big)_{\hat{q}_m}-\big(\sc_{2^{m-1},2}^{(m)}\big)_{\hat{q}_{m+1}}
&\geq \big(1\,\overline{\tau_1\dots\tau_{2^{m-1}}}^{\,+}\big)_{\hat{q}_m}
-\big(1\,\overline{\tau_1\dots\tau_{2^{m-1}}}^{\,+}\big)_{\hat{q}_{m+1}}\\
&\qquad -\hat{q}_{m+1}^{-2^m-1}\big((\overline{\tau_1\dots \tau_{2^m}}^{\,+})^\infty\big)_{\hat{q}_{m+1}}.
\end{align*}
By \eqref{eq:identity1},
\[
\big(1\,\overline{\tau_1\dots\tau_{2^{m-1}}}^{\,+}\big)_{\hat{q}_m}=\frac{1}{\hat{q}_m}\left(1+\big(\overline{\tau_1\dots\tau_{2^{m-1}}}^{\,+}\big)_{\hat{q}_m}\right)=\frac{1}{\hat{q}_m(\hat{q}_m-1)}.
\]
Similarly, by \eqref{eq:identity3},
\begin{align*}
\big(1\,\overline{\tau_1\dots\tau_{2^{m-1}}}^{\,+}\big)_{\hat{q}_{m+1}}
&=\frac{1}{\hat{q}_{m+1}}\left[1+\frac{2-\hat{q}_{m+1}}{\hat{q}_{m+1}-1}\left(\frac{1}{1-\hat{q}_{m+1}^{-2^{m-1}}}-\hat{q}_{m+1}^{-2^{m-1}}\right)\right]\\
&=\frac{1}{\hat{q}_{m+1}(\hat{q}_{m+1}-1)}+\frac{2-\hat{q}_{m+1}}{\hat{q}_{m+1}(\hat{q}_{m+1}-1)}\left(\frac{1}{1-\hat{q}_{m+1}^{-2^{m-1}}}-1-\hat{q}_{m+1}^{-2^{m-1}}\right)\\
&=\frac{1}{\hat{q}_{m+1}(\hat{q}_{m+1}-1)}+\frac{2-\hat{q}_{m+1}}{\hat{q}_{m+1}(\hat{q}_{m+1}-1)}\cdot\frac{q_{m+1}^{-2^m}}{1-\hat{q}_{m+1}^{-2^{m-1}}}.
\end{align*}
Thus, using \eqref{eq:tail-estimate} and some algebra, we arrive at
\begin{align*}
\big(\sc_{2^{m-1},1}^{(m)}\big)_{\hat{q}_m}-\big(\sc_{2^{m-1},2}^{(m)}\big)_{\hat{q}_{m+1}}
&>\frac{1}{\hat{q}_m(\hat{q}_m-1)}-\frac{1}{\hat{q}_{m+1}(\hat{q}_{m+1}-1)}\\
&\qquad-\frac{2-\hat{q}_{m+1}}{\hat{q}_{m+1}(\hat{q}_{m+1}-1)}\cdot\frac{\hat{q}_{m+1}^{-2^m}}{1-\hat{q}_{m+1}^{-2^{m-1}}}-\hat{q}_{m+1}^{-2^m-1}\frac{2-\hat{q}_m}{\hat{q}_m-1}.
\end{align*}
Now observe that
\begin{align*}
\frac{1}{\hat{q}_m(\hat{q}_m-1)}-\frac{1}{\hat{q}_{m+1}(\hat{q}_{m+1}-1)}
&=(\hat{q}_{m+1}-\hat{q}_m)\frac{\hat{q}_m+\hat{q}_{m+1}-1}{\hat{q}_m\hat{q}_{m+1}(\hat{q}_m-1)(\hat{q}_{m+1}-1)}\\
&>(\hat{q}_{m+1}-\hat{q}_m)\frac{2\hat{q}_m-1}{\hat{q}_{m+1}^2(\hat{q}_{m+1}-1)^2}.
\end{align*}
Therefore, by Lemma \ref{lem:qm-difference} it follows that
\begin{align*}
\big(\sc_{2^{m-1},1}^{(m)}\big)_{\hat{q}_m}&-\big(\sc_{2^{m-1},2}^{(m)}\big)_{\hat{q}_{m+1}}\\
&>\hat{q}_{m+1}^{-2^m-1}\left[(.265)\frac{2\hat{q}_m-1}{\hat{q}_{m+1}(\hat{q}_{m+1}-1)^2}-\frac{2-\hat{q}_{m+1}}{\hat{q}_{m+1}-1}\cdot\frac{1}{1-\hat{q}_{m+1}^{-2^{m-1}}}-\frac{2-\hat{q}_m}{\hat{q}_m-1}\right].
\end{align*}
The expression in square brackets is positive for all $m\geq 3$, as (cf. \eqref{eq:q-hat-estimates})
\begin{gather*}
\frac{2\hat{q}_m-1}{\hat{q}_{m+1}(\hat{q}_{m+1}-1)^2}\geq \frac{2\hat{q}_3-1}{q_{KL}(q_{KL}-1)^2}\geq 2.3,\\
\frac{2-\hat{q}_{m+1}}{\hat{q}_{m+1}-1}\cdot\frac{1}{1-\hat{q}_{m+1}^{-2^{m-1}}}\leq \frac{2-\hat{q}_4}{\hat{q}_4-1}\cdot\frac{1}{1-\hat{q}_4^{-4}}\leq .3,
\end{gather*}
and
\[
\frac{2-\hat{q}_m}{\hat{q}_m-1}\leq\frac{2-\hat{q}_3}{\hat{q}_3-1}\leq .28.
\]
This proves (\ref{eq:may-31-2}), and completes the proof.
\end{proof}

We also need the following precise inequality.

\begin{lemma} \label{lem:c-inequality}
Let  $n\in S\cup\set{1}$ {and let $k$ be the integer such that $2^{k-1}<n\leq 2^k$}. Let $m>k$ and $j\in\N$. Then for any $q>\hat{q}_m$,
\[
\big(\sc_{n,j+2}^{(m)}\big)_q-\big(\sc_{n,j+1}^{(m)}\big)_q<q^{-2^{m-1}(j+3)+n-1}.
\]
\end{lemma}

\begin{proof}
Since $\overline{\tau_1\dots\tau_{2^m}}^{\,+}=\overline{\tau_1\dots\tau_{2^{m-1}}}\,\tau_1\dots\tau_{2^{m-1}}$ and $\alpha(\hat{q}_m)=(\tau_1\dots\tau_{2^m}^-)^\infty$, it follows that
{\begin{align*}
\big(\sc_{n,j+2}^{(m)}\big)_q-\big(\sc_{n,j+1}^{(m)}\big)_q &=(\tau_n\ldots \tau_{2^{m-1}}(\overline{\tau_1\ldots\tau_{2^{m-1}}}^+)^{j+2}\,\overline{\al(\hat q_m)})_q\\
&\qquad-(\tau_n\ldots \tau_{2^{m-1}}(\overline{\tau_1\ldots\tau_{2^{m-1}}}^+)^{j+1}\overline{\tau_1\ldots \tau_{2^{m-1}}}\,\al(\hat q_m))_q\\
&=q^{-2^{m-1}(j+3)+n-2}\left[\big(1\overline{\alpha(\hat{q}_m)}\big)_q-\big(0\alpha(\hat{q}_m)\big)_q\right].
\end{align*}}
Now
\begin{align*}
\big(1\overline{\alpha(\hat{q}_m)}\big)_q-\big(0\alpha(\hat{q}_m)\big)_q &= \frac{1}{q}\left(1+\sum_{i=1}^\infty \frac{1-\alpha_i(\hat{q}_m)}{q^i}-\sum_{i=1}^\infty \frac{\alpha_i(\hat{q}_m)}{q^i}\right)\\
&=\frac{1}{q}\left(1+\frac{1}{q-1}-2\sum_{i=1}^\infty \frac{\alpha_i(\hat{q}_m)}{q^i}\right).
\end{align*}
Thus, it suffices to show that
\[
2\sum_{i=1}^\infty \frac{\alpha_i(\hat{q}_m)}{q^i}>\frac{1}{q-1}.
\]
If $m\geq 2$, then by (\ref{eq:qn}) the sequence $\alpha(\hat{q}_m)$ begins with $11$, while if $m=1$, then $\alpha(\hat{q}_m)=\alpha(\hat{q}_1)=(10)^\infty$. In either case,
\[
2\sum_{i=1}^\infty \frac{\alpha_i(\hat{q}_m)}{q^i}>2\left(\frac{1}{q}+\frac{1}{q^3}\right)>\frac{1}{q-1},
\]
since the last inequality can be verified to hold for all $q>1.55$. This completes the proof.
\end{proof}

\begin{proof}[Proof of Theorem \ref{thm:Komornik-Loreti-cascades}]
Fix $n\in S\cup\set{1}$ and $k\in\N$ such that $2^{k-1}<n\leq 2^k$. For simplicity we write $\sc_j^{(m)}:=\sc_{n,j}^{(m)}$ and $\sc_\infty^{(m)}:=\sc_{n,\infty}^{(m)}$.
We first prove (i). Let $\xi_{m,j-1}\leq x<\xi_{m,j}$, where $m>k$ and $j\geq 2$. Note that $n\leq 2^k\leq 2^{m-1}$. Since $\big(\sc_j^{(m)}\big)_{\hat{q}_m}=\xi_{m,j}$, and $\big(\sc_j^{(m)}\big)_{\hat{q}_{m+1}}<\big(\sc_{j-1}^{(m)}\big)_{\hat{q}_m}=\xi_{m,j-1}$ by Proposition \ref{prop:Derong-overlap-general}, there is a base $q\in(\hat{q}_m,\hat{q}_{m+1})$ such that $\big(\sc_j^{(m)}\big)_q=x$. Observe that $\sc_j^{(m)}\in\us_q$. This implies $q_s(x)\le q$.

Next we prove $q_s(x)=q$. Suppose, by way of contradiction, that there is a base $p<q$ such that $x\in\u_p$. Then $\Phi_x(p)\prec\Phi_x(q)=\sc_j^{(m)}$, and hence $\Phi_x(p)\preceq \sc_{j-1}^{(m)}$. If $p>\hat{q}_m$, this implies
\[
x=\big(\Phi_x(p)\big)_p\leq \big(\sc_{j-1}^{(m)}\big)_p<\big(\sc_{j-1}^{(m)}\big)_{\hat{q}_m}=\xi_{m,j-1}\leq x,
\]
a contradiction. So assume $p\leq \hat{q}_m$. Since $\Phi_x(q)=\sc_j^{(m)}$ begins with $\tau_n\dots\tau_{2^k}\tau_{2^k+1}\dots\tau_{2^{k+1}}$, it follows just as in the proof of Proposition \ref{prop:power-of-two} that $\Phi_x(p)$ must begin with $\tau_n\dots\tau_{2^k}$. We now consider two cases.

{\em Case 1.} $\Phi_x(p)$ begins with $\tau_n\dots\tau_{2^m}$. Since $p\leq\hat{q}_m$, using $\Phi_x(p)\in\us_p\subseteq\us_{\hat q_m}$ it follows that
\[
\Phi_x(p)\succeq \tau_n\dots\tau_{2^m}(\overline{\tau_1\dots\tau_{2^{m-1}}}^{\,+})^\infty=\tau_{n}\dots\tau_{2^{m-1}}(\overline{\tau_1\dots\tau_{2^{m-1}}}^{\,+})^\infty,
\]
so that
\begin{align*}
x=\big(\Phi_x(p)\big)_p&\geq \big(\tau_{n}\dots\tau_{2^{m-1}}(\overline{\tau_1\dots\tau_{2^{m-1}}}^{\,+})^\infty\big)_p\\
&\geq \big(\tau_{n}\dots\tau_{2^{m-1}}\big(\overline{\tau_1\dots\tau_{2^{m-1}}}^{\,+}\big)^\infty\big)_{\hat{q}_m}=\xi_m.
\end{align*}
Hence $x\geq \xi_m>\xi_{m,j}>x$, a contradiction.

{\em Case 2.} There is a smallest integer $\nu$ with $k<\nu\leq m$ such that $\Phi_x(p)\prec \tau_n\dots\tau_{2^\nu}$. In this case, using $\Phi_x(p)\in\us_p\subset\us_{q_{KL}}$ it follows that $\Phi_x(p)$ must begin with $\tau_n\dots\tau_{2^\nu}^-$, and we have, by definition of quasi-greedy expansion,
\[
\sum_{i=n}^{2^\nu}\frac{\tau_i}{p^{i-n+1}}\geq x.
\]
From this, we deduce
\[
\sum_{i=1}^{2^\nu}\frac{\tau_i}{p^i}\geq \sum_{i=1}^{n-1}\frac{\tau_i}{p^i}+\frac{x}{p^{n-1}}>{\sum_{i=1}^{n-1}\frac{\tau_i}{q_{KL}^i}+\frac{x_n}{q_{KL}^{n-1}}=1},
\]
where the second inequality follows since $p<q<q_{KL}$ and $x>x_n$, and {the last equality}  uses the definition of $x_n$. We conclude that $p<\hat{q}_\nu$. But this is impossible, since $\Phi_x(p)$ contains the word $1\tau_{2^{\nu-1}+1}\dots\tau_{2^{\nu}}^-=1\overline{\tau_1\dots\tau_{2^{\nu-1}}}$, which is forbidden in $\us_p$ for $p<\hat{q}_\nu$.

Since all cases above lead to a contradiction, it follows that $q_{s}(x)=q=\min\ub(x)$. This proves (\ref{eq:minimal-expansion1}).

Next, we determine $q_s(x)$ for  $\xi_{m+1}\leq x<\xi_{m,1}$. Let $x\in[\xi_{m+1}, \xi_{m,1})$. Since $\sc_\f^{(m+1)}=\sc_1^{(m)}$, we have $\big(\sc_\f^{(m+1)}\big)_{\hat{q}_m}=\big(\sc_1^{(m)}\big)_{\hat{q}_m}=\xi_{m,1}$ and $\big(\sc_\f^{(m+1)}\big)_{\hat{q}_{m+1}}=\xi_{m+1}$. Thus, there is a unique $q\in(\hat{q}_m,\hat{q}_{m+1}]$ such that $\big(\sc_\f^{(m+1)}\big)_q=x$. Since $\sc_\f^{(m+1)}\in\us_q$, this proves $q_s(x)\le q$.

Suppose there is a base $p<q$ such that $x\in\u_p$. If $p\leq\hat{q}_m$, the same argument as above gives a contradiction. So assume $p>\hat{q}_m$. We claim that $\Phi_x(p)$ must begin with $\tau_n\dots\tau_{2^m}$. To see this, recall the functions $f_n$ from \eqref{eq:f_n}. By Lemma \ref{lem:polynomial-inequality}, $x>x_n=f_{n-1}(q_{KL})>f_{n-1}(\hat{q}_m)$. Since $n\leq 2^m$, $\alpha(\hat{q}_m)$ begins with $\tau_1\dots\tau_{n-1}$, and so
\begin{align*}
f_{n-1}(\hat{q}_m)&=\hat{q}_m^{n-1}-\sum_{i=1}^{n-1}\tau_i\hat{q}_m^{n-i-1}=\hat{q}_m^{n-1}-\sum_{i=1}^{n-1}\alpha_i(\hat{q}_m)\hat{q}_m^{n-i-1}\\
&=\hat{q}_m^{n-1}\left(1-\sum_{i=1}^{n-1}\frac{\alpha_i(\hat{q}_m)}{\hat{q}_m^{i}}\right)=\big(\alpha_n(\hat{q}_m)\alpha_{n+1}(\hat{q}_m)\dots\big)_{\hat{q}_m}.
\end{align*}
It follows that
\[
\Phi_x(p)\succ\Phi_{f_{n-1}(\hat{q}_m)}(p)\succ\Phi_{f_{n-1}(\hat{q}_m)}(\hat{q}_m)=\alpha_n(\hat{q}_m)\alpha_{n+1}(\hat{q}_m)\dots,
\]
and hence $\Phi_x(p)$ begins with $\tau_n\dots\tau_{2^m}$. But, since $p<\hat{q}_{m+1}$ and $x\in\u_p$, $\Phi_x(p)\in\us_{\hat{q}_{m+1}}$ and so
\[
\Phi_x(p)\succeq \tau_n\dots\tau_{2^m}(\overline{\tau_1\dots\tau_{2^m}}^{\,+})^\f=\sc_\infty^{(m+1)}=\Phi_x(q),
\]
contradicting that $p<q$. Therefore, $q_{s}(x)=q=\min\ub(x)$. This proves (\ref{eq:x-in-special-interval}), and thus establishes (i).

We next prove (ii).
Note that for any sequence $(d_i)\in\{0,1\}^\N$ except $0^\infty$, the function $q\mapsto \sum_{i=1}^\infty d_i q^{-i}$ is strictly decreasing and convex on $(1,\infty)$. Hence, \eqref{eq:x-in-special-interval} and \eqref{eq:minimal-expansion1}  give $x$ as a strictly decreasing convex function on the domain $(q_{s}(\xi_{m,1}^-),q_{s}(\xi_{m+1})]$, resp. $(q_{s}(\xi_{m,j}^-),q_{s}(\xi_{m,j-1})]$,  where $q_{s}(z^-):=\lim_{x\nearrow z}q_{s}(x)$. Since the inverse of a strictly decreasing convex function is again strictly decreasing and convex, this proves the first statement in (ii).


For the second statement of (ii) let $q_{m,j}:=q_{s}(\xi_{m,j})$. From (i), we have
\begin{equation}
\big(\sc_{j+1}^{(m)}\big)_{q_{m,j}}=\xi_{m,j}=\big(\sc_j^{(m)}\big)_{\hat{q}_m} \qquad\mbox{and}\qquad
\big(\sc_{j+2}^{(m)}\big)_{q_{m,j+1}}=\xi_{m,j+1}=\big(\sc_{j+1}^{(m)}\big)_{\hat{q}_m}.
\label{eq:c-identities}
\end{equation}
To show $q_{m,j}>q_{m,j+1}$ is equivalent to showing that
\[
\big(\sc_{j+2}^{(m)}\big)_{q_{m,j}}<\big(\sc_{j+2}^{(m)}\big)_{q_{m,j+1}}.
\]
By \eqref{eq:c-identities}, this will follow if we can show that
\begin{equation}
\big(\sc_{j+2}^{(m)}\big)_q-\big(\sc_{j+1}^{(m)}\big)_q<\big(\sc_{j+1}^{(m)}\big)_{\hat{q}_m}-\big(\sc_{j}^{(m)}\big)_{\hat{q}_m} \qquad\forall q>\hat{q}_m.
\label{eq:c-difference-inequality}
\end{equation}
Just take $q=q_{m,j}$. To see \eqref{eq:c-difference-inequality}, observe that, since $\alpha(\hat{q}_m)=(\tau_1\ldots\tau_{2^{m-1}}\overline{\tau_1\ldots\tau_{2^{m-1}}})^\f$, $\big(\sc_{k}^{(m)}\big)_{\hat{q}_m}$ can be written alternatively as
\[
\big(\sc_{j}^{(m)}\big)_{\hat{q}_m}=\big(\tau_n\dots\tau_{2^{m-1}}(\overline{\tau_1\dots\tau_{2^{m-1}}}^{\,+})^{j+1}\big)_{\hat{q}_m},
\]
and similarly for $\big(\sc_{j+1}^{(m)}\big)_{\hat{q}_m}$, whence
\[
\big(\sc_{j+1}^{(m)}\big)_{\hat{q}_m}-\big(\sc_{j}^{(m)}\big)_{\hat{q}_m}=\hat{q}_m^{-2^{m-1}(j+2)+n-1}\big(\overline{\tau_1\dots\tau_{2^{m-1}}}^{\,+}\big)_{\hat{q}_m}\geq \hat{q}_m^{-2^{m-1}(j+3)+n-1}.
\]
An appeal to Lemma \ref{lem:c-inequality} completes the proof of \eqref{eq:c-difference-inequality}.

Finally we prove the last statement of (ii). Note that $\xi_{m,j}=\big(\sc_j^{(m)}\big)_{\hat{q}_m}$. Then by (\ref{eq:minimal-expansion1}) it follows that
\[
\lim_{x\nearrow \xi_{m,j}}q_s(x)=\hat q_m<q_s(\xi_{m,j}).
\]
It remains to show that
$
\lim_{j\to\infty} q_{s}(\xi_{m,j})=\hat{q}_m.
$
Recall that  $q_{m,j}=q_{s}(\xi_{m,j})$. Then
\begin{equation}\label{eq:mar-9-1}
\Phi_{\xi_{m,j}}(q_{m,j})=\sc_{j+1}^{(m)}\to \tau_n\dots\tau_{2^{m-1}}(\overline{\tau_1\dots\tau_{2^{m-1}}}^{\,+})^\infty=\Phi_{\xi_m}(\hat{q}_m)
\end{equation}
as $j\to\infty$. Suppose, by way of contradiction, that there is an $\ep>0$ such that $q_{m,j}>\hat{q}_m+\ep$ for all $j$. Then by Lemma \ref{lem:char-quasi-expansion} (i) we also have $\Phi_{\xi_{m,j}}(q_{m,j})\succ\Phi_{\xi_{m,j}}(\hat{q}_m+\ep)$ for all $j$. But, since by Lemma \ref{lem:char-quasi-expansion} (ii) $\Phi_x(q)$ is left-continuous in $x$ for fixed $q$, we have $\Phi_{\xi_{m,j}}(\hat{q}_m+\ep)\to \Phi_{\xi_m}(\hat{q}_m+\ep)$ as $j\to\infty$. Hence, by (\ref{eq:mar-9-1}),
\[
\Phi_{\xi_m}(\hat{q}_m)=\lim_{j\to\infty} \Phi_{\xi_{m,j}}(q_{m,j})\succeq \lim_{j\to\infty}\Phi_{\xi_{m,j}}(\hat{q}_m+\ep)=\Phi_{\xi_m}(\hat{q}_m+\ep),
\]
which is impossible. Therefore, $\lim_{j\to\infty} q_{m,j}=\hat{q}_m$, completing the proof.
\end{proof}

\section{The maximum of $q_s(x)$} \label{sec:maximum}

Here we show that the maximum of $q_s(x)$ is uniquely attained, and calculate its value exactly.

\begin{theorem} \label{thm:maximum}
The maximum value of $q_s(x)$ is uniquely attained at $x=1/q_G=(\sqrt{5}-1)/2$, and
\[
q_{\max}=q_s(1/q_G)=r\approx 1.888453328,
\]
where $r$ is the unique base such that
\[
\big(1(0010001100011)^\infty\big)_r=\big(100(1000110001100)^\f\big)_r=\frac{1}{q_G}.
\]
The number $r$ is an algebraic integer of degree 26: it is a zero of the irreducible polynomial
\[
t^{26}-t^{25}-t^{24}-t^{22}-2t^{21}-2t^{18}-3t^{17}-2t^{16}-2t^{14}-5t^{13}-t^{12}-t^{10}-3t^9-t^8-t^5-t^4+1.
\]
\end{theorem}

\begin{proof}
The proof consists of two steps.

\medskip
{\em Step 1.} We first show that $q_s(1/q_G)=r$ by using our algorithm from Section \ref{sec:algorithm}. Note that $q_1$ is the base such that $(01^\f)_{q_1}=x=1/q_G$. Thus $q_1$ is the root in $(1,2)$ of
\[
\frac{1}{q(q-1)}=\frac{1}{q_G},
\]
or equivalently, $q(q-1)=(\sqrt{5}+1)/2$. Hence $q_1$ satisfies the quartic equation $q^4-2q^3+q-1=0$, and this implies $\alpha(q_1)=(111000)^\f$, because solving $\big((111000)^\f\big)_q=1$ leads to the same quartic equation. Thus, in the algorithm, $n_1=4$ and $B_1=100$. Next, $q_2$ is the base such that $(100\,01^\f)_{q_2}=1/q_G$, so $q_2$ is the root in $(1,2)$ of
\[
\frac{1}{q}+\frac{1}{q^4(q-1)}=\frac{1}{q_G}=\frac{\sqrt{5}-1}{2}.
\]
This gives $q_2\approx 1.888452074$, and a computer calculation shows
\[
\al(q_2)=111001110011100000101010\dots.
\]
We see that $n_2=14, B_2=1000110001100$ and $m_2=17$. Thus, according to the algorithm, $r_2$ is the unique base such that
\[
\left(100\big(1000110001100\big)^\f\right)_{r_2}=(B_1B_2^\f)_{r_2}=\frac{1}{q_G}.
\]
It follows that $r_2$ is the number $r$ defined in the theorem, and again with help from the computer,
\[
\al(r_2)=1110011100111000011000\dots.
\]
(These quasi-greedy expansions can be verified rigorously up to the 17th digit, which is all that matters, but doing so is rather cumbersome.) Hence
$$\al_1(q_2)\dots\al_{m_2}(q_2)=\al_1(r_2)\dots\al_{m_2}(r_2),$$
so $q_s(1/q_G)=r_2=r$.

\bigskip
{\em Step 2.} We now show that $q_s(x)<r$ for all $x>0$ with $x\neq 1/q_G$. This is a bit tedious; we have to divide $(0,\infty)$ into infinitely many intervals, and on each interval find a unique expansion of $x$ in some base smaller than $r$. Unfortunately, we have not found a simpler way to prove this.

First, for $x\geq 1$, we know from Theorem \ref{thm:Komornik-Loreti-cascades} that $q_s\leq q_{KL}$. Moving to $x<1$, we first consider the expansion $(10)^\f$. This is a unique $q$-expansion for all $q>q_G$, and since $((10)^\f)_r\approx .7359$, there is for each $x\in(.7359,1)$ a base $q<r$ such that $x\in\u_q$. Hence the sequence $(10)^\f$ ``covers" the interval $(.7359,1)$. In a similar way, we can check that the sequence $1(0011)^\f$ covers the interval $\big((1(0011)^\f)_r,(1(0011)^\f)_{\hat{q}_2}\big)\approx(.6601,.7549)$, which overlaps the previous one. As we get closer to $1/q_G$, we have to use ever more complicated expansions, taking care that each interval overlaps the one before it. 

We summarize the expansions used for $x>1/q_G$ in the table below. For each interval in the last column, the left endpoint is the corresponding expansion evaluated in base $r$, and the right endpoint is the same expansion evaluated in the minimal base $q$ given in the third column. The overlaps are generous enough to accomodate for round-off errors in the computation of the endpoints.

\bigskip

\begin{tabular}{c|c|c|c}
Expansion & Minimal $\al(q)$ & Minimal $q$ & Interval covered\\ \hline
$(10)^\f$ & $(10)^\f$ & $q_G\approx 1.6180$ & $(.7359,1)$\\
$1(0011)^\f$ & $(1100)^\f$ & $\hat{q}_2\approx 1.7549$ & $(.6601,.7549)$\\
$1(00101)^\f$ & $(11010)^\f$ & $1.8124$ & $(.6346,.6792)$\\
$(100)^\f$ & $(110)^\f$ & $1.8393$ & $(.6219,.6478)$\\
$1001000(110)^\f$ & $111(001)^\f$ & $1.87135$ & $(.61927,.62796)$\\
$1001000(1100100)^\f$ & $111(0011011)^\f$ & $1.88596$ & $(.61822,.61945)$\\
$100100011000(11010)^\f$ & $11100111(00101)^\f$ & $1.88799$ & $(.618068,.618296)$\\
$100100011000(110010)^\f$ & $11100111(001101)^\f$ & $1.888354$ & $(.6180408,.6180895)$\\
$100100011000(1100100)^\f$ & $11100111(0011011)^\f$ & $1.888413$ & $(.6180363,.6180562)$\\
$1(0010001100011)^\f$ & $(1110011100110)^\f$ & $1.888444$ & $(1/q_G,.6180386)$
\end{tabular}

\bigskip

For $0<x<1/q_G$ the situation is simpler. The sequence $0(10)^\f$ covers the interval $(.3897,1/q_G)$; next the sequence $01(0011)^\f$ covers the interval $(.3495,.4301)$, so together these two sequences cover the interval $[1/q_G^2,1/q_G)$, as $1/q_G^2\approx .381966$. Then, for each $n\geq 2$, the sequence $0^n(10)^\f$ covers the interval $[1/q_G^{n+1},1/q_G^n)$ because
\[
\big(0^n(10)^\f\big)_r=\frac{1}{r^{n-1}(r^2-1)}<\frac{1}{q_G^{n+1}}, \qquad n\geq 2.
\]
Hence $q_s(x)<r$ for all $x>0$ with $x\neq 1/q_G$.
\end{proof}


\section{Open Problems} \label{sec:open-problems}

We end the paper with some natural questions that we have been unable to answer. In Section \ref{sec:classification} we classified points in $(0,1)$ into three sets $X_I, X_{II}$ and $X_{III}$ according to our algorithm. We proved that $X_I$ is Lebesgue null and $X_{II}$ has positive Lebesgue measure.

\begin{question}
Does $X_{II}$ have full Lebesgue measure? Equivalently, is $X_{III}$ Lebesgue null? If so, what is its Hausdorff dimension?
\end{question}

A large part of this paper focuses on the level sets $L(q)$ of $q_s$. In Theorem \ref{thm:infinite-level-set-qKL-intro} we showed that $L(q_{KL})$ is infinite, and it contains infinitely many one-sided accumulation points.

\begin{question}
Does $L(q_{KL})$ have any {\em two}-sided accumulation points? If the answer is no, then it follows that $L(q_{KL})$ is countable. More generally, are all level sets of $q_s$ countable?
\end{question}

We do not know much about the topology of the level sets of $q_s$.

\begin{question}
Is $L(q_{KL})$ closed? If so, is $L(q_{KL})$ the closure of the set of points $x_n,x_n'$ and $x_n''$ from Theorem \ref{thm:set-with-qKL-as-min}? More generally, are all level sets of $q_s$ closed?
\end{question}

We have shown in Theorem \ref{thm:finite-level-sets-intro} that $L(q)$ is finite for all $q\in(1,2]\backslash\overline{\ub}$. We have also shown in Theorems \ref{thm:infinite-level-set-qKL-intro} and \ref{thm:other-infinite-level-sets-intro} that $L(q)$ is infinite when $q=q_{KL}$ or $q$ is a de Vries-Komornik number below $1.87064$. This still leaves a gap. Specifically,

\begin{question}
Is $L(q)$ infinite also for de Vries-Komornik numbers between $1.87064$ and $q_{\max}\approx1.88845$?
More strongly, is $L(q)$ infinite for {\em all} $q\in\overline{\ub}\cap(1,q_{\max}]$?
\end{question}

\begin{question}
Do the level sets $L(\hat{q}_c(\sa))$ with $\hat{q}_c(\sa)$ a de Vries-Komornik number below $Q_3$ have infinitely many accumulation points? Observe that the argument from the proof of Proposition \ref{prop:accumulation-points} does not work for these values of $q$.
\end{question}

In Theorem \ref{thm:Komornik-Loreti-cascades} we described the Komornik-Loreti cascades occurring at infinitely many scales in the graph of $q_s$.

\begin{question}
Does the behavior described in Theorem \ref{thm:Komornik-Loreti-cascades} also appear at other levels in the graph? In particular, does the graph of $q_s$ contain ``de Vries-Komornik cascades", which cascade down from the points $(x_n(\sa),\hat{q}_c(\sa)$ of Theorem \ref{thm:deVries-Komornik}? We suspect that the answer is yes, but proving it would likely involve considerable technicalities.
\end{question}

{
We saw in Example \ref{ex:1} that $q_s(x)$ is algebraic over $\mathbb{Q}(x)$ for all $x\in X_{II}$. Clearly this is also the case for $x\in X_I$, and we observed in Remark \ref{rem:type-III-algebraic} that it is also the case for the points $x_n$ with $n\in S$, which are of type III.

\begin{question}
Is $q_s(x)$ algebraic over $\mathbb{Q}(x)$ for every $x>0$?
\end{question}
}

Finally, in this paper we considered the smallest univoque base $q_s(x)$ of a given $x>0$ with alphabet $\set{0,1}$. It would be of interest to extend this research to more general alphabets $\{0,1,\dots,M\}$.

\section*{Acknowledgements}
The authors wish to thank an anonymous referee for a very careful reading of the manuscript, which helped improve the presentation of the paper. D.~Kong was supported by NSFC No.~11971079 and the Fundamental
and Frontier Research Project of Chongqing No.~cstc2019jcyj-msxmX0338 and No.~cx2019067.

%
%

\end{document}